\documentclass[microtype]{gtpart}

\usepackage{pinlabel} 
\usepackage{graphicx}
\usepackage[all]{xy}
\usepackage{amscd}
\usepackage{tikz}
\usepackage{xcolor}
\usepackage{float}
\usepackage{caption}
\usepackage{bm}
\usepackage{blkarray}
\usepackage[margin=3.7cm]{geometry}

\title{Sheaves of maximal intersection \\ and multiplicities of stable log maps}

\author[J. Choi]{Jinwon Choi}
\address{Department of Mathematics \& Research Institute of Natural Sciences, Sookmyung Women's University, Cheongpa-ro 47-gil 100, Youngsan-gu, Seoul 04310, Republic of Korea}
\email{jwchoi@sookmyung.ac.kr}

\author[M. van Garrel]{Michel van Garrel}
\address{School of Mathematics, University of Birmingham, B15 2TT, Birmingham, United Kingdom}
\email{m.vangarrel@bham.ac.uk}

\author[S. Katz]{Sheldon Katz}
\address{Department of Mathematics, MC-382, University of Illinois at Urbana-Champaign, Urbana, IL 61801, USA}
\email{katz@math.uiuc.edu}

\author[N. Takahashi]{Nobuyoshi Takahashi}
\address{Department of Mathematics, Graduate School of Science, Hiroshima University, 1-3-1 Kagamiyama, Higashi-Hiroshima, 739-8526 JAPAN}
\email{tkhsnbys@hiroshima-u.ac.jp}

\usepackage{multicol}
\usepackage{hyperref}

\arxivreference{}

\newcommand{\bA}{{\mathbb{A}}}

\newcommand{\bP}{{\mathbb{P}}}
\newcommand{\boldm}{{\boldsymbol{m}}}
\newcommand{\bv}{{\boldsymbol{v}}}
\newcommand{\bw}{{\boldsymbol{w}}}
\newcommand{\bz}{{\boldsymbol{z}}}
\newcommand{\bmu}{{\boldsymbol{\mu}}}
\newcommand{\bM}{{\boldsymbol{M}}}
\newcommand{\CC}{{\mathbb{C}}}

\newcommand{\NN}{{\mathbb{N}}}
\newcommand{\PP}{{\mathbb{P}}}

\newcommand{\ZZ}{{\mathbb{Z}}}
\newcommand{\cC}{{\mathcal{C}}}

\newcommand{\cE}{{\mathcal{E}}}

\newcommand{\cM}{{\mathcal{M}}}

\newcommand{\cU}{{\mathcal{U}}}
\newcommand{\cV}{{\mathcal{V}}}

\newcommand{\barM}{{\overline{\mathrm{M}}}}
\newcommand{\cO}{{\mathcal{O}}}

\newcommand {\Spec} {\operatorname{Spec}}

\newcommand {\shM}  {\mathcal{M}}
\newcommand {\shN}  {\mathcal{N}}
\newcommand {\shO}  {\mathcal{O}}
\newcommand {\cOX}{{\mathcal{O}_X}}

\newcommand{\tx}{{\tilde{x}}}
\newcommand{\mc}{\mathcal}
\newcommand{\sho}{|\shO_X(\beta,P)|}
\newcommand{\shoo}{|\shO_X(\beta,P)|^\circ}
\newcommand{\shooo}{|\shO_X(\beta,P)|^{\circ\circ}}

\newcommand{\nxd}{\shN_\beta(X, D)}
\newcommand{\nxdp}{\shN^P_\beta(X, D)}

\DeclareMathOperator {\hhh}{H}
\DeclareMathOperator {\HHH}{h}
\DeclareMathOperator {\Pic}{Pic}

\newtheorem{Theorem}{Theorem}[section]
\newtheorem{Lemma}[Theorem]{Lemma}
\newtheorem{Proposition}[Theorem]{Proposition}
\newtheorem{Corollary}[Theorem]{Corollary}

\newtheorem{Conjecture}[Theorem]{Conjecture}

\theoremstyle{definition}
\newtheorem{Definition}[Theorem]{Definition}
\newtheorem{Example}[Theorem]{Example}

\newtheorem{Assumptions}[Theorem]{Assumptions}
\newtheorem{Notation}[Theorem]{Notation}

\theoremstyle{remark}
\newtheorem{Remark}[Theorem]{Remark}

\numberwithin{equation}{section}

\begin{document}

\begin{abstract}
A great number of theoretical results are known about log Gromov-Witten invariants \cite{GS13, Chen14, AbramChen14}, but few calculations are worked out. In this paper we restrict to surfaces and to genus 0 stable log maps of maximal tangency. We ask how various natural components of the moduli space contribute to the log Gromov-Witten invariants. The first such calculation \cite[Proposition 6.1]{GPS10} by Gross-Pandharipande-Siebert deals with multiple covers over rigid curves in the log Calabi-Yau setting. As a natural continuation, in this paper we compute the contributions of non-rigid irreducible curves in the log Calabi-Yau setting and that of the union of two rigid curves in general position. For the former, we construct and study a moduli space of ``logarithmic'' $1$-dimensional sheaves and compare the resulting multiplicity with tropical multiplicity. For the latter, we explicitly describe the components of the moduli space and work out the logarithmic deformation theory in full, which we then compare with the deformation theory of the analogous relative stable maps.
\end{abstract}

\keywords{Log Gromov-Witten theory, moduli spaces of sheaves, log Calabi-Yau surfaces}

\subject{primary}{msc2020}{14N35, 14A21, 14B10, 14D15, 14D20, 14J26}
\subject{secondary}{msc2020}{}

\maketitle
\setcounter{tocdepth}{2}
\tableofcontents

\section{Introduction}

Let $X$ be a smooth surface, let $D$ be an effective divisor on $X$ and denote by $D_\mathrm{sm}$ the smooth part of $D$.
An \emph{$\bA^1$-curve} on $(X, D)$ is a proper irreducible curve $C$ on $X$ 
such that the normalization of $C\setminus D$ is isomorphic to $\bA^1$. 
We calculate the contributions of the following curves to the genus $0$ 
log Gromov-Witten invariants of maximal tangency: 
\begin{enumerate}
\item[(A)] Corollary \ref{cor:ecompjac}:
Under the assumption that $X$ is a projective rational surface, $K_X+D\sim 0$ and 
$P\in D_\mathrm{sm}$, 
an $\bA^1$-curve $C$ which is smooth at $P=C\cap D$.
\item[(B)] Theorem \ref{theorem_main}:
The sum of two immersed $\bA^1$-curves $Z_1, Z_2$ 
with $(K_X+D).Z_i=0$, 
intersecting at $P\in D_\mathrm{sm}$ in a general way. 
Note that $Z_1$ and $Z_2$ are rigid as $\bA^1$-curves, see the proof of \cite[Proposition 4.21(2)]{CGKT2}. 
\end{enumerate}
In the proof of (A), we study a moduli space of ``logarithmic'' $1$-dimensional sheaves, 
which shows an intriguing analogy with the case of K3 surfaces as described in Section \ref{sec:K3}. 
For (B), we give a rather concrete description of a space of stable log maps and its deformation theory. Section \ref{sec:ex} illustrates (A) and (B) and describes some future directions. Section \ref{sec:exx} describes (B) in detail through an example and compares it with the case of relative stable maps.

There are two natural and well-studied geometries to which our results apply:
\begin{enumerate}
    \item The setting of the tropical vertex \cite{GPS10} consisting of appropriate blow ups of toric surfaces at smooth points of the toric boundary, as summarized in Section \ref{ex:GPS}. 
    \item Log K3 surfaces $(X,D)$ for $X$ a del Pezzo surface and $D\in|-K_X|$ smooth. 
    Despite recent breakthroughs \cite{CGKT1,CGKT2,Bou19a,Bou19b,Gra20,BN20}, many aspects of their enumerative geometry remain mysterious. Section \ref{subsection:applications} describes one such open problem.
\end{enumerate}
For both (1) and (2), Corollary \ref{cor:ecompjac} and Theorem \ref{theorem_main} calculate the contributions to the invariants of typical zero-dimensional components of the moduli spaces.

Section \ref{sec:ex:sec} contains a fully worked out example that illustrates (A) and (B) and includes some new computations.

\subsection{Idealized geometries}

Gromov-Witten invariants were devised as a virtual count of curves in projective or compact symplectic manifolds.  However, 
their relationship with actual counts of curves, even when understood, is often quite subtle. 

Let $Y$ be a Calabi-Yau (CY) threefold and consider its genus 0 Gromov-Witten (GW) invariants $N_\beta(Y)$ for $\beta\in\hhh_2(Y,\ZZ)$. They are rational numbers in general, and BPS numbers $n_\beta(Y)$ were proposed as underlying $\ZZ$-valued invariants. They were originally defined via the recursive relationship
\begin{equation}\label{cy3-bps}
N_\beta(Y) = \sum_{k|\beta} \; \frac{1}{k^3} \; n_{\beta/k}(Y).
\end{equation}
Still, even in the case of a compact $Y$, typically the BPS numbers $n_\beta(Y)$ are enumerative only in low degrees, in the sense that they agree with the count of rational curves. For example, if $Y$ is a general quintic threefold, then $n_d(Y)$ equals the number of rational curves of degree $d$ only when $d\leq 9$. For larger degrees, the story is more subtle.

Equation \eqref{cy3-bps} is derived by postulating that $Y$ symplectically deforms to an \emph{idealized geometry} $\widetilde{Y}$  where all rational curves are infinitesimally rigid, i.e.\ have normal bundle $\shO(-1)\oplus\shO(-1)$. If such a curve is in class $\beta$, then its contribution to $N_{l\beta}(\widetilde{Y})$ for $l\in\NN$ is given by $1/l^3$ \cite{AM93,Kon95,Ma95,Voi96}, leading to the above formula. By deformation invariance one would then conclude that $n_\beta(Y)$ is a count of rational curves in $\widetilde{Y}$. While the existence of such an idealized geometry is unknown, it is remarkable that the so defined $n_\beta(Y)$ are integers \cite{IP18}.

Assume instead that $Y$ is the local Calabi-Yau threefold given as the total space $\mathrm{Tot} \, \shO_X(K_X)$ of the canonical bundle over a del Pezzo surface $X$. Then the enumerative interpretation of the local BPS numbers $n_\beta(K_X):=n_\beta(Y)$ is even more mysterious. Not only is their relationship to counts of rational curves in $Y$ previously not known, they also are alternating negative with interesting divisibility properties \cite[Conjecture 1.2]{CGKT1}. As an illustration, the BPS numbers for local $\PP^2$ in degrees $d$ up to 6 are $3, \, -6, \, 27, \, -192, \, 1695, \, -17064$, all of which are divisible by $3d$, a conjecture which was proven in \cite{Bou19b} based on \cite{Gra20,Bou19a}.

An interpretation of $n_\beta(K_X)$, which also makes it clear why they are integral, was given using moduli spaces of sheaves. 
Denote by $M_{\beta, 1}(X)$ the (smooth) moduli space of $(-K_X)$-stable $1$-dimensional sheaves of class $\beta$ and of holomorphic Euler characteristic $1$ on $X$, and let $w:=-K_X\cdot\beta$. By \cite{Katz08,Bri06,Toda09,Toda10,Toda12}, the genus 0 local BPS invariants can be identified with the topological Euler characteristics of $M_{\beta,1}(X)$:
\[
n_\beta (K_X) = (-1)^{w-1}e(M_{\beta,1}(X)).
\]

Another interpretation comes from log geometry. 
Based on the predictions of \cite{Tak01},  in \cite{CGKT1,CGKT2} we started a program to show that $(-1)^{w-1}n_\beta(K_X)/w$ is a count of log curves in the surface $X$, namely that it equals the log BPS invariants of Definition \ref{defn:logbps}. Let $D$ be a smooth anticanonical curve on $X$. 
Denote by $\overline{\text{M}}_\beta(X,D)$  the moduli space of genus 0 basic stable log maps \cite{GS13, Chen14, AbramChen14} in $X$ of class $\beta$ and of maximal tangency with $D$, see Section \ref{sec:loggw}. 
From this space one defines the log Gromov-Witten invariants $\mathcal{N}_\beta(X,D)$, which virtually count $\bA^1$-curves, 
i.e. curves $C$ such that the normalizations of $C\setminus D$ are isomorphic to $\bA^1$.

While there are a great number of theoretical results about log Gromov-Witten invariants, there are very few worked out examples. One of the aims of this paper is to remedy to that shortcoming. Our two main results will apply to a broad range of computations. One such application is \cite{BN20}.

The stable log maps can meet $D$ in a finite number of points and for such a point $P\in D$, one can consider $\mathcal{N}^P_\beta(X,D)$, the log GW invariant at $P$. We say the triple $(X,D,P)$ is an \emph{idealized log CY geometry for $\beta$} if $\mathcal{N}^P_\beta(X,D)$ equals the number of $\bA^1$-curves of class $\beta$ at $P$. The advantage of the log setting is that for generic $P$ and general $D$, the expectation is that $(X,D,P)$ is idealized. The disadvantage is that there always are points $P$ where $(X,D,P)$ is not idealized, so looking at idealized geometries only captures a part of the moduli space of stable log maps.

At this point, one may define the log BPS numbers as the number of $\bA^1$-curves in an idealized log CY geometry, which conjecturally is equivalent to Definition \ref{defn:logbps}. The next step then is to understand how $\bA^1$-curves contribute to $\mathcal{N}_\beta(X,D)$ in non-idealized log geometries. Unlike the CY case, there are countably many ways in which $\bA^1$-curves contribute to (higher degree) log GW invariants. The case of multiple covers over rigid integral curves was treated in \cite[Proposition 6.1]{GPS10}. In this paper we treat the next two cases: of non-rigid $\bA^1$-curves and of two rigid distinct $\bA^1$-curves glued together. We expect that combining these 3 cases will lead to a solution of the general case.

In non-idealized geometries, passing from $\bA^1$-curves to virtual counts is related to surprisingly interesting geometry. For example, in \cite[Proposition 1.16]{CGKT2} we proved that the contribution of multiple covers over rigid $\bA^1$-curves to the log BPS numbers is given as the Donaldson-Thomas invariants of loop quivers. In the first part of this paper, we introduce a certain moduli space $\mathcal{MMI}_\beta$ of sheaves of maximal intersection,  which can be regarded as a logarithmic analogue of $M_{\beta,1}(S)$. Using $\mathcal{MMI}_\beta$, we show that $\bA^1$-curves in $(X,D)$ share the same properties as rational curves in K3 surfaces.

Let us start with a maximally tangent stable log map of the simplest type, namely $f : \mathbb{P}^1 \to C \subset X$ with $f$ immersed and $C$ an integral rational curve maximally tangent to $D$. Such a curve contributes 1 to $\mathcal{N}_\beta(X,D)$. In other words, the naive multiplicity of the $\bA^1$-curve $C\setminus D$ is the correct multiplicity. 
Let us consider possible degenerations of $f$. 
For example, by deforming $D$, two log maps with image curve nodal cubics might collapse to one log map with image curve a cuspidal cubic.
Then the $\bA^1$-curve $C\setminus D$ is not immersed and it contributes 2 to $\mathcal{N}_\beta(X,D)$. More generally, we show (Corollary \ref{cor:ecompjac}) that an integral rational curve $C$ maximally tangent to $D$ and smooth at $D$ contributes its natural stable map multiplicity to $\mathcal{N}_\beta(X,D)$, i.e.\ the log structure introduces no new infinitesimal deformations.

In the second part of this paper, we give an in-depth description of the log deformation theory of stable log maps obtained by gluing two $\bA^1$-curves. 
If $C_1$ and $C_2$ are distinct immersed integral rational curves maximally tangent to $D$ at the same point, smooth at that point and intersect there in a general way, then they will contribute to $\mathcal{N}_{[C_1]+[C_2]}(X,D)$, a contribution we calculate in Theorem \ref{theorem_main}.
We compare it to the case of relative stable maps and find that the log structures more finely distinguish between the possible maps.

\subsection{Overview of methods}

The moduli space $\overline{\text{M}}_\beta(X,D)$ admits a finite forgetful morphism to the moduli space of stable maps \cite{Wise19}, and it is natural to ask about the interplay between infinitesimal deformations of the underlying stable maps and infinitesimal deformations of the log structures. While this is difficult to answer in general, we get explicit solutions in terms of topological data for certain components of dimension $0$. 
In this paper, we compute the contributions of such $0$-dimensional components of $\overline{\text{M}}_\beta(X,D)$ to the associated log Gromov-Witten invariants and log BPS numbers.

The $0$-dimensional components of $\overline{\text{M}}_\beta(X,D)$ we consider here fall into two categories. The simplest components are built from $\bA^1$-curves $C$ of class $\beta$. 
In the first part of this paper, 
we deal with such curves. 

The arguments are modelled on the case of K3 surfaces. 
Let $S$ be a K3 surface, $\gamma$ a curve class on $S$ 
and $C$ a rational curve of class $\gamma$ on $S$. 
Then the multiplicity of $C$ is $e(\overline{\Pic}^0(C))$, the Euler characteristic of the compactified Jacobian $\overline{\Pic}^0(C)$ of $C$, 
by \cite{FGS99} (see also \cite{Beau99}). 
Let us elaborate a little. 
For a rational curve $C$ 
we consider the moduli space $M_{0, 0}(C, [C])$ of genus $0$ stable maps to $C$, 
which is a thickened point corresponding to the normalization map $n: \mathbb{P}^1 \to C$. 
Let $l(C):=l(M_{0, 0}(C, [C]))$ be its length. 
Since $C$ has only planar singularities, 
it follows that $l(C)$ is equal to $m(C)$, the degree of the genus $0$ locus 
in the versal deformation space of $C$ (\cite[Theorem 1]{FGS99}), 
which in turn is equal to $e(\overline{\Pic}^0(C))$ 
by \cite[Theorem 2]{FGS99}. 

As a curve on a surface $S$, 
the natural multiplicity would be $l(M_{0, 0}(S, \gamma), n)$, 
the length of the moduli space of genus $0$ stable maps to $S$ at $n$, 
and it is equal to $l(C)$ if $S$ is a K3 surface 
(\cite[Theorem 2]{FGS99}). 
The key fact used in the proof is the smoothness 
of the relative compactified Jacobian over the complete linear system (at the points over $C$), 
or equivalently, 
of the moduli space $M_{\gamma,1}(S)$ of stable $1$-dimensional sheaves of class $\gamma$, proven in \cite{Muk84}.

\begin{Remark}
Let $C$ be a projective rational curve with planar singularities 
and $\check{\pi}: \check{C}\to C$ 
its minimal unibranch partial normalization(\cite[3.2]{Beau99}), 
i.e.~the partial normalization with $\check{C}$ unibranch 
such that any unibranch partial normalization factors through $\check{\pi}$. 
Then $m(C)=m(\check{C})$ holds 
(\cite[Proposition 3.3]{Beau99}, 
\cite[\S1]{FGS99}). 
In particular, 
if $C$ is immersed, i.e.~the differential of $n$ is nowhere vanishing, then there are no infinitesimal deformations and $m(C)=1$. 
In general, $m(C)$ is a product over the singularities of $C$ of factors depending on the analytic type 
(\cite[\S1]{FGS99}, \cite[Proposition 3.8]{Beau99}, \cite{She12}). 
See \cite[\S4]{Beau99} for explicit calculations.

\end{Remark}

Now we return to the case of an $\bA^1$-curve $C$ in $(X, D)$, 
with $K_X+D\sim 0$, $P=C\cap D\in D_\mathrm{sm}$ and $P\in C_\mathrm{sm}$. 
Denote by $n : \PP^1 \to C$ the normalization map. 
Then $n$ gives an isolated point in the moduli space of log stable maps. 
One of our main results, Corollary \ref{cor:ecompjac}, 
states that $n$ contributes $l(C)$ 
to the log Gromov-Witten invariant. This follows from two facts: (i) the infinitesimal deformations of $n$ as a log map can be identified with the infinitesimal deformations of the underlying stable map preserving the maximal tangency condition, and
(ii) such infinitesimal deformations of $n$ factor  scheme-theoretically through $C$.

For the proof of (ii), we introduce a certain moduli space $\mathcal{MMI}_\beta$ of sheaves of maximal intersection, 
which can be regarded as a logarithmic analogue of $M_{\gamma,1}(S)$. Just as in the case of K3 surfaces, 
we show the smoothness of $\mathcal{MMI}_\beta$ (Theorem \ref{thm:nonsing}), 
from which we deduce that 
infinitesimal deformations of $n$ as a log map  
factor through $C$. 
This might give a glimpse into a logarithmic version of sheaf-theoretic methods in curve counting \cite{LW15,MR20}, in analogy to the interpretation of genus 0 BPS numbers as Donaldson-Thomas invariants. 
Also, the relation between $M_{\beta,1}(X)$ and $\mathcal{MMI}_\beta$ will be the subject of further investigation.

Before explaining how these results relate to the local BPS numbers, let us introduce a little more general setting. 
Let $X$ be a smooth projective surface. In the first part of this paper, we will often require $X$ to be \emph{regular}, by which we mean that its irregularity $\HHH^1(\shO_X)$ vanishes. We denote by $D$ an effective divisor on $X$. 
We will sometimes require additional conditions on $X$ and $D$.

\begin{Definition}\label{defn:star}
Assume $X$ is regular, let $P\in D$ and let $\beta\in\hhh_2(X,\ZZ)$ be a curve class. Consider the linear system $\sho$ of curves of class $\beta$ that meet $D$ maximally at $P$ (see Definition \ref{defn:sho}) and let $L\subseteq \sho$.
We consider the following condition on $L$: 
\begin{center}
\emph{Condition $(\bullet)$} \quad
Every rational curve $C\in L$ that is unibranch at $P$ is in fact smooth at $P$.
\end{center}
\end{Definition}

In the setting of Theorem \ref{thm:bps} below, we expect Condition $(\bullet)$ to hold for general choices of $(X, D)$.

From Corollary \ref{cor} 
we derive an enumerative meaning of the local BPS numbers $n_\beta(K_X)$ subject to Conjecture \ref{conj:logbps} (\cite[Conjecture 1.3]{CGKT2}). 
This is a BPS version of the log-local principle pursued in \cite{vGWZ13,vGGR, BBvG1, BBvG2, BBvG3, NR}. 
Notice that Conjecture \ref{conj:logbps} is proven for $X=\PP^2$ in \cite{Bou19b} based on \cite{Gra20,Bou19a}, and for any del Pezzo surface and classes of arithmetic genus $\leq 2$ in \cite{CGKT1,CGKT2}.

\begin{Theorem}[Log-local principle for BPS numbers]\label{thm:bps}

Let $X$ be a del Pezzo surface, let $D$ be a smooth anticanonical curve on $X$ and let $P\in D$ be $\beta$-primitive (Definition \ref{defn:prim}). Then there is a finite number of rational curves in $\sho$. Assume that:
\begin{itemize}
\item Conjecture \ref{conj:logbps} holds for $X$. (For example when $X=\PP^2$ by \cite{Bou19a,Bou19b}, or for arithmetic genus $\leq2$ by \cite{CGKT1,CGKT2}).
\item Condition $(\bullet)$ holds for $\sho$. 
\end{itemize}
Then
\begin{equation} \label{eqn:loglocal}
n_\beta (K_X) = (-1)^{\beta\cdot D-1} \, (\beta\cdot D) \, \sum_{\substack{C\in\sho \\ \text{rational and} \\ \text{unibranch at } P}} l(C). 
\end{equation}
Note that $l(C)=e(\overline{\Pic}^0(C))$ by \cite{FGS99}.
\end{Theorem}

Note that each component of $\overline{\text{M}}_\beta(X,D)$ 
with tangency at a $\beta$-primitive $P$  comes from an (irreducible) $\bA^1$-curve, 
explaining the terms $l(C)$ in Equation \eqref{eqn:loglocal}.

The other category of zero-dimensional components consists of stable log maps $C\to X$ with image consisting of two distinct rational curves $Z_1$ and $Z_2$, each maximally tangent to $D$ at $P$. 
This situation occurs very often for a non-$\beta$-primitive point $P$. 
For example, if there are two $\bA^1$-curves $Z_1$ and $Z_2$ in the class $\beta$ 
which meet $D$ at the same point $P$, 
then their sum contributes to $\overline{\text{M}}_{2\beta}(X,D)$. 

In this case, $C$ consists of three components, two mapped to $Z_1$ and $Z_2$ and one mapped to $P$. 
It is not straightforward to see what log structure $C$ should have, 
unlike the case treated in the first part. 
Moreover, it turns out that the moduli spaces of stable log maps and relative stable maps are not isomorphic 
in the neighborhood of such a map, 
although they are guaranteed to give the same numerical invariants by \cite{AMW}. 

As the second main result of this paper, we calculate in Theorem \ref{theorem_main} the number of such log maps and the contributions of each to the log Gromov-Witten invariants in terms of intersection data, subject to genericity conditions. The proof involves a rather concrete (and long) calculation on infinitesimal families of log maps. We explicitly separate the infinitesimal deformations coming from the underlying stable maps from the ones coming from the log structure and explicitly describe both. This result sheds light on the interplay of the log structures with the underlying stable maps.

Now let us give a little more detailed explanation 
on what we are going to deal with.

\subsection{Log BPS numbers}\label{sec:logbps}

Let $X$ be a smooth projective surface and 
let $\beta\in\hhh_2(X,\ZZ)$ be a curve class. 
We write $w=\beta\cdot D$ and assume that $w>0$. If $X$ is regular, then there is a unique $L\in\Pic(X)$ such that $c_1(L)$ is Poincar\'e dual to $\beta$. By $\beta|_D$ 
we mean $L|_D\in\Pic^w(D)$. For $X$ regular, set
\begin{equation}\label{eqn:dbeta}
D(\beta):=\{P\in D_{\mathrm{sm}} \, : \, \beta|_D = \shO_D(wP) \text{ in } \Pic^w(D)\}.
\end{equation}

\begin{Remark}
(1)
If $D$ is an elliptic curve, 
$D(\beta)$ is a torsor for $\Pic^0(D)[w]\simeq \ZZ/w\ZZ \times \ZZ/w\ZZ$ 
(cf.\ \cite[Lemma 2.14]{CGKT2}). 

(2)
If $D=\bigcup_{i=1}^k D_i$ is the decomposition into irreducible components 
and $D(\beta)\not=\emptyset$, 
then $D(\beta)\subset D_{i_0}$ for a unique $i_0$ 
and $\beta|_{D_i}\sim 0$ for $i\not=i_0$. 
\end{Remark}

Part (2) is true because if $P\in D(\beta)$ and $P\in D_{i_0}$, then from $\beta|_{D}=\shO_{D}(wP)$ we see that $\beta|_{D_i}=\shO_{D_i}$ for $i\ne i_0$. In particular, $\beta|_{D_i}$ cannot be of the form $\shO_{D_i}(wP)$ for any $P\in D_i$ for degree reasons, and so $D(\beta)\subset D_{i_0}$.

\begin{Definition} \label{defn:prim}
Let $P\in D(\beta)$. Then $P$ is \emph{$\beta$-primitive} if there is no decomposition into non-zero pseudo-effective classes $\beta=\beta'+\beta''$, with $\beta'\cdot D>0$ and such that $P\in D(\beta')$.
\end{Definition}

\begin{Proposition}[Proposition 4.11 in \cite{CGKT2}] 
Assume that $X$ is a del Pezzo surface and that $D$ is a smooth anticanonical curve on $X$. 
If the pair $(X, D)$ is general, then there is a $\beta$-primitive point $P\in D(\beta)$.
\end{Proposition}

If $(X, D)$ is a log smooth pair, 
denote by $\nxd$ the genus 0 log Gromov-Witten invariant of maximal tangency and class $\beta$ of $(X, D)$, whose definition we review in Section \ref{sec:loggw}. 
If $X$ is regular, $D$ is smooth and $p_a(D)>0$, 
then $D(\beta)$ is a finite set 
and the moduli space decomposes into a disjoint union according to $P\in D(\beta)$:
\[
\overline{\text{M}}_\beta(X,D) = \bigsqcup_{P\in D(\beta)} \overline{\text{M}}^P_\beta(X,D).
\]
Hence we can define the contribution $\nxdp$ from each $P$ so that 
\[
\nxd = \sum_{P\in D(\beta)} \nxdp 
\]
holds.

We often take $D$ to be anticanonical. 
Note that a regular surface with a nonzero anticanonical curve is rational 
by Castelnuovo's criterion. 

\begin{Definition}\label{defn:logbps}
Assume that $X$ is rational and $D$ is smooth anticanonical, and let $P\in D(\beta)$. The \emph{log BPS number at $P$}, $m^P_\beta$, is defined recursively via 
\[
\nxdp =  \sum_{k|\beta}     
\frac{(-1)^{(k-1)\,w/k}}{k^2} \, m^P_{\beta/k},
\]
where $m^P_{\beta'}:=0$ if $P\not\in D(\beta')$. 
\end{Definition}

\begin{Conjecture}[Conjecture 1.3 in \cite{CGKT2}] \label{conj:logbps}
For all $P,P'\in D(\beta)$,
\[
m_\beta^P = m_\beta^{P'}.
\]
Equivalently,
for $P\in D(\beta)$,
\[
n_\beta(K_X) = (-1)^{\beta\cdot D-1} \, (\beta\cdot D) \, m_\beta^P.
\]

\end{Conjecture}

What makes this highly nontrivial is that 
$\overline{\text{M}}^P_\beta(X,D)$ can be quite different 
according to the local geometry of $D$ near $P$ (see \cite[\S\S6.1]{CGKT2} and \S\S\ref{subsection:applications}).
Conjecture \ref{conj:logbps} was proven for $\PP^2$ in \cite{Bou19a,Bou19b} using \cite{Gra20}.

\begin{Definition}\label{defn:sho}
Assume that $X$ is regular and $P\in D_{\mathrm{sm}}$. We denote by $|\beta|$ the linear system of curves of class $\beta$ and set
\[
\sho := \{ C \in |\beta| \, : \, C|_D \supseteq wP \text{ as subschemes of } D \},
\]
as well as its open subsets 
\[
\shoo := \{ C \in |\beta| \, : \, C|_D = wP \text{ as subschemes of } D \}
\]
and 
\[
\shooo := \{ C \in |\beta| \, : \, C|_D = wP \text{ as subschemes of } D 
\hbox{ and $C$ is integral}\}. 
\]
Moreover, we write
\[
p_a(\beta) := \frac{1}{2} \beta (\beta+K_X)+1
\]
for the arithmetic genus of members of $|\beta|$.
\end{Definition}

\begin{Remark}
For a rational $X$ with smooth anticanonical $D$, notice that the set of rational curves in $\shoo$ is identified with the set of rational curves in $\sho$.
For a regular surface $X$, a curve $D$ on $X$ and a $\beta$-primitive point $P\in D(\beta)$, note that every member of $\shoo$ is an integral curve, i.e. $\shoo=\shooo$.  Note however that if a curve $C \in |\beta|$ contains the component $D_i$ of $D$ passing through $P$, then $C \in \sho$ but $C\not\in \shoo$.
\end{Remark}

In Section \ref{sec:nonsing} we construct certain moduli spaces, denoted
$\mc{MMI}_\beta$ and $\mc{MMI}_\beta^P$, associated to any smooth surface $X$ and a curve $D$ on it. 
These moduli spaces parametrize
certain sheaves supported on integral curves and having ``maximal intersection'' with $D$. 
 For $\mc{MMI}_\beta^P$ with $P\in D_{\mathrm{sm}}$, the additional condition is imposed that the tangency is at $P$. 
If $X$ is regular and the Abel map of $D$ is immersive at $P$ 
(e.g.\ if $D$ is integral with $p_a(D)>0$ 
or $D$ is anticanonical in rational $X$; see Lemma \ref{lem_mi_is_reduced}(3)), 
then, by Lemma \ref{lem_mi_is_reduced}(2), $\mc{MMI}_\beta$ decomposes 
scheme-theoretically as a disjoint union
\[
\mc{MMI}_\beta = \coprod_{P\in D(\beta)} \mc{MMI}_\beta^P.
\]

\begin{Theorem}[=Theorem \ref{thm:smooth}]\label{thm:nonsing}
Let $X$ be a smooth projective rational surface, $D$ an anticanonical curve on $X$ 
and $P\in D_\mathrm{sm}$. 
Then $\mathcal{MMI}_\beta$ and $\mathcal{MMI}_\beta^P$ 
are nonsingular of dimension $2p_a(\beta)=\beta^2-w+2$. 

Consequently, 
the relative compactified Picard scheme over $\shooo$ 
is nonsingular at a point $[F]$ over $[C]$ 
if $F$ is an invertible $\cO_C$-module near $P$ 
(or, equivalently, $F|_D\cong \cO_C|_D$). 
\end{Theorem}

We use this theorem to calculate the contribution of an $\bA^1$-curve to the log Gromov-Witten invariant. 
The simplest components in the moduli space of genus 0 basic stable log maps consist of (possibly thickened) points. 
We will mainly be concerned with the case $(K_X+D).\beta=0$, 
since otherwise the virtual dimension is nonzero. 
Then, one such case arises from an irreducible $\bA^1$-curve. 
The following result, 
proven in \cite{CGKT2} subject to Theorem \ref{thm:nonsing} (and Lemma \ref{lem_mi_is_reduced}(2)), 
calculates the contribution of such a point to $\nxd$. 

\begin{Corollary}[Proposition 1.7(3) in \cite{CGKT2}]\label{cor:ecompjac}
Let $X$ be a smooth projective rational surface and $D$ an anticanonical curve. 
Let $C$ be an irreducible rational curve of class $\beta$ maximally tangent to $D$ at $P\in D(\beta)$. Denote the normalization map by $n : \PP^1 \to C$ and assume that $C$ is smooth at $P$. 
Then $n$ contributes $l(C):=l(M_{0, 0}(C, [C]))$
to $\nxd$.

\end{Corollary}

Consequently, we have:

\begin{Corollary}[Proposition 1.7(5) in \cite{CGKT2}]\label{cor}
Let $X$ be a del Pezzo surface
and $D$ a smooth anticanonical curve on $X$, 
and let $P\in D(\beta)$ be $\beta$-primitive. 
Assume that $\shoo$ satisfies Condition $(\bullet)$. 
Then
\[
m^P_\beta = l(\overline{\text{M}}^P_\beta(X,D)) = \sum_{\substack{C\in\shoo \\ \text{rational and} \\ \text{unibranch at } P}} l(C).  
\]

\end{Corollary}
In Corollary \ref{cor}, Condition $(\bullet)$ is needed to ensure that the compactified Picard variety of $C$ is contained in $\mc{MMI}_\beta$.

\subsection{Contribution of curves with two image components }

In the second part of this paper, Sections \ref{sec:loggw} and \ref{sec:thmmain}, we consider another type of zero-dimensional component, where the images of the stable log maps consist of two maximally tangent rational curves.

\begin{Theorem}\label{theorem_main}
Let $(X, D)$ be a pair consisting of 
a smooth surface and an effective divisor. 
Denote by $\barM_\beta=\barM_\beta(X, D)$ 
the moduli stack of 
maximally tangent genus $0$ basic stable log maps of class $\beta$ 
to the log scheme associated to $(X, D)$. 

Let $Z_1$ and $Z_2$ be proper integral curves on $X$ satisfying the following: 
\begin{enumerate}
\item 
$Z_i$ is a rational curve of class $\beta_i$ maximally tangent to $D$, 
\item
$(K_X+D).\beta_i=0$, 
\item
$Z_1\cap D$ and $Z_2\cap D$ consist of the same point $P\in D_\mathrm{sm}$, 
and 
\item
The normalization maps $f_i: \PP^1\to Z_i$ are immersive and 
$(Z_1.Z_2)_P=\min\{d_1, d_2\}$, 
where $d_i=D.Z_i$. 
\end{enumerate}

Write $d_1=de_1, d_2=de_2$ with $\gcd(e_1, e_2)=1$. 
Then there are $d$ stable log maps in $\barM_{\beta_1+\beta_2}$ 
whose images are $Z_1\cup Z_2$, 
and they are isolated with multiplicity $\min\{e_1, e_2\}$.

When $X$ is projective and $(X,D)$ is log smooth, then these curves contribute $\min\{d_1, d_2\}$ to the log Gromov-Witten invariant $\mathcal{N}_{\beta_1+\beta_2}(X, D)$. 
\end{Theorem}

\begin{Remark}
In Theorem \ref{theorem_main}, the condition that $(Z_1.Z_2)_P=\min\{d_1, d_2\}$ means that $Z_1$ and $Z_2$ are assumed to intersect generically at $P$. 
We expect that this condition is satisfied for general $D$ when $Z_1\not=Z_2$. 

If $d_1\not=d_2$, then as a consequence of the immersivity of $f_i$ and maximal tangency, $Z_1$ and $Z_2$ are smooth at $P$ and the condition $(Z_1.Z_2)_P=\min\{d_1, d_2\}$ holds. 
In the case $d_1=d_2=d$, in analytic coordinates $x, y$ near $P$ with $D=(y=0)$, 
we can write $Z_i=(y=a_ix^d+\dots)$. Then $(Z_1.Z_2)_P=\min\{d_1, d_2\}$ translates into $a_1\not=a_2$. 

An example where this condition is obviously not satisfied is the case $Z_1=Z_2$. 
In this case, the space of log maps with image cycle $Z_1+Z_2$, 
as well as its contribution to the log Gromov-Witten invariant, 
is quite different (\cite[Proposition 6.1]{GPS10}).

\end{Remark}

\begin{Remark}

In the different setting of the degeneration formula \cite{Lirel1,Lirel2,AF16,Chen14deg,KLR,Ran19}, the terms $d_1$ and $d_2$ occur as the number of log lifts. 
For us, $d$ is the number of ways of endowing the underlying stable map with a log structure. And $\min\{e_1, e_2\}$ is the length of the corresponding points of $\barM_{\beta_1+\beta_2}$.

\end{Remark}

It is illuminating to compare Theorem \ref{theorem_main} with the analogous result \cite{Tak17} for the relative stable maps of \cite{Lirel1,Lirel2}. Whereas there is only one relative stable map with multiplicity $\min\{d_1, d_2\}$, there are $d$ log maps each with multiplicity $\min\{e_1, e_2\}$, making the same contribution as expected by \cite{AMW}. 
This illustrates that there can be several ways of
associating a log map to a relative map.

Theorem \ref{theorem_main} is illustrated by Example \ref{ex1}, which the reader may consider as the running example for the second part of this paper. We fully work out the same example in the language of relative stable maps \cite{Lirel1,Lirel2} in Example \ref{ex2} following \cite{Tak17}.

\addtocontents{toc}{\protect\setcounter{tocdepth}{0}}
\section*{Acknowledgements}
\addtocontents{toc}{\protect\setcounter{tocdepth}{1}}

We wish to thank Mark Gross and Dhruv Ranganathan for a discussion on multiplicities of stable log maps that motivated parts of this work. We are grateful to Alexei Oblomkov for discussions on multiplicities of stable maps to surfaces. We thank Helge Ruddat and Travis Mandel for a discussion clarifying the relationship between stable map multiplicity and tropical multiplicity. We especially thank Helge Ruddat for suggesting to us the example of Sections \ref{sec:ex3comppre} and \ref{sec:ex3comp}.

JC is supported by the Korea NRF grant NRF-2018R1C1B6005600. JC would like to thank Korea Institute for Advanced Study for the support where some of the work for this paper was completed. MvG is supported by the EC REA MSCA-IF-746554. SK is supported in part by NSF grant DMS-1502170 and NSF grant DMS-1802242. NT is supported by JSPS KAKENHI Grant Number JP17K05204. This project has received funding from the European Union's Horizon 2020 research and innovation programme under the Marie Sklodowska-Curie grant agreement No 746554.

\addtocontents{toc}{\protect\setcounter{tocdepth}{2}}

\section{Illustration of the main results}\label{sec:ex}

We illustrate the two main results, Corollary \ref{cor:ecompjac} and Theorem \ref{theorem_main}. The examples below may form the basis of future research directions.

\subsection{Analogy with K3 surfaces}\label{sec:K3}

Assume now that $X$ is a del Pezzo surface and that $D$ is smooth anticanonical. Let $\beta\in\hhh_2(X,\ZZ)$ be the class of an integral curve, denote by $p_a(\beta)$ its arithmetic genus and choose $P\in D(\beta)$ to be $\beta$-primitive. Then the linear system $\sho$ is of dimension $p_a(\beta)$ \cite[Proposition 4.15]{CGKT2}. 

Denote by $S$ a K3 surface and let $\gamma$ be a curve class of arithmetic genus $h$. In analogy to Definition \ref{defn:logbps}, one associates genus 0 BPS numbers $r_{0,h}$ to $\gamma$, see \cite{Pan17}. Remarkably \cite{KMPS10}, $r_{0,h}$ depends only on $h$ (and the genus 0 reduced Gromov--Witten invariant of class $\gamma$ depends only on $\gamma^2$ and the divisibility of $\gamma$ in ${\rm H}_2(S,\ZZ)$). Choose a complete linear system $L$ of curves of arithmetic genus $h$. Then $L$ is $h$-dimensional as is $\sho$. Under the assumption that all curves in $L$ are integral, $r_{0,h}$ is given as the sum of $l(C)=e(\overline{\Pic}^0(C))$ for $C$ a rational curve in $L$ (\cite{YZ96,Beau99,Xi99,FGS99}).
This is in perfect analogy to Corollary \ref{cor}.

By \cite[Lemma 4.10]{CGKT2}, saying that $P$ is $\beta$-primitive amounts to $P$ being of maximal order in the group structure on $D$ arising from choosing a suitable element of $D(\beta)$ as zero element. Keeping track of the order of $P$ is analogous to keeping track of the divisibility of $\gamma$ as an element of ${\rm H}_2(S,\ZZ)$. So requiring $P$ being $\beta$-primitive corresponds to $\gamma$ being primitive as an element of ${\rm H}_2(S,\ZZ)$. And if $\gamma$ is primitive, then $r_{0,h}$ also agrees with the genus 0 reduced Gromov--Witten invariant of class $\gamma$.



In \cite{CGKT2}, we calculated $m^P_\beta$  for $\beta$-primitive $P$ and $p_a(\beta)\leq2$. We found (Theorems 1.8 and 1.9) in these cases that $m_\beta^P$ depends only on the intersection number $e(S)-\eta$ for $\eta$ the number of line classes $l$ with $\beta.l=0$. Similarly, $r_{0,h}$ only depends on the intersection number $\beta\cdot\beta=2h-2$ \cite{KMPS10}.

The analogy carries over to SYZ fibrations. For K3 surfaces, \cite{Lin17a,Lin17b} proves that counts of Maslov index 0 disks with boundary on a SYZ-fiber correspond to tropical curves in the base. In the case of $\PP^2$, the analogous correspondence \cite{Gra20} is between log BPS numbers and tropical curves in the scattering diagram.

\subsection{Fully worked out example and comparison with tropical multiplicity}
\label{sec:ex:sec}

Armed with Corollary \ref{cor:ecompjac} and Theorem \ref{theorem_main}, we can compute genus 0 maximal tangency log Gromov--Witten invariants of low degrees by explicitly finding all the stable log maps that contribute and weighting them with their multiplicity. We fully work this out in one example adapted from the tropical vertex \cite{GPS10} and compare it with the analogous tropical picture.

One of the features of log Gromov--Witten theory is that each stable log map admits a tropicalization coming from the domain curve. These tropical curves carry multiplicities that are related to the log structure. In this, it is different from the multiplicities of Corollary \ref{cor:ecompjac}, which come from the stable maps. This example illustrates this difference in the case of cuspidal cubics. Classically their multiplicity is given by Corollary \ref{cor:ecompjac}. On the tropical side, while we may guess what cuspidal tropical cubics are (moving $P_1,\dots,P_6$ in Figures \ref{fig:trop1} and \ref{fig:trop2} leading to vertices of valency $>3$), it is not clear what their multiplicities are and the example considered might give some insight into that.

In addition, we consider the contributions to the log invariants of reducible curves in Sections \ref{sec:223} and \ref{sec:ex3comppre}. One may write down the corresponding reducible tropical curves and stipulate what their tropical multiplicity is. We leave that to future work and simply find the tropical curves in a generic situation.

Start with $\PP^2$ with anticanonical boundary a cycle of 3 disjoint lines $\widetilde{D}=\widetilde{D}_1+\widetilde{D}_2+D_{\rm out}$. We blow up 3 smooth points $P_1,P_2,P_3$ on $\widetilde{D}_1$ and 3 smooth points $P_4,P_5,P_6$ on $\widetilde{D}_2$ leading to 6 exceptional divisors $E_{ij}$, $i=1,2$, $j=1,2,3$. The resulting surface $S$ is a weak del Pezzo surface. We choose as its anti-canonical boundary the strict transform of $\widetilde{D}$, namely $D=D_1+D_2+D_{\rm out}$ with $D_i\sim H-\sum_{j=1}^3 E_{ij}$ for $H$ the pullback of the hyperplane class in $\PP^2$. Let $\beta=3H-\sum_{i,j} E_{ij}$ be the anticanonical curve class, which is of arithmetic genus 1. We compute the invariant $\mathcal{N}_{\beta}(S,D)$ of genus 0 curves of class $\beta$ maximally tangent to $D$, necessarily meeting $D$ at a smooth point of $D_{\rm out}$.

We first compute $\mathcal{N}_{\beta}(S,D)$ classically. 
We use the fact that by \cite[Proposition 5.3]{GPS10}, $\mathcal{N}_{\beta}(S,D)$ equals the virtual count of rational curves in $\PP^2$ of class $3H$ passing through $P_1,\dots,P_6$ and maximally tangent to $D_{\rm out}$.

The classical count consists in finding all the rational cubics contributing to the count and weighting them by their multiplicities of Corollary \ref{cor:ecompjac} and Theorem \ref{theorem_main}. Provided they are smooth at the point of contact with $D_{\rm out}$, nodal cubics have multiplicity 1 and cuspidal cubics have multiplicity 2. More interesting contributions arise when the point of contact is not smooth.

By dimensional reasons, there are only a finite number of possible points of contact $P$ with $D_{\rm out}$. Let $C_1$ and $C_2$ be two maximally tangent rational curves passing through $P_1,\dots,P_6$ each and meeting $D_{\rm out}$ at $Q_1: z=z_1$ and $Q_2: z=z_2$, respectively, where $z$ is a coordinate with $(D_1\cup D_2)\cap D_{\rm out}=\{0, \infty\}$. We take the relationship
\[
(C_1-E_1-\cdots-E_6)-(C_2-E_1-\cdots-E_6)\sim 0 \text{ on } S,
\]
which restricts to
\[
3Q_1-3Q_2\sim 0\in {\rm Pic}^0(D_{\rm out}\cup D_1\cup D_2)\simeq \mathbb{C}^*
\]
to obtain that
\[
(z_1/z_2)^3 = 1
\]
as the condition for $3Q_1-3Q_2$ to be a principal divisor on $D_{\rm out}\cup D_1\cup D_2$ (explicitly, the divisor of the rational function equal to $(z-z_1)^3/(z-z_2)^3$ on  $D_{\rm out}$ and identically equal to 1 on $D_1\cup D_2$).  Thus, the possible points of contact form a torsor for $\mu_3$ and in particular there are 3 of them.

We fix $P$ one of these points of contact and compute the invariant $\mathcal{N}_{\beta}^P(S,D)$ at $P$. Then the cubics passing through $P_1,\dots,P_5$ and maximally tangent to $D_{\rm out}$ at $P$ form a pencil and all pass through $P_6$.
Resolving the base points, we obtain an elliptic fibration $Y$ that is a surface of Euler number 12. The resolution is obtained by blowing up $S$ at $P$ and then blowing up 2 more times in succession at the unique point over $P$ in the strict transform of $D_{\rm out}$. Denote the exceptional divisors $E_1$, $E_2$, $E_3$ according to the order of blow up.

We find all the rational cubics that contribute to the count. The cubic $\widetilde{D}_1+\widetilde{D}_2+D_{\rm out}$ does not contribute to $\mathcal{N}_{\beta}(S,D)$.
Its proper transform $F_0$ is a fiber of $Y$, a cycle of 3 $\PP^1$s with Euler number 3.
By the same argument as in \cite[Section 5.3]{BBvG2}, the pencil contains a unique member corresponding to a cubic $C$ that is singular at $P$. Unlike \cite[Section 5.3]{BBvG2}, $C$ can have up to three branches at $P$:
\begin{enumerate}
    \item $C$ may be irreducible and nodal at $P$. Then it does not contribute to $\mathcal{N}_{\beta}(S,D)$.
    \item $C$ may be cuspidal at $P$. Then it contributes 1 to $\mathcal{N}_{\beta}(S,D)$ by \cite[Proposition 4.21(2)]{CGKT2}.
    \item $C$ may be reducible and nodal at $P$. Then it is the union of a conic tangent to $D_{\rm out}$ at $P$ and a line passing through $P$. By Theorem \ref{theorem_main}, $C$ contributes 1 to $\mathcal{N}_{\beta}(S,D)$.
    \item $C$ is the union of three lines passing through $P$. This is a case that Theorem \ref{theorem_main} does not cover. We compute its contribution to be 3 below.
\end{enumerate}

In all cases, we will see that
\[
\mathcal{N}^P_{\beta}(S,D) = 6,
\]
so that
\[
\mathcal{N}_{\beta}(S,D) = 3\times 6 = 18.
\]

\subsubsection{$C$ is irreducible and nodal at $P$}

Assume first that $P_1,\dots,P_6$ are general (within the restriction of lying on $\widetilde{D}_1\cup\widetilde{D}_2$), so that $C$ is nodal at $P$, with one branch tangent to $D_{\rm out}$. The cycle of 3 $\PP^1$s consisting of the strict transforms of $C$, $E_1$ and $E_2$ gives a fiber $F_1$ of $Y$. There is a kind of symmetry between $F_0$ and $F_1$: Contracting $E_3$, $D_{\rm out}$ and either $\widetilde{D}_1$ or $\widetilde{D}_2$, we get the dual picture. The fiber $F_1$ has Euler number 3 and does not contribute to the count.
All other curves in the fibration are either smooth cubics, which have Euler number 0 or rational cubics smooth at $P$. Nodal cubics have Euler number 1 and cuspidal cubics Euler number 2. They all contribute to $\mathcal{N}_{\beta}(S,D)$.

So if $C$ is nodal at $P$, by the additivity of Euler numbers,
\[
\mathcal{N}^P_{\beta}(S,D) = \# \{ \text{ nodal cubics } \} + 2 \, \# \{ \text{ cuspidal cubics } \} = e(Y) - e(F_0) - e(F_1) = 6.
\]

For specific choices of $P_1,\dots,P_6$, cuspidal cubics smooth at $P$ appear. For example, let $D_{\rm out}$ be the line at infinity, let $\widetilde{D}_1$ be given by $y+1=0$ and let $\widetilde{D}_1$ be given by $x+y+1=0$. Take $P_1,\dots,P_6$ to be the intersections of $y^2=x^3$ with $\widetilde{D}_1$ and $\widetilde{D}_2$. Then $y^2=x^3$ is a cuspidal cubic smooth at $D_{\rm out}$ that contributes 2 to $\mathcal{N}_{\beta}(S,D)$. As in the case of regular Gromov--Witten theory, this is the example of two nodal cubics coming together in a deformation to form a cuspidal one.

\subsubsection{$C$ is cuspidal at $P$}

If $C$ is cuspidal at $P$, then $F_1$, the strict transform of $C$ joined with $E_1$ and $E_2$ is a chain of 3 $\PP^1$s, of Euler number 4 and leading to a different fibration $Y$. Nonetheless,
\begin{align*}
\mathcal{N}^P_{\beta}(S,D) 
&= (\text{Contribution of } C) + \# \{ \text{ nodal cubics } \} + 2 \, \# \{ \text{ cuspidal cubics } \} \\ &= 1 + e(Y) - e(F_0) - e(F_1) = 6.
\end{align*}

\subsubsection{$C$ is reducible and nodal at $P$}

\label{sec:223}

Assume that $P_1,P_2,P_4,P_5$ are general. Denote by $C_2$ one of the two conics that pass through $P_1,P_2,P_4,P_5$ and are tangent to $D_{\rm out}$. Denote by $P$ the point of intersection of $C_2$ with $D_{\rm out}$. Choose $L$ a general line passing through $P$ and denote by $P_3$ and $P_6$ its points of intersection with $\widetilde{D}_1$ and $\widetilde{D}_2$, respectively.

Given these choices of $P_1,\dots,P_6$, the singular cubic in the pencil is given by $C_2\cup L$. It contributes 1 to  $\mathcal{N}_{\beta}(S,D)$ by Theorem \ref{theorem_main}. As a fiber $F_1$ of $Y$ it becomes a cycle of 4 $\PP^1$s of Euler number 4. Then
\begin{align*}
\mathcal{N}^P_{\beta}(S,D) &= (\text{Contribution of } C_2\cup L) + \# \{ \text{ nodal cubics } \} + 2 \, \# \{ \text{ cuspidal cubics } \} \\ &= 1 + e(Y) - e(F_0) - e(F_1) = 6.
\end{align*}

\subsubsection{$C$ has three branches at $P$}
\label{sec:ex3comppre}

Assume that $P_1,\dots,P_6$ are such that the lines joining $P_{i\!\!\mod 3}$ all meet at $P\in D_{\rm out}$. Then the union of these lines $L_1\cup L_2 \cup L_3$ is the singular member of the pencil. Denote by ${\rm Contr}^{(S,D)}(1,1,1)$ the contribution of $L_1\cup L_2 \cup L_3$ to $\mathcal{N}_{\beta}(S,D)$. In $Y$, it yields a tree of 5 $\PP^1$s, $F_1$, of Euler number 6. By deformation-invariance,
\begin{align*}
6= & \,\mathcal{N}^P_{\beta}(S,D) = {\rm Contr}^{(S,D)}(1,1,1) + \# \{ \text{ nodal cubics } \} + 2 \, \# \{ \text{ cuspidal cubics } \} \\
= & \, {\rm Contr}^{(S,D)}(1,1,1) + e(Y) - e(F_0) - e(F_1) = {\rm Contr}(1,1,1) + 3.
\end{align*}

We conclude that ${\rm Contr}^{(S,D)}(1,1,1)=3$. Moreover, among the other curves contributing, there are either 3 nodal cubics smooth at $P$ or 1 nodal cubic and 1 cuspidal cubic. The former case can be verified by looking at $\PP^2$ with its toric boundary and the pencil
\[
a(Z^3 + X^3 + Y^3 + 3X^2Y + 3XY^2) + bXYZ,
\]
for $[a:b]\in\PP^1$.

\subsubsection{Tropical count}

We next compute the same invariant tropically. To do so, we unwind the tropical computation of \cite{GPS10}. In fact, $\mathcal{N}_{\beta}(S,D) = 18$ is computed in \cite[Section 6.4]{GPS10} from a scattering diagram computation. The tropical count is the count of tropical curves in the fan of $\PP^2$, weighted by their tropical multiplicity, that have only one ray of weight 3 going into the direction corresponding to $D_{\rm out}$ and 3 rays each of weight 1 coming from fixed directions corresponding to $\widetilde{D}_1$ and $\widetilde{D}_2$. We refer to tropical correspondence results and multiplicity calculations to \cite{Mikh05,NiSi,vGOR} and especially \cite{ManRu,ManRu19} for incidence conditions along the toric boundary as is the case here. Here we content ourselves with describing the tropical curves and computing their multiplicity. If we choose $P_1,\dots,P_6$ as in Figure \ref{fig:trop1}, solving the combinatorial problem leads to the 3 tropical curves of Figures \ref{fig:trop1} and \ref{fig:trop2}. Their multiplicity is given by the product of the multiplicities of the 3-valent vertices and are indicated in the figures. One of them has multiplicity 12, the other two each have multiplicity 3. We thus recover $\mathcal{N}_{\beta}(S,D) = 18$.

\begin{figure}[htb]
\begin{center}
\caption{A tropical curve of multiplicity 12 in the fan of $\PP^2$}
\label{fig:trop1}
\begin{tikzpicture}[smooth, scale=0.7]
\draw[step=1cm,gray,very thin] (-6.5,-2.5) grid (2.5,6.5);
\draw[ultra thick] (-2,0) to (-6.5,0);
\draw[ultra thick] (-2,-2.5) to (-2,0);
\draw[ultra thick] (-2,0) to (2.5,4.5);
\node at (-1.6,-2.3) {$\scriptstyle{\widetilde{D}_2}$};
\node at (2.4,3.7) {$\scriptstyle{D_{\rm out}}$};
\node at (-6.3,-0.35) {$\scriptstyle{\widetilde{D}_1}$};
\node at (-6.3,1.2) {$\scriptstyle{P_3}$};
\node at (-6.3,2.2) {$\scriptstyle{P_2}$};
\node at (-6.3,3.7) {$\scriptstyle{P_1}$};
\node at (-5.3,-2.3) {$\scriptstyle{P_4}$};
\node at (-3.3,-2.3) {$\scriptstyle{P_5}$};
\node at (1.3,-2.3) {$\scriptstyle{P_6}$};
\draw[thick] (-6.5,1) to (-5,1) to (-4,2);
\draw[thick] (-5,-2.5) to (-5,1);
\draw[thick] (-6.5,2) to (-4,2) to (-3,2.5);
\draw[thick] (-3,-2.5) to (-3,2.5) to (-2,3.5);
\node at (-2.3,2.9) {$\scriptsize{2}$};
\filldraw (-3,2.5) circle (1mm);
\node at (-3.25,2.7) {$\scriptsize{2}$};
\filldraw (-2,3.5) circle (1mm);
\node at (-2,3.8) {$\scriptsize{2}$};
\draw[thick] (-6.5,3.5) to (-2,3.5) to (1,5.5);
\filldraw (1,5.5) circle (1mm);
\node at (0.75,5.6) {$\scriptsize{3}$};
\draw[thick] (1,-2.5) to (1,5.5) to (2.3,6.8);
\node at (1.85,6.7) {$\scriptsize{3}$};
\end{tikzpicture}
\end{center}
\end{figure}
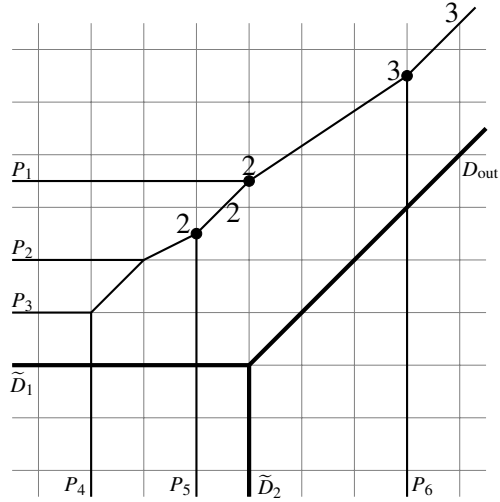

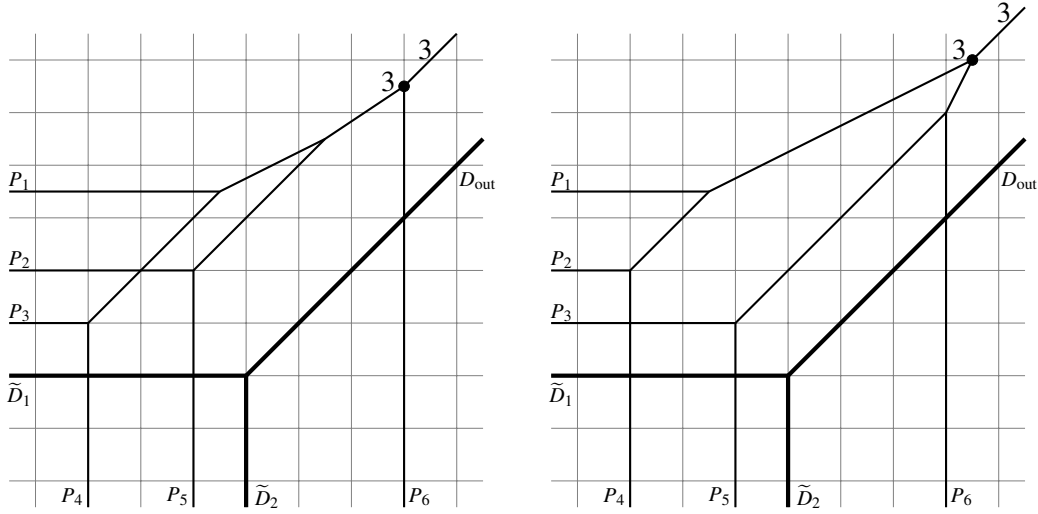
\begin{figure}[htb]
\begin{center}
\caption{Two tropical curves of multiplicity 3 each in the fan of $\PP^2$}
\label{fig:trop2}
\begin{tikzpicture}[smooth, scale=0.7]
\draw[step=1cm,gray,very thin] (-6.5,-2.5) grid (2.5,6.5);
\draw[ultra thick] (-2,0) to (-6.5,0);
\draw[ultra thick] (-2,-2.5) to (-2,0);
\draw[ultra thick] (-2,0) to (2.5,4.5);
\node at (-1.6,-2.3) {$\scriptstyle{\widetilde{D}_2}$};
\node at (2.4,3.7) {$\scriptstyle{D_{\rm out}}$};
\node at (-6.3,-0.35) {$\scriptstyle{\widetilde{D}_1}$};
\node at (-6.3,1.2) {$\scriptstyle{P_3}$};
\node at (-6.3,2.2) {$\scriptstyle{P_2}$};
\node at (-6.3,3.7) {$\scriptstyle{P_1}$};
\node at (-5.3,-2.3) {$\scriptstyle{P_4}$};
\node at (-3.3,-2.3) {$\scriptstyle{P_5}$};
\node at (1.3,-2.3) {$\scriptstyle{P_6}$};
\draw[thick] (-6.5,1) to (-5,1) to (-2.5,3.5);
\draw[thick] (-5,-2.5) to (-5,1);
\draw[thick] (-6.5,2) to (-3,2) to (-0.5,4.5) to (1,5.5);
\draw[thick] (-3,-2.5) to (-3,2);
\draw[thick] (-6.5,3.5) to (-2.5,3.5) to (-0.5,4.5);
\filldraw (1,5.5) circle (1mm);
\node at (0.7,5.6) {$\scriptsize{3}$};
\draw[thick] (1,-2.5) to (1,5.5) to (2,6.5);
\node at (1.4,6.2) {$\scriptsize{3}$};
\end{tikzpicture}
\hspace{2mm}
\begin{tikzpicture}[smooth, scale=0.7]
\draw[step=1cm,gray,very thin] (-6.5,-2.5) grid (2.5,6.5);
\draw[ultra thick] (-2,0) to (-6.5,0);
\draw[ultra thick] (-2,-2.5) to (-2,0);
\draw[ultra thick] (-2,0) to (2.5,4.5);
\node at (-1.6,-2.3) {$\scriptstyle{\widetilde{D}_2}$};
\node at (2.4,3.7) {$\scriptstyle{D_{\rm out}}$};
\node at (-6.3,-0.35) {$\scriptstyle{\widetilde{D}_1}$};
\node at (-6.3,1.2) {$\scriptstyle{P_3}$};
\node at (-6.3,2.2) {$\scriptstyle{P_2}$};
\node at (-6.3,3.7) {$\scriptstyle{P_1}$};
\node at (-5.3,-2.3) {$\scriptstyle{P_4}$};
\node at (-3.3,-2.3) {$\scriptstyle{P_5}$};
\node at (1.3,-2.3) {$\scriptstyle{P_6}$};
\draw[thick] (-6.5,1) to (-3,1) to (1,5) to (1.5,6);
\draw[thick] (-5,-2.5) to (-5,2);
\draw[thick] (-6.5,2) to (-5,2) to (-3.5,3.5) to (1.5,6) to (2.5,7);
\draw[thick] (-3,-2.5) to (-3,1);
\draw[thick] (-6.5,3.5) to (-3.5,3.5);
\filldraw (1.5,6) circle (1mm);
\node at (1.25,6.2) {$\scriptsize{3}$};
\node at (2.1,6.9) {$\scriptsize{3}$};
\draw[thick] (1,-2.5) to (1,5);
\end{tikzpicture}
\end{center}
\end{figure}

The tropical curves come from the tropicalization of the domain curves. This process is insensitive to the singularities of the image curves, which is what is picked up by the stable map multiplicity of Corollary \ref{cor:ecompjac}.

One may move around the points $P_1,\dots,P_6$ as in the classical case and find reducible tropical curves. This may lead to an understanding of what their multiplicity should be.

\subsubsection{Generalization}

\label{ex:GPS}

More generally, one could consider the following setting of the tropical vertex \cite{GPS10}: start with a toric surface $X$, choose a prime toric divisor $D_{\rm out}$ and denote the other prime toric divisors by $D_1,\dots,D_n$. For a curve class $\beta$, choose an intersection profile ${\rm \mathbf{P}}=({\rm \mathbf{P}}_1,\dots,{\rm \mathbf{P}}_n)$ for ordered partitions ${\rm \mathbf{P}}_i=p_{i1}+\cdots p_{il_i}$ and $|{\rm \mathbf{P}}_i|=\beta\cdot D_i$. We further choose distinct points $x_{i1},\dots,x_{il_i}\in D_i\setminus\cup_{i\neq j}D_j$. We blow up the $x_{ij}$ leading to the surface $\nu : \widetilde{X} \to X$ with exceptional divisors $E_{ij}$. Choose as anticanonical curve $\widetilde{D}$ the strict transform of the toric boundary of $X$. Then the curve class $\widetilde{\beta}=\nu^*(\beta)-\sum_{i=1}^{n}\sum_{j=1}^{l_i}p_{ij}E_{ij}$ on $\widetilde{X}$ meets $\widetilde{D}$ only in smooth points of $D_{\rm out}$.

The invariants $\mathcal{N}_{\widetilde{\beta}}(\widetilde{X},\widetilde{D})$ can be computed by the scattering diagrams/tropical methods of \cite{GPS10}. Proposition 5.3 in \cite{GPS10} expresses $\mathcal{N}_{\widetilde{\beta}}(\widetilde{X},\widetilde{D})$ in terms of invariants of $X$, maximally tangent to $D_{\rm out}$ and with incidence conditions along $D_1,\dots,D_n$. The latter can be computed by finding their associated tropical curves and tropical multiplicities as in \cite{ManRu}.
In fact, by \cite[Proposition 4.3]{GPS10}, for generic choices of $x_{ij}$, the higher-dimensional components of the moduli space only consist of multiple covers whose contributions are computed by \cite[Proposition 6.1]{GPS10}. Then Corollary \ref{cor:ecompjac} gives the contribution of the remaining zero-dimensional components. If the $x_{ij}$ are not generic, there can be more complicated contributions such as ${\rm Contr}^{(S,D)}(1,1,1)$ as in Section \ref{sec:ex3comppre}.


It is informative to compare with the invariants of maximal tangency with each boundary component as studied in \cite{BBvG1,BBvG2,BBvG3,NR}, for any cluster variety. To obtain a problem of virtual dimension 0, we need some insertions. If these are point insertions with psi classes, then \cite[Proposition 6.1]{Man19} guarantees that the invariants are enumerative for generic choices of points. 
The case of log K3 surfaces leads to more complicated components of the moduli space and we turn to it now.

\subsection{Applications of Theorem \ref{theorem_main} and future directions}
\label{subsection:applications}

The components that occur in $\overline{\text{M}}_\beta(X,D)$ are classified by \cite[Corollary 2.10]{CGKT2} (see Proposition \ref{cor_description_log_maps}). Outside of the calculations of this paper, the only other components whose contribution to log Gromov-Witten invariants is known are multiple covers over rigid maximally tangent rational curves \cite[Proposition 6.10]{GPS10}. 

Knowledge of the contributions of some components would allow for new enumerative calculations. As an illustration, consider $(X,D)=(\PP^2,E)$ for $E$ an elliptic curve, and the log BPS numbers $m^P_{5H}$ of degree $5H\in\hhh_2(\PP^2,\ZZ)$ of Definition \ref{defn:logbps} for $P\in D(5H)$. According to Conjecture \ref{conj:logbps} (proven in \cite{Bou19a,Bou19b}), $m^P_{5H}$ is constant whether $P$ is a flex point or a $5H$-primitive point. For the latter, $m^P_{5H}=113$ is an actual count of rational curves with multiplicities given by Corollary \ref{cor:ecompjac}.

Take now $P$ to be a flex point. Denote by $k_5$ the number of degree 5 rational curves maximally tangent to $D$ at $P$. Cf.\ \cite[Section 6]{CGKT2} and \cite{Tak96}, provided $D$ is general, in lower degree there are 1 flex line, 2 nodal cubics and 8 nodal quartics. Taking the description of \cite[Section 6]{CGKT2}, we have that
\begin{align*}
m^P_{5H}=113&=\mathrm{DT}^{(2)}_5+8\cdot 3 \cdot \min\{4,1\} + 2 \cdot \mathrm{Contr}^{(\PP^2,E)}(3,1^2) + k_5 \\
&= 5 + 24 + 2 \cdot \mathrm{Contr}^{(\PP^2,E)}(3,1^2) + k_5.
\end{align*}
Here $\rm{DT}^{(2)}_5$ is the 5th $2$-loop quiver invariant, which is the contribution of $5:1$ multiple covers over the flex line to $m^P_{5H}$ \cite[Proposition 6.4]{CGKT2}. The term $3 \cdot \min\{4,1\}$ is the contribution according to Theorem \ref{theorem_main} of the 3 stable log maps with image $\text{(fixed quartic)} \cup \text{(flex line)}$. $\rm{Contr}^{(\PP^2,E)}(3,1^2)$ is the unknown contribution of the component whose general points correspond to stable log maps $\PP^1\cup\PP^1\cup\PP^1\to \PP^2$ with the central component collapsed, the first component an immersion into one of the nodal cubics and the third component a $2:1$ cover over the flex line. 

So knowing $\rm{Contr}^{(\PP^2,E)}(3,1^2)$ one would be able to calculate $k_5$. In fact, local BPS numbers are calculated through local mirror symmetry \cite{CKYZ99} and so are the $m^P_{dH}$ via Conjecture \ref{conj:logbps} (proven in \cite{Bou19a,Bou19b}). Hence, knowing the contributions of each component of $\overline{\text{M}}_\beta(X,D)$ corresponding to stable log maps with reducible image would recursively allow to calculate $k_d$, the number of rational degree $d$ curves maximally tangent to a flex point, for all $d$. Moreover, the same analysis holds for counts of maximally tangent rational curves at any other point $P\in D(dH)$.

For a more in depth analysis of the above situation we refer to \cite{BN20}.

\subsubsection{The case of 3 components}
\label{sec:ex3comp}

We saw in Section \ref{sec:ex3comppre} that for the surface $(S,D)$, ${\rm Contr}^{(S,D)}(1,1,1)=3$. This may give insight as how to generalize Theorem \ref{theorem_main} to more complicated components. We leave this to future work.

\addtocontents{toc}{\protect\setcounter{tocdepth}{1}}

\section{Nonsingularity of the relative compactified Picard scheme} \label{sec:nonsing}

We start with a couple of lemmas that might belong to common knowledge.

\begin{Lemma}\label{lem_hilb}
If $X$ is a regular surface, 
a connected component of $\mathrm{Hilb}(X)$ containing a curve $C$ 
is nonsingular and coincides with $|C|$. 
\end{Lemma}
\begin{proof}
Let $\Lambda$ be a complete linear system and $C$ a member of $\Lambda$. 
The first order deformations of $C$ in $\mathrm{Hilb}(X)$ are given by 
\[
\mathrm{Hom}_\cOX(\mathcal{I}_C,\mathcal{O}_C)\simeq 
\mathrm{Hom}_\cOX(\mathcal{O}_X(-C),\mathcal{O}_C)\simeq H^0(\mathcal{O}_C(C)).
\]
From the short exact sequence 
$0\to\mathcal{O}_X\to \mathcal{O}_X(C)\to \mathcal{O}_C(C)\to 0$ 
we obtain the exact sequence 
\[
0\to H^0(\mathcal{O}_X)\to H^0(\mathcal{O}_X(C))
\to H^0(\mathcal{O}_C(C)) \to H^1(\mathcal{O}_X)=0. 
\]
The space $H^0(\mathcal{O}_X(C))/H^0(\mathcal{O}_X)$ 
can be regarded as the tangent space of $\Lambda$ at $C$, 
and this exact sequence shows that the natural map 
$T_C\Lambda\to T_C\mathrm{Hilb}(X)$ is an isomorphism. 
Since $\Lambda$ is projective and nonsingular (by definition) 
and is embedded into $\mathrm{Hilb}(X)$, 
it can be identified with a connected component of $\mathrm{Hilb}(X)$. 
\end{proof}

For a curve class $\beta$, 
let $M_\beta(X)$ denote 
the moduli space of stable $1$-dimensional sheaves of class $\beta$ on $X$ 
with respect to a certain polarization. 
In the following, we will mainly consider sheaves $F$ 
that can be regarded as torsion-free sheaves of rank $1$ 
on integral curves of class $\beta$. 
Such a sheaf $F$ defines a point $[F]\in M_\beta(X)$ for any polarization. 

\begin{Lemma}
Let $X$ be a smooth surface, $\beta$ a curve class with $K_X.\beta<0$ 
and $C$ an integral curve of class $\beta$. 
If $F$ is a torsion-free rank 1 sheaf on $C$, 
then $M_\beta(X)$ is nonsingular of dimension $\beta^2+1$ 
at $[F]$. 
\end{Lemma}
\begin{proof}
The first order deformations of $F$ in $M_\beta(X)$ are described 
by $\mathrm{Ext}_\cOX^1(F, F)$. 
We have $\mathrm{ch}_0(F)=0$ 
and by the Riemann-Roch theorem 
one calculates
$$\sum_{i=0}^2 (-1)^i\dim \mathrm{Ext}_\cOX^i(F, F)
=-c_1(F)^2+2\mathrm{ch}_0(F)\mathrm{ch}_2(F)+\mathrm{ch}_0(F)^2\chi(\cOX)
=-c_1(F)^2=-\beta^2.$$ 
By stability of $F$ (or, rather, by the arguments for the proof of stability),
$\mathrm{Hom}_\cOX(F, F)=\mathbb{C}$.
Moreover, $\mathrm{Ext}_\cOX^2(F, F)$ is dual to 
$\mathrm{Hom}_\cOX(F\otimes_\cOX \cOX(-K_X), F)$, 
and the latter is $0$ by the inequality 
$\chi(F\otimes_\cOX \cOX(-K_X))-\chi(F)=-K_X.C>0$ 
and the stability of $F$. 
\end{proof}

\begin{Lemma}\label{lem_mmi_and_jac}
Let $C$ be an integral curve on a smooth surface $X$ 
and $F$ a torsion-free sheaf of rank 1 on $C$. 
Then locally near $[F]$, the Chow morphism $M_\beta(X)\to \mathrm{Hilb}(X)$ lifts to an isomorphism of a neighborhood of $[F]$ in
$M_\beta(X)$ to an open set in $\overline{\mathrm{Pic}}({\mathcal{C}/\mathcal{H}})$, 
where $\mathcal{H}$ is the connected component of $\mathrm{Hilb}(X)$ 
containing $[C]$ and $\mathcal{C}$ is the universal subscheme over $\mathcal{H}$. 
\end{Lemma}
\begin{proof}
First, take a deformation of $F$ in $M_\beta(X)$: 
Let $T$ be a scheme over $\mathbb{C}$, $0\in T$ a point, 
and $\mathcal{F}$ a coherent sheaf on $X\times T$ 
which is flat over $T$, 
such that $\mathcal{F}_0\cong F$ 
and $\mathcal{F}_t$ is stable for any geometric point $t$ of $T$. 
We may replace $T$ by a neighborhood of $0$ 
(actually, it suffices to take $T$ to be a neighborhood of $[F]$ in $M_\beta(X)$), 
and we have the Chow morphism $\varphi: T\to \mathcal{H}$. 

Let us show that, after shrinking $T$ if necessary, 
the ideal of $\mathcal{C}_T$ in $X\times T$ 
annihilates $\mathcal{F}$.  For this purpose, we recall the definition of the Fitting ideal of $\mathcal{F}$.
Let $R$ denote the local ring of $X\times T$ at a point over $0$ 
and $M$ an $R$-module corresponding to $\mathcal{F}$. 
Since $F$ is pure of dimension $1$ and $\mathcal{F}$ is flat over $T$, by \cite[Proposition 1.1.10]{HL}, $M$ has a two-step resolution 
\[
0\to R^n\overset{\phi}{\to} R^n \to M\to 0. 
\]
The Fitting ideal is locally generated by $\det\phi$, and is globally well-defined, independent of the local resolution of $\mathcal{F}$.   It is immediate to see that the Fitting ideal defines the flat family of subschemes $\mathcal{C}_T$ corresponding to the Chow morphism $T\to \mathrm{Hilb}(X)$.  Now $\det \phi$ certainly annihilates $M$, 
since on $R^n$ it can be written as the composition of $\phi$ and its adjoint. 
Hence $\mathcal{F}$ can be regarded as a family of sheaves on $\mathcal{C}_T$. 

Note that $\mathcal{F}_t$ is a torsion-free sheaf on $(\mathcal{C}_T)_t$ 
for any geometric point $t$ of $T$, since a torsion subsheaf would destabilize $\mathcal{F}_t$. 
By shrinking $T$, we may assume that $(\mathcal{C}_T)_t$ is integral for any $t$, 
and then $\mathcal{F}_t$ is of rank $1$ since its first Chern class is $\beta$. 
We therefore obtain a morphism 
$T\to \overline{\mathrm{Pic}}({\mathcal{C}/\mathcal{H}})$. 

Conversely, if $\mathcal{G}$ is a family of rank $1$ torsion-free modules 
on a family of integral curves $\mathcal{C}_T$ for some $T\to \mathcal{H}$, 
then it can be considered as a family of stable sheaves 
on $X$ over $T$. 

These correspondences are inverse to each other, 
and isomorphisms between families also coincide. 
Thus we have a local isomorphism of the moduli spaces. 
\end{proof}

We will use the following theorem on relative compactified Picard schemes. 
\begin{Theorem}[\cite{AIK77}]
\label{thm_aik}
Let $\mathcal{C}/S$ be a projective family of integral curves of arithmetic genus $p_a$ 
that can be embedded into a smooth projective family of surfaces over $S$. 
Then its relative compactified Picard scheme 
is flat over $S$ and the geometric fibers are 
integral locally complete intersections of dimension $p_a$. 
\end{Theorem}

We return to the setting of a smooth surface $X$ and 
a curve $D$ on $X$. 
For a curve class $\beta$, 
let us denote $D.\beta$ by $w$, 
which we assume to be positive. 
For a scheme $T$ over $\mathbb{C}$, 
we consider the following condition (*)
on a coherent sheaf $\mathcal{F}$ on $X\times T$: 
\begin{itemize}
\item[(a)]
$\mathcal{F}$ is flat over $T$, 
and for each geometric point $t$ of $T$, 
$\mathcal{F}_t$ is a torsion-free sheaf of rank $1$ 
on an integral curve $C_t$ of class $\beta$, 
not contained in $D$. 
\item[(b)]
There exists a section $\sigma: T\to D_\mathrm{sm}\times T\hookrightarrow X\times T$ 
with $\mathcal{F}|_{D\times T}\cong \shO_{w\cdot\sigma(T)}$ as $\cO_{D\times T}$-modules, 
where $w\cdot\sigma(T)$ is the closed subscheme of $D\times T$ 
defined by the $w$-th power of the ideal sheaf of $\sigma(T)\subset D\times T$. 
\end{itemize}
We will later see that $\sigma$ is unique. 

\begin{Lemma}
For a sheaf $\mathcal{F}$ satisfying condition (*), 
the following also holds: 
\begin{itemize}
\item
In a neighborhood of each point of $\sigma(T)$, the sheaf
$\mathcal{F}$ is isomorphic to 
the structure sheaf of the family of curves induced by the Chow morphism. 
\end{itemize}
Also, the conditions $\mathrm{rank}_{C_t}\mathcal{F}_t=1$ 
and $C_t\not\subseteq D$ 
follow from the rest of the conditions. 
\end{Lemma}
\begin{proof}
From $\mathcal{F}|_{D\times T}\cong \shO_{w\cdot\sigma(T)}$ and Nakayama's Lemma, 
$\mathcal{F}$ is generated by $1$ element near any point of $\sigma(T)$, 
giving rise to a surjective homomorphism $\shO_{X\times T}\to \mathcal{F}$ locally. 
The kernel contains the ideal of the associated family of curves, 
and from the torsion-freeness, they coincide. 
\end{proof}

\begin{Definition}\label{defn:mmi}

(1)
We define a moduli functor $\mathcal{MMI}_\beta$ 
(for ``modules with maximal intersection'')
on the category of schemes over $\mathbb{C}$ 
as the sheafification of the presheaf 
\[
T\mapsto \{ \mathcal{F} \mid 
\hbox{$\mathcal{F}$ is a sheaf on $X\times T$ 
satisfying the condition (*)}\} / \cong,
\]
where $\cong$ denotes isomorphisms of coherent sheaves on $X\times T$. 

(2)
For $P\in D_\mathrm{sm}$, 
we define $\mathcal{MMI}_\beta^P$ 
as the subfunctor of $\mathcal{MMI}_\beta$ 
parameterizing families where $\sigma$ can be locally taken to be the constant section 
$T\cong P\times T\hookrightarrow D_\mathrm{sm}\times T$. 

(3)
We define a moduli functor $\mathcal{MI}_\beta$ of integral curves on $X$ of class $\beta$ 
with maximal intersection with $D$: 
For a scheme $T$ over $\mathbb{C}$, 
an element of $\mathcal{MI}_\beta(T)$ is a closed subscheme $\mathcal{Z}$ of $X\times T$, 
flat over $T$, 
with fibers integral curves of class $\beta$ 
such that the intersection of $\mathcal{Z}$ and $D\times T$ 
is $w\cdot\sigma(T)$ for a section $\sigma: T\to D_\mathrm{sm}\times T$. 

(4)
For $P\in D_\mathrm{sm}$, 
we define $\mathcal{MI}_\beta^P$ as the subfunctor of $\mathcal{MI}_\beta$ 
parameterizing families where $\sigma$ can be taken to be the constant section with value $P$. 
\end{Definition}

\begin{Lemma}\label{lem_mmi_mi}
(1)
The functor $\mathcal{MMI}_\beta$ is represented 
by a locally closed subscheme of $M_\beta(X)$, 
and $\mathcal{MMI}_\beta^P$ is represented by a closed subscheme 
of $\mathcal{MMI}_\beta$. 

(2)
The functor $\mathcal{MI}_\beta$ is represented 
by a locally closed subscheme of $\mathrm{Hilb}(X)$, 
and $\mathcal{MI}_\beta^P$ is represented by a closed subscheme 
of $\mathcal{MI}_\beta$. 

(3)
We may also regard $\mathcal{MMI}_\beta$ (resp. $\mathcal{MMI}_\beta^P$) 
as an open subscheme 
of $M_\beta(X)\times_{\mathrm{Hilb}(X)} \mathcal{MI}_\beta$ 
(resp. $M_\beta(X)\times_{\mathrm{Hilb}(X)} \mathcal{MI}_\beta^P$), 
or of the relative compactified Picard scheme over $\mathcal{MI}_\beta$ 
(resp. $\mathcal{MI}_\beta^P$). 

(4)
There exist unique morphisms $\mathcal{MMI}_\beta\to D$ and 
$\mathcal{MI}_\beta\to D$ 
representing sections $\sigma$ 
such that $\mathcal{F}|_{D\times T}\cong\cO_{w\cdot \sigma(T)}$ locally over $T$  
and $\mathcal{Z}|_{D\times T}=w\cdot\sigma(T)$. 
These morphisms commute with the Chow morphism. 

The spaces 
$\mathcal{MMI}_\beta^P$ and $\mathcal{MI}_\beta^P$ 
are the scheme theoretic inverse images 
of $P$ by these morphisms. 
\end{Lemma}
\begin{proof}
First we prove 
(1) and (4) for $\mathcal{MMI}_\beta$. 
Proofs of (2) and (4) for $\mathcal{MI}_\beta$ are similar. 

In $M_\beta(X)$, the condition (a) in (*) is an open condition. 
The condition that $\mathcal{F}$ is generated by one section near $D$ 
is also open, so these conditions define an open subscheme $M^\circ$ of $M_\beta(X)$. 

Let $\mathcal{F}$ be a family in $M^\circ$ over $T$. 
In a neighborhood of $\sigma(T)$, 
$\mathcal{F}$ is isomorphic to the structure sheaf of a flat family of curves, 
which is a family of principal divisors. 
Hence, locally over $T$, the restriction $\mathcal{F}|_{D\times T}$ is isomorphic to 
the structure sheaf of a flat family 
of $0$-dimensional subschemes of $D$ of length $w$. 
This family can also be described as the one 
defined by the annihilator of $\mathcal{F}|_{D\times T}$, 
hence is determined by $\mathcal{F}$. 

Thus we have a morphism $ M^\circ\to \mathrm{Hilb}^w(D)$, 
the latter being isomorphic to the symmetric $w$-th power $D^{(w)}$ of $D$. 
We claim that subschemes of the form $wP$ are represented by 
the diagonal set $\Delta\subset D^{(w)}$ with the reduced induced structure. 
To show this, we can work with the formal neighborhood of $P$ 
since we are concerned with infinitesimal deformations of $0$-dimensional subschemes. 
In a formal coordinate $x$ on $D$, 
the Hilbert scheme can be described as 
the spectrum of $\mathbb{C}[[a_1, \dots, a_w]]$ 
with the universal subscheme $x^w-a_1x^{w-1}+\dots+(-1)^w a_w=0$. 
The diagonal set $\Delta$, as a reduced closed subscheme,  
is given by $a_i=\begin{pmatrix} w \\ i \end{pmatrix} (a_1/w)^i$ 
($i=2, \dots, w$). 
Consider a family over a complete local ring $R$, 
corresponding to 
$\varphi: \mathrm{Spec}\ R\to \mathrm{Spec}\ \mathbb{C}[[a_1, \dots, a_w]]$ given by 
$a_i=r_i$. 
If it satisfies the condition (b), 
with $\sigma$ corresponding to $x\mapsto r\in R$, 
then $x^w-r_1x^{w-1}+\dots+(-1)^w r_w$ 
is equal to $(x-r)^w$ 
(this follows from the fact that the coefficients of $x^0, x^1, \dots, x^{w-1}$ are 
the coordinates of the representing space, 
or more concretely, by writing $x^w-r_1x^{w-1}+\dots+(-1)^w r_w=(\hbox{unit})(x-r)^w$ 
and using Weierstrass preparation theorem). 
This means that $r_i=\begin{pmatrix} w \\ i \end{pmatrix} r^i$ ($i=1, \dots, w$) 
and $\varphi$ factors through $\Delta$. 

Thus $\mathcal{MMI}_\beta$ is the scheme theoretic inverse image of $\Delta$. 

In the calculation above, $r$ is determined by $(r_1, \dots, r_w)$: 
Specifically, $r=r_1/w$. 
This shows the existence and uniqueness 
of $\mathcal{MMI}_\beta\to D$ as in (4), 
and $\mathcal{MMI}_\beta^P=\mathcal{MMI}_\beta\times_D P$ follows from the definition 
of $\mathcal{MMI}_\beta^P$. 

(3) 
follows from the description above of families in $\mathcal{MMI}_\beta$ 
as families in $M^\circ$ whose support curves have maximal intersection with $D$. 
\end{proof}

\begin{Lemma}\label{lem_mi_is_reduced}
Assume that $X$ is a regular surface. 
\begin{enumerate}
\item
The space $\mathcal{MI}_\beta^P$ can be identified 
with $\shooo$ considered as a nonsingular variety. 
\item
Assume furthermore that 
the Abel map of $D$ is immersive at $P$ 
and $C$ belongs to $\mathcal{MI}_\beta^P$. 
Then, in a neighborhood of $[C]$, 
the morphism $\mathcal{MI}_\beta\to D$ representing the intersection point 
is scheme-theoretically the constant map with value $P$. 

Consequently, 
if the Abel map of $D$ is immersive at each $P\in D(\beta)$, 
\[
\mathcal{MI}_\beta=\coprod_{P\in D(\beta)}\mathcal{MI}_\beta^P
\]
and
\[
\mathcal{MMI}_\beta=\coprod_{P\in D(\beta)}\mathcal{MMI}_\beta^P 
\]
scheme theoretically, and 
$\mathcal{MI}_\beta$ can be identified with $\coprod_{P\in D(\beta)} \shooo$ 
and $\mathcal{MMI}_\beta$ can be considered as an open subscheme 
of the relative compactified Picard scheme over 
$\coprod_{P\in D(\beta)} \shooo$. 
\item
If $h^0(\cO_D)=1$ and 
$P\in D_\mathrm{sm}$, 
the following are equivalent: 
\begin{enumerate}
\item[(a)]
The Abel map of $D$ is immersive at $P$. 
\item[(b)]
$h^0(D, \cO_D(P))=1$. 
\item[(c)]
$\omega_D$ has a global section nonzero at $P$. 
\end{enumerate}
In particular, the Abel map of $D$ is immersive at $P$ 
if either $h^0(\cO_D)=1$ and the component $D_0$ of $D$ containing $P$ 
satisfies $p_a(D_0)>0$, 
or $K_X+D\sim 0$. 
\end{enumerate}
\end{Lemma}
If $D$ is connected and reduced, we can show that 
the conditions of (3) are equivalent to saying 
that the $D_0$ is not a loosely connected rational tail of $D$
(cf. Step II of the proof of \cite[Theorem D]{Catanese1982}). 
For more about the immersivity of the Abel map of reduced Gorenstein curves we refer to \cite{Catanese1982} and \cite{CCE08}.
\begin{proof}
(1)
Set-theoretically, this is obvious. 
Let $[C]$ be a point of $\mathcal{MI}_\beta^P$. 
By Lemma \ref{lem_hilb}, 
the component of $\mathrm{Hilb}(X)$ containing $C$ 
can be identified with $|C|$. 
Taking a basis $\varphi_0, \dots, \varphi_d$ of $H^0(\cOX(C))$ with $\varphi_0$ 
corresponding to $C$, 
$|C|$ has natural local coordinates $(s_1, \dots, s_d)$ near the point $[C]$ and 
$\mathcal{MI}_\beta^P$ is defined by the vanishing 
of $\varphi_0+s_1\varphi_1+\dots+s_d\varphi_d$ on $wP$ as a section of $\cO_{wP}(C)$. 
This gives linear equations on $s_1, \dots, s_d$, 
and $\mathcal{MI}_\beta^P$ can be scheme-theoretically identified with $\shooo$. 

(2)
Let $[C]$ be a point of $\mathcal{MI}_\beta$ with $P=C\cap D$ and 
$\mathcal{Z}$ the family corresponding to a small neighborhood $T$ of $[C]$. 
By Lemma \ref{lem_hilb}, 
there is a morphism $T\to \PP(H^0(\mathcal{O}_X(C)))$ 
for which $\mathcal{Z}$ is the pullback of the universal curve $\mathcal{C}$. 
Taking the pullback of the universal curve 
by $U:=H^0(\mathcal{O}_X(C))\setminus \{0\}\to \PP(H^0(\mathcal{O}_X(C)))$, 
we have a universal section of $\mathcal{O}_{X\times U}(C\times U)$ defining the family of curves $\mathcal{C}_U$ 
and hence an isomorphism 
$\mathcal{O}_{X\times U}(\mathcal{C}_U-(C\times U))\cong\mathcal{O}_{X\times U}$. 
By taking the pullback by a local lift $T\to U$ and restricting, 
we have 
$\mathcal{O}_{D\times T}(w\cdot\sigma(T)-w(P\times T))\cong \mathcal{O}_{D\times T}$, 
where $\sigma$ is as in (b) of (*). 

Let $u: D_\mathrm{sm}\to \mathrm{Pic}(D)$ 
be the ``Abel morphism'', defined roughly by $Q\mapsto \cO_D(Q)$, 
and $[w]: \mathrm{Pic}(D)\to \mathrm{Pic}(D)$ the multiplication-by-$w$ morphism. 
Then the above isomorphism shows that 
$[w]\circ u\circ\sigma$ is a constant map, 
where we regard $\sigma$ as a morphism $T\to D$. 

From the assumption that $u$ is immersive at $P$ and the \'etaleness of $[w]$, 
we see that $\sigma$ is the constant map with value $P$. 
Thus the family $\mathcal{Z}\to T$ belongs to $\mathcal{MI}_\beta^P$. 
Since $\mathcal{MMI}_\beta\to D$ factors through $\mathcal{MI}_\beta$, 
the assertion on $\mathcal{MMI}_\beta$ also holds. 

(3)
The exact sequence $0\to \cO_D\to\cO_D(P)\to\cO_P(P)\to 0$ induces 
$$0\to H^0(\cO_D)\overset{f}{\to} H^0(\cO_D(P))\to T_PD
\overset{d_P u}{\to} T_{[\cO_D(P)]}\mathrm{Pic}(D),$$
and so (a) and (b) are equivalent. 

From the exact sequence $0\to \omega_D(-P)\to \omega_D\to \omega_D|_P\to 0$ 
we have a long exact sequence 
$$H^0(\omega_D)\to H^0(\omega_D|_P)\to 
H^1(\omega_D(-P))\overset{g}{\to} H^1(\omega_D),$$ 
and $g$ is the Serre dual to $f$. 
Thus (c) is also equivalent. 

If $p_a(D_0)>0$, 
then $|P|$ consists of one point since otherwise it would give an isomorphism $D\cong \mathbb{P}^1$. 
Thus we have $H^0(D_0, \cO_{D_0}(P))=\CC=H^0(D_0, \cO_{D_0})$, 
hence (b) holds. 

If $K_X+D\sim 0$, then (c) is obvious. 
For an anticanonical curve $D(\not=0)$ on a rational surface $X$ 
the fact that $h^0(\cO_D)=1$ (and $h^1(\cO_D)=1$) is standard: 
This follows from the long exact sequence associated to 
$0\to \cOX(-D)\cong\cOX(K_X)\to \cOX\to \cO_D\to 0$, 
using Serre duality. 
\end{proof}

Recall that, if a class in $\mathrm{Ext}_\cOX^1(F, F)$ is 
represented by an extension $0\to F\overset{\alpha}{\to} \tilde{F}\overset{\beta}{\to} F\to 0$, 
then the corresponding deformation over $\mathbb{C}[\varepsilon]/\varepsilon^2$ 
is given by $\tilde{F}$ with the action of $\varepsilon$ on $\tilde{F}$ 
defined as $\alpha\circ\beta$.

\begin{Lemma}\label{lem_local_def}
Let $P\in D_\mathrm{sm}$ be a point, 
$x$ a local parameter on $D$ at $P$, 
and $F= \mathcal{O}_{D, P}/(x^w)$. 
Then a first order deformation of $F$ as a coherent sheaf on $D$ is given by 
($\mathcal{O}_{D, P}\otimes \mathbb{C}[\varepsilon])/(x^w-g(x)\varepsilon, \varepsilon^2)$ 
for a unique polynomial $g\in\mathbb{C}[X]$ of degree $\leq w-1$. 
The corresponding extension is isomorphic to 
\[
0\to 
\cO_{D, P}/(x^w) 
\overset{\alpha}{\to}
(\cO_{D, P}\otimes \mathbb{C}[\varepsilon])/(x^w-g(x)\varepsilon, \varepsilon^2)
\overset{\beta}{\to}
\cO_{D, P}/(x^w) 
\to 0, 
\]
where $\alpha(\bar{f})=\overline{f\varepsilon}$ 
and $\beta$ is the reduction modulo $\varepsilon$. 
\end{Lemma}
\begin{proof}
As in the proof of Lemma \ref{lem_mmi_mi}, 
a small deformation of $F$ is equivalent to the deformation of 
the supporting scheme, 
and we may replace $\cO_{D, P}$ by $\mathbb{C}[[x]]$. 
Then the assertion follows from the description of Hilbert schemes of points on a smooth curve 
as symmetric powers. 
\end{proof}

\begin{Remark}
In the following, we describe the tangent spaces of $\mc{MMI}_\beta$ and $\mc{MMI}_\beta^P$ 
as the images of the natural maps 
$\mathrm{Ext}_\cOX^1(F, F(-(w-1)P))\to \mathrm{Ext}_\cOX^1(F, F)$ 
and $\mathrm{Ext}_\cOX^1(F, F(-D))\to \mathrm{Ext}_\cOX^1(F, F)$, respectively. 
In the former, we allow the intersection point to move along $D$. 
In the case of our main concern, the intersection point does not move in $D$ 
by Lemma \ref{lem_mi_is_reduced}, 
and the tangent spaces coincide. 
We use $\mathrm{Ext}_\cOX^1(F, F(-D))$ to prove our main result here, 
but a similar proof using $\mathrm{Ext}_\cOX^1(F, F(-(w-1)P))$ 
is also possible. 

\end{Remark}

\begin{Lemma}
Let $[F]$ be a point of $\mathcal{MMI}_\beta^P$ 
for $P\in D_\mathrm{sm}$. 

\begin{enumerate}
\item The tangent space of $\mathcal{MMI}_\beta^P$ at $[F]$ 
is naturally isomorphic to the image of 
the natural map $\mathrm{Ext}_\cOX^1(F, F(-D))\to \mathrm{Ext}_\cOX^1(F, F)$. 
\item
Define $F(-(w-1)P):=\mathrm{Ker} (F\to F|_{(w-1)P})$, 
where $(w-1)P$ is the closed subscheme of $D$ 
defined by $(\mathcal{I}_{P\subset D})^{w-1}$. 

Then the tangent space of $\mathcal{MMI}_\beta$ at $[F]$ 
is naturally isomorphic to the image of 
the natural map $$\mathrm{Ext}_\cOX^1(F, F(-(w-1)P))\to \mathrm{Ext}_\cOX^1(F, F).$$
Note that, if the supporting curve $C$ of $F$ is smooth at $P$, 
$F(-(w-1)P)$ can also be described 
as $F\otimes_{\mathcal{O}_C}\mathcal{O}_C(-(w-1)P)$. 
\end{enumerate}
\end{Lemma}

\begin{proof}
We begin by observing that $\mathrm{Tor}_1^{\mathcal{O}_X}(F,\mathcal{O}_D)=0$.  This follows from the exact sequence
\[
0\to \mathrm{Tor}_1^{\mathcal{O}_X}(F,\mathcal{O}_D)\to F(-D)\to F,
\]
since the final map is immediately seen to be injective by the local form of $F$ near $D$.  We will use this vanishing without comment in the remainder of the proof to conclude that  short exact sequences ending in $F$ remain exact after restriction to $D$.

(1)
Take a tangent vector of $\mathcal{MMI}_\beta^P$ at $[F]$ 
and let $0\to F\to \tilde{F}\to F\to 0$ 
be the corresponding extension. 

The restriction 
$0\to F|_D\to \tilde{F}|_D\to F|_D\to 0$ is a split extension, 
so let $\tilde{G}\subset \tilde{F}|_D$ be the image of a splitting. 
Then we have a commutative diagram with exact rows and columns: 
\[
\xymatrix{
 &  & 0 \ar[d] &  &  \\
 &  & \tilde{G} \ar[d]\ar[r]^\sim & F|_D \ar@{=}[d] & \\
0 \ar[r] & F|_D \ar@{=}[d]\ar[r] & \tilde{F}|_D \ar[d]\ar[r] & F|_D \ar[r] & 0 \\
 & F|_D \ar[r]^\sim & (\tilde{F}|_D)/\tilde{G} \ar[d] &  &  \\
 &  & 0. &  & 
}
\]
The sheaf $F(-D)$ is the kernel of $F\to F|_D$, 
so if we write $\tilde{F}'$ for the inverse image of $\tilde{G}$ in $\tilde{F}$, 
we have a commutative diagram with exact rows: 
\[
\xymatrix{
0 \ar[r] & F(-D) \ar[r]\ar[d] & \tilde{F}' \ar[r]\ar[d] & F \ar[r]\ar[d] & 0 \\
0 \ar[r] & F \ar[r] & \tilde{F} \ar[r] & F \ar[r] & 0, 
}
\]
where the top row is the kernel of the natural surjection from the bottom row to the bottow row of the preceding commutative diagram.
Thus our extension comes from a class in $\mathrm{Ext}_\cOX^1(F, F(-D))$. 

Conversely, if we are given an element of $\mathrm{Ext}_\cOX^1(F, F(-D))$, 
let 
$0 \to F(-D) \to \tilde{F}' \to F \to 0$ be 
the corresponding extension. 
By push-out, 
we obtain a sheaf $\tilde{F}$ and a commutative diagram with exact rows as above, 
where the lower row represents the induced class in $\mathrm{Ext}_\cOX^1(F, F)$. 
By restricting to $D$, we have a commutative diagram with exact rows: 
\[
\xymatrix{
0 \ar[r] & F(-D)|_D \ar[r]\ar[d]_i & \tilde{F}'|_D \ar[r]\ar[d]_j & F|_D \ar[r]\ar[d] & 0 \\
0 \ar[r] & F|_D \ar[r] & \tilde{F}|_D \ar[r] & F|_D \ar[r] & 0. 
}
\]
Here $i$ is $0$, and therefore 
the induced map $\mathrm{Im}(j)\to F|_D$ is an isomorphism. 
Thus the lower row is split, and $\tilde{F}$ gives a tangent vector to $\mc{MMI}_\beta^P$.

(2)
Let $0\to F\to \tilde{F}\to F\to 0$ 
be an extension corresponding to a tangent vector of $\mathcal{MMI}_\beta$ at $[F]$. 
By restriction, we have $0\to F|_D\to \tilde{F}|_D\to F|_D\to 0$ 
satisfying $F|_D\cong \mathcal{O}_{wP}$. 

We take a local parameter $x$ of $D$ at $P$, 
and identify finite-length modules over $\mathcal{O}_{D, P}$ with 
those on $\mathbb{C}[[x]]$. 

The section $\sigma$ in the definition of $\mc{MMI}_\beta$ 
can be written as $x=a\varepsilon$, $a\in\mathbb{C}$. 
Then, with $c=wa$, we have  
$\tilde{F}|_D\cong \mathbb{C}[[x, \varepsilon]]/(x^w-cx^{w-1}\varepsilon)$. 
Here the map $F|_D\to \tilde{F}|_D$ is the map induced from 
the map $e: \tilde{F}|_D\to \tilde{F}|_D; s\mapsto \varepsilon s$, 
where $(\tilde{F}|_D)/\mathrm{ker}(e)$ is identified with $F|_D$. 
Explicitly, it is given by 
$f(x) \mod (x^w)\mapsto  f(x)\varepsilon\mod (x^w-cx^{w-1}\varepsilon, \varepsilon^2)$. 

Write $G:=\mathcal{I}_P^{w-1}\cdot (F|_D)\cong (x^{w-1})/(x^w)$ 
and let $\tilde{G}$ be the $\mathbb{C}$-subspace of 
$\mathbb{C}[[x, \varepsilon]]/(x^w-cx^{w-1}\varepsilon, \varepsilon^2)$
with a basis 
$\overline{1}, \overline{x}, \dots, \overline{x^{w-1}}, \overline{x^{w-1}\varepsilon}$. 
Then the latter is also a $\mathbb{C}[[x]]$-submodule, 
since $x\cdot \overline{x^{w-1}}= \overline{x^w}=c\overline{x^{w-1}\varepsilon}$ 
and 
$x\cdot \overline{x^{w-1}\varepsilon }=
\overline{x^w\varepsilon}= \overline{cx^{w-1}\varepsilon^2}=0$, 
and they fit in a commutative diagram with exact rows and columns: 
\[
\xymatrix{
 & 0 \ar[d] & 0 \ar[d] &  &  \\
0 \ar[r] & G \ar[r]\ar[d] & \tilde{G} \ar[r]\ar[d] & F|_D \ar[r]\ar@{=}[d] & 0 \\
0 \ar[r] & F|_D \ar[d]\ar[r] & \tilde{F}|_D \ar[d]\ar[r] & F|_D \ar[r] & 0 \\
 & F|_{(w-1)P} \ar[d]\ar[r]^\sim & (\tilde{F}|_D)/\tilde{G} \ar[d] &  &  \\
 & 0 & 0. &  & 
 }
\]
The inverse image of $G$ in $F$ is exactly $F(-(w-1)P)$. 
If we write $\tilde{F}'$ for the inverse image of $\tilde{G}$ in $\tilde{F}$, 
we have a commutative diagram with exact rows: 
\[
\xymatrix{
0 \ar[r] & F(-(w-1)P) \ar[r]\ar[d] & \tilde{F}' \ar[r]\ar[d] & F \ar[r]\ar[d] & 0 \\
0 \ar[r] & F \ar[r] & \tilde{F} \ar[r] & F \ar[r] & 0, 
}
\]
and $\tilde{F}$ comes from $\mathrm{Ext}_\cOX^1(F, F(-(w-1)P))$. 

Conversely, given an element of $\mathrm{Ext}_\cOX^1(F, F(-(w-1)P))$, 
let 
$0 \to F(-(w-1)P) \to \tilde{F}' \to F \to 0$ be 
the corresponding extension. 
We obtain a sheaf $\tilde{F}$ by push-out, 
and a commutative diagram with exact rows as above, 
the lower row representing the induced class in $\mathrm{Ext}_\cOX^1(F, F)$. 
Restriction to $D$ gives a commutative diagram with exact rows: 
\[
\xymatrix{
0 \ar[r] & F(-(w-1)P)|_D \ar[r]\ar[d]_i & \tilde{F}'|_D \ar[r]\ar[d]_j & F|_D \ar[r]\ar[d] & 0 \\
0 \ar[r] & F|_D \ar[r] & \tilde{F}|_D \ar[r] & F|_D \ar[r] & 0. 
}
\]
This induces an exact sequence 
$0 \to \mathrm{Im}(i)\to\mathrm{Im}(j)\to F|_D\to 0$, 
and we see by a local calculation 
that $\mathrm{Im}(i)$ is of length $1$ and is annihilated by $x$. 

Let us identify the lower row with the exact sequence in the Lemma~\ref{lem_local_def}. 
We have to show that $g(x)=cx^{w-1}$ for some $c\in\mathbb{C}$. 
Since $\mathrm{Im}(j)$ maps surjectively to $F|_D$, 
it contains an element of the form $\overline{1+h(x)\varepsilon}$, 
and we have 
$x^w\cdot\overline{1+h(x)\varepsilon}=\overline{g(x)\varepsilon}\in \mathrm{Im}(j)$. 
Since it is mapped to $0\in F|_D$, 
we have $\overline{g(x)\varepsilon}\in \mathrm{Im}(\alpha\circ i)$. 
By the remark in the previous paragraph, we have $\overline{xg(x)\varepsilon}=0$ 
in $\mathbb{C}[[x, \varepsilon]]/(x^w-g(x)\varepsilon)$. 
If we write $g(x)=\sum_{i=0}^{w-1} c_ix^i$, this amounts to 
$\sum_{i=0}^{w-2} c_i\overline{x^{i+1}\varepsilon }=0$. 
Since $\overline{1}, \overline{x}, \dots, \overline{x^{w-1}}, 
\overline{\varepsilon}, \overline{x\varepsilon }, \dots, \overline{x^{w-1}\varepsilon }$ 
form a $\mathbb{C}$-basis of 
$\mathbb{C}[[x, \varepsilon]]/(x^w-g(x)\varepsilon, \varepsilon^2)$, 
we have $c_0=\dots=c_{w-2}=0$, which shows our assertion. 
\end{proof}

\begin{Theorem}[=Theorem \ref{thm:nonsing}]\label{thm:smooth}
Let $X$ be a smooth projective rational surface, $D$ an anticanonical curve on $X$ 
and $P\in D_\mathrm{sm}$. 
Then $\mathcal{MMI}_\beta$ and $\mathcal{MMI}_\beta^P$ 
are nonsingular of dimension $2p_a(\beta)=\beta^2-w+2$. 

Consequently, 
the relative compactified Picard scheme over $\shooo$ 
is nonsingular at a point $[F]$ over $[C]$ 
if $F$ is an invertible $\cO_C$-module near $P$ 
(or, equivalently, $F|_D\cong \cO_C|_D$). 
\end{Theorem}
\begin{proof}
Let $[F]$ be a point of $\mc{MMI}_\beta^P$ with support curve $C$. 
Since $|\beta|$ contains $C$, 
which is integral by the definition of $\mathcal{MMI}_\beta^P$, 
the dimension of 
$\mathcal{MI}_\beta^P$ is $p_a(C)$ by the remark after \cite[Definition 4.16]{CGKT2}. 
(Note that \cite[Proposition 4.15]{CGKT2} holds 
for a rational surface $S$, 
$E\in |-K_S|$ possibly reducible or non-reduced, 
$\beta$ a curve class containing $C$ with $h^0(\cO_C)=1$ 
and $P\in E(\beta)$.)
By Theorem \ref{thm_aik} and Lemma \ref{lem_mmi_mi}(3), 
$\mathcal{MMI}_\beta^P$ is of dimension $2p_a(C)$ at $[F]$. 
Thus it suffices to show that the dimension of the tangent space at $[F]$ 
is $2p_a(C)$. 

To see this, we note that the tangent space is 
the image of $\mathrm{Ext}_\cOX^1(F, F(-D))\to \mathrm{Ext}_\cOX^1(F, F)$ 
by the previous lemma. 
Consider the natural exact sequence 
\begin{eqnarray*}
\mathrm{Hom}_\cOX(F, F(-D))\to \mathrm{Hom}_\cOX(F, F)
\to \mathrm{Hom}_\cOX(F, F|_{wP}) \\
\overset{\delta}{\to} \mathrm{Ext}_\cOX^1(F, F(-D))\to \mathrm{Ext}_\cOX^1(F, F). 
\end{eqnarray*}
The first term is $0$, the second term is $\mathbb{C}$, 
and the third is of dimension $w$. 
Thus the rank of $\delta$ is $w-1$, and 
we have only to show that $\mathrm{Ext}_\cOX^1(F, F(-D))$ 
has dimension $\beta^2+1$. 

Riemann-Roch theorem tells us that 
$\sum_{i=0}^2 (-1)^i\dim\mathrm{Ext}_\cOX^i(F, F(-D))=-\beta^2$, 
and we already have $\mathrm{Hom}_\cOX(F, F(-D))=0$. 
Since $-D\sim K_X$, 
we have 
$\mathrm{Ext}_\cOX^2(F, F(-D))\cong
\mathrm{Ext}_\cOX^2(F, F\otimes\mathcal{O}_X(K_X))$. 
This is dual to $\mathrm{Hom}_\cOX(F, F)$, which is $\mathbb{C}$, 
thus yielding the assertion.

The second assertion follows from Lemma \ref{lem_mmi_and_jac},  
\ref{lem_mmi_mi} and  \ref{lem_mi_is_reduced}. 
\end{proof}

\begin{Remark}
This theorem can be seen as a partial logarithmic analogue of 
the unobstructedness of sheaves on K3 surfaces \cite{Muk84}.

It might also be possible to think of the discreteness of $D(\beta)$ 
as analogous to the obstructedness of divisor classes 
in the moduli space of K\"ahler K3 surfaces. 
\end{Remark}

\section{Basic stable log maps} \label{sec:loggw}

In this section, 
we recall some definitions and facts about basic stable log maps 
from \cite{GS13, Chen14, AbramChen14} 
and \cite[Section 2]{CGKT2}. We restrict to maps of genus $0$ with $1$ distinguished marked point, target varieties endowed with log structure coming from a divisor. The general practice in log geometry is to underline log schemes to denote the underlying scheme. We only follow this convention when we wish to emphasize the distinction. See \cite{vGOR} for an introduction to log geometry.

Let $X$ be a smooth variety and $D$ a divisor on $X$. 
We will consider a stable log map $f$ maximally tangent to $D$ at a smooth point of $D$, 
and study the local structure of the moduli space at $[f]$. 
This depends only on a neighborhood of $\mathrm{Im}\,f$, 
so we may assume that $D$ is smooth and connected. 
We view $X$ as the log scheme $(X, \shM_X)$ 
endowed with the divisorial log structure associated to $D$. 
Let $\beta\in\hhh_2(X,\ZZ)$ be a curve class.

\subsection{Basic $1$-marked genus $0$ stable log maps to a smooth pair}
\label{ss_basic}

For a log scheme $(X, \shM_X)$, 
the monoid homomorphism $\shM_X\to\cO_X$ will be 
denoted by $\alpha$. 

\begin{Definition}
Let $(\cC/W, \{x_1\})$ be a $1$-marked pre-stable log curve 
(in the sense of \cite[Def. 1.3]{GS13}) 
over a log scheme $W$ 
and $(\cC/W, \{x_1\}, f)$ a stable log map to $X$ over $W$ 
(i.e., $f: \cC\to X$ is a log morphism 
\[
\xymatrix{
\cC \ar[r]^f \ar[d] & X \ar[d] \\
W \ar[r] & \Spec \CC  
}
\]
and $\underline{f}$ is a stable map over $\underline{W}$, see \cite[Def. 1.6]{GS13}). 

It is called 
a stable log map of \emph{maximal tangency} 
of genus $0$ and class $\beta$ 
if the following hold for any geometric point $\bar{w}$ of $\underline{W}$: 

(i)
$\cC_{\bar{w}}$ is of arithmetic genus $0$ and $(f_{\bar{w}})_*[\cC_{\bar{w}}]=\beta$. 

(ii)
$f(x_1(\bar{w}))\in D$, and 
the natural map 
\[
\NN\cong\Gamma(X, \overline{\shM}_X)\longrightarrow 
\overline{\shM}_{\cC_{\bar{w}}, x_1(\bar{w})}
\to \overline{\shM}_{\cC_{\bar{w}}, x_1(\bar{w})}/Q 
\cong \NN, 
\]
where $Q=\overline{\shM}_{W, \bar{w}}$, is given by $1\mapsto D.\beta$. 

In many cases, the condition (ii) follows from the rest of the conditions 
(\cite[Proposition 2.9]{CGKT2}, Proposition \ref{cor_description_log_maps}). 
\end{Definition}

\begin{Remark}
(1)
In the language of \cite[Def. 3.1]{GS13}, 
this is the case $g=0$, $k=1$, the data ``$A$'' provided by $\beta$, 
$Z_1=D$ 
and $s_1\in\Gamma({D}, (\overline{\shM}_D^{\textrm{gp}})^*)$ given by 
$\overline{\shM}_D^{\textrm{gp}}\simeq \ZZ_D \to \ZZ_D, \: 
1\mapsto D.\beta$. 

(2)
The log structure on $X$ is defined in the Zariski topology. 
Also, we work with genus $0$ domain curves, 
and the base scheme $\underline{W}$ will mainly be 
the spectrum of a finite dimensional local algebra over 
an algebraically closed field containing $\CC$, 
and in that case 
it suffices to consider log structures in the Zariski topology. 
\end{Remark}

Now we recall the definition of basicness (minimality) for stable log maps 
in the case of our concern, 
i.e. genus $0$ stable log maps of maximal tangency 
to a log scheme associated to a smooth pair. 
See \cite[\S\S1.4]{GS13} for the general case.

Let $\underline{W}=\Spec \kappa$ for an algebraically closed field $\kappa\supseteq\CC$ 
and let $(\underline{C}/\underline{W}, \{x_1\}, f)$ be the 
$1$-marked genus $0$ stable map of class $\beta$ 
that underlies a maximally tangent stable log map. 
Let $\pi: \underline{C}\to \underline{W}$ denote the structure morphism. 
Denote the nodes of $\underline{C}$ by $R_1, \dots, R_k$,
the irreducible components by $C_0, \dots, C_k$, 
and their generic points by $\eta_0, \dots, \eta_k$. 
We choose maps $g_1, g_2: \{1, \dots, k\}\to\{0, \dots, k\}$ 
such that $R_i=C_{g_1(i)}\cap C_{g_2(i)}$. 

We consider data 
$(Q, \overline{\shM}_C, \psi, \varphi)$, 
where 
\begin{itemize}
\item
$Q$ is a fine saturated (fs) monoid without invertible elements, 
regarded as a sheaf on $\underline{W}$, 
\item
$\overline{\shM}_C$ is a fs sheaf of monoids on $\underline{C}$, and 
\item
$\psi: \pi^{-1}Q\to\overline{\shM}_C$, 
$\varphi: f^{-1}\overline{\shM}_X\to\overline{\shM}_C$ are local homomorphisms 
(i.e. non-invertible elements are mapped to non-invertible elements). 
\end{itemize}
Here $\overline{\shM}_X$ is the ghost sheaf of the log structure on $X$, 
but $\overline{\shM}_C$ is a priori just a general sheaf of monoids satisfying the conditions above. 

We require that $(Q, \overline{\shM}_C, \psi)$ 
endows $\underline{C}$ with the structure of a pre-stable log curve 
on the level of ghost sheaves, i.e., 
\begin{itemize}
\item
$\psi$ is an isomorphism on $\underline{C}\setminus\{x_1, R_1, \dots, R_k\}$, 
\item
there is an isomorphism 
$\overline{\shM}_{C, x_1}\to Q\oplus\NN$ compatible with $\psi$ 
(which is unique), and 
\item
for each $i$, 
there is an element $\rho_{R_i}\in Q\setminus\{0\}$ such that, 
if we consider the homomorphism $\NN\to Q; 1\mapsto \rho_{R_i}$ 
and the diagonal map $\NN\to\NN^2$ and take the pushout, 
there is an isomorphism 
$\overline{\shM}_{C, R_i}\cong Q\oplus_{\NN}\NN^2$ 
characterized by the condition that the generization map 
$\overline{\shM}_{C, R_i}\to\overline{\shM}_{C, \eta_{g_j(i)}}\cong Q$ 
is given by $(q, (a_1, a_2))\mapsto q+a_j\cdot\rho_{R_i}$. 
\end{itemize}

In fact, these conditions are satisfied 
for the ghost sheaves and their homomorphisms 
obtained from a stable log map. 

Assume that 
$f(x_1)\in D$. 
We may also assume that $\{i\mid f(R_i)\in D\}=\{1, \dots, k'\}$ 
and 
$\{i\mid f(\eta_i)\in D\}=\{0, \dots, k''\}$ 
after renumbering.

\begin{Definition}
The type of $(Q, \overline{\shM}_C, \psi, \varphi)$ as above 
is the element 
\[
\boldsymbol{u}=(u_{x_1}, u_{R_1}, \dots, u_{R_{k'}}) \in \NN\times\ZZ^{k'} 
\]
defined as follows. 
\begin{itemize}
\item
$\mathbb{N} \ni u_{x_1}$ is the image of $1$ 
by the map 
\[
\NN\cong (f^{-1}\overline{\shM}_X)_{x_1}
\overset{\varphi_{x_1}}{\longrightarrow} \overline{\shM}_{C, x_1}\longrightarrow 
\overline{\shM}_{C, x_1}/Q\cong \NN. 
\]
\item
$\mathbb{Z} \ni u_{R_i}$ is characterized by the equality
\[
u_{R_i}\cdot\rho_{R_i}=\varphi_{\eta_{g_2(i)}}(\chi_{i, 2}) - \varphi_{\eta_{g_1(i)}}(\chi_{i, 1}) \in Q, 
\]
where $\chi_{i, j}=1\in \NN\cong (f^{-1}\overline{\shM}_X)_{\eta_{g_j(i)}}$ 
if $g_j(i)\leq k''$ (i.e. $f(\eta_{g_j(i)})\in D$), 
and $\chi_{i, j}=0$ otherwise.
\end{itemize}
\end{Definition}

Given a stable map $f$ as above and a type $\boldsymbol{u}$, 
we define a monoid $Q^{\mathrm{basic}}$. 
We regard $\ZZ^{k'} \times \ZZ^{k''+1}$ as the direct product 
of $\ZZ$, one copy for each of $R_1, \dots, R_{k'}$ and 
$\eta_0, \dots, \eta_{k''}$. 
In this additive group, 
let $e'_1, \dots, e'_{k'}, e''_0, \dots, e''_{k''}$ be the standard basis, 
$e''_{k''+1}, \dots, e''_k$ the zero vector and 
\[
a_{R_i} = u_{R_i}\cdot e'_i + e''_{g_i(1)} - e''_{g_i(2)}
=
\begin{blockarray}{ccccccc}
 & \hbox{\scriptsize $i$} & &  \hbox{\scriptsize $k'+g_i(1)+1$} & &  \hbox{\scriptsize $k'+g_i(2)+1$} & \\
\begin{block}{(ccccccc)}
\dots, & u_{R_i}, & \dots, & 1, & \dots, & -1, & \dots \\
\end{block}
\end{blockarray} 
\]
where the component $1$ or $-1$ does not appear 
if $g_i(1)>k''$ or $g_i(2)>k''$. 
Then let $\mathfrak R$ be the saturation of the subgroup 
generated by $a_{R_1}, \dots, a_{R_{k'}}$. 
We define $Q^{\mathrm{basic}}$ 
to be the saturation of the image of  the natural map 
\[
\NN^{k'} \times \NN^{k''+1}
\to 
(\ZZ^{k'} \times \ZZ^{k''+1})/\mathfrak R. 
\]
If we change the numbering of $R_i$, $\eta_i$ and the maps $g_j$ 
in a compatible way, 
there is a canonical homomorphism between the corresponding monoids. 
For example, exchanging $g_1(i)$ and $g_2(i)$ 
reverses the sign of $u_{R_i}$, hence that of $a_{R_i}$ 
and thus $\mathfrak{R}$ does not change. 

\begin{Definition}[{\cite[Definition 1.20, Proposition 1.19]{GS13}}]
Let $(C/W, \{x_1\}, f)$ be a stable log map 
over an fs log point $W=(\Spec\kappa, \shM_W)$ 
and let $\boldsymbol{u}$ denote its type. 
Then it is called basic if the induced homomorphism 
$f^{-1}\overline{\shM}_X\to\overline{\shM}_C$ 
is universal in the category of maps of the ghost sheaves 
of type $\boldsymbol{u}$ from  $f^{-1}\overline{\shM}_X$. 
This is equivalent to the condition that 
the natural homomorphism $Q^{\mathrm{basic}}\to\overline{\shM}_W$, 
induced by 
$(\rho_{R_1}, \dots, \rho_{R_{k'}}, \varphi_{\eta_0}(1), \dots, \varphi_{\eta_{k''}}(1))$, 
is an isomorphism. 

A stable log map $(C/W, \{x_1\}, f)$ over a general fs log scheme $W$ 
is called basic if, for any geometric point $\bar{w}\to \underline{W}$, 
the induced log map over $\bar{w}$ is basic. 
\end{Definition}

\subsection{Moduli}

By \cite[Theorem 0.1]{GS13} the stack of basic stable log maps to $X$ is an algebraic log stack, which is locally of finite type. By \cite[Theorem 0.2]{GS13}, imposing genus 0, class $\beta$ and maximal tangency with $D$ cuts out a proper Deligne-Mumford stack $\overline{\text{M}}_\beta(X,D)$, which admits a virtual fundamental class, thus yielding the log Gromov-Witten invariant $\nxd$ of maximal tangency of $(X,D)$. In fact, by \cite[Corollary 1.2]{Wise19}, the forgetful morphism 
from $\overline{\text{M}}_\beta(X,D)$ to the moduli space of genus 0 stable maps of $X$ of class $\beta$ is finite. In what follows, we analyze special components of $\overline{\text{M}}_\beta(X,D)$ 
corresponding to log maps whose images have two irreducible components. 

The log structure at the nodes gives rise to tropical curves. We illustrate it by the example that will be relevant in the next section, cf.\ Corollary \ref{cor_configuration}. Consider a chain of smooth rational curves $C_1$, $C_0$, $C_2$ meeting successively in nodes. We map it to $(X,D)$ by collapsing $C_0$ to $P\in D$ and mapping $C_1$ and $C_2$ each to a rational curve in $X$ meeting $D$ only at $P$ in maximal tangency. This gives a ``tropical curve in $\mathbb{R}$'' as in Figure \ref{trop}. All edges are parallel, weighted and satisfy the balancing condition, see Corollary \ref{cor_configuration}. For more details, see also \cite[Section 2]{CGKT2}.

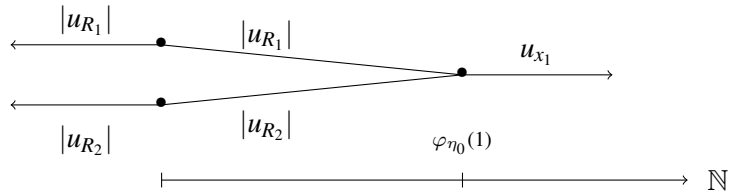
\begin{figure}[htb]
\begin{center}
\caption{Tropical curve}
\label{trop}
\begin{tikzpicture}[smooth, scale=2]
\draw[|-] (-0,0.3) to (2,0.3);
\draw[|->] (2,0.3) to (3.5,0.3);
\draw[->] (2,1) to (3,1);
\draw[<->] (-1,0.8) to (0,0.8) to (2,1) to (0,1.2) to (-1,1.2);
\node at (3.7,0.3) {$\mathbb{N}$};
\node at (0,1.6) {$\quad$};
\node at (0,0.8) {$\bullet$};
\node at (0,1.2) {$\bullet$};
\node at (2,1) {$\bullet$};
\node at (2,0.55) {${\scriptstyle \varphi_{\eta_0}(1)}$};
\node[above] at (2.5,1) {$u_{x_1}$};
\node[above] at (-0.5,1.2) {$|u_{R_1}|$};
\node[above] at (0.7,1.1) {$|u_{R_1}|$};
\node[above] at (-0.5,0.4) {$|u_{R_2}|$};
\node[above] at (0.7,0.5) {$|u_{R_2}|$};
\end{tikzpicture}
\end{center}
\end{figure}

\section{Proof of Theorem \ref{theorem_main}}\label{sec:thmmain}

In this section, we prove Theorem \ref{theorem_main}. 

\subsection{Maximally tangent genus $0$ log maps with $2$ non-collapsed components}

The following gives a rough idea of what a genus $0$, $1$-marked stable log map looks like. 

\begin{Proposition}[{\cite[Corollary 2.10]{CGKT2}}]\label{cor_description_log_maps}
Let $X$ be a divisorial log scheme 
given by a smooth variety ${X}$ 
and a smooth divisor ${D}$. 

For a genus 0 stable log map $f: (C/W, x_1)\to X$ 
with $\underline{W}=\Spec\CC$, 
assume the following: 
\begin{itemize}
\item
$w:={D}.{f}_*[{C}]>0$ 
and $d_i:={D}.{f}_*[C_i]\geq 0$ 
for any irreducible component $C_i$ of ${C}$. 
\item
If $C_i$ is an irreducible component of ${C}$ 
that is not collapsed by ${f}$, 
then ${f}(C_i)\not\subseteq{D}$. 
\end{itemize}
Then it is of maximal tangency, 
and the following holds. 
\begin{enumerate}
\item
${f}({C})\cap {D}$ consists of one point $P$. 
\item
If there is only $1$ non-collapsed component, 
then ${C}\cong \PP^1$ and $f^*(D)=w x_1$. 
\item
If there are at least $2$ non-collapsed components, 
and ${D}.{f}_*[C_i]> 0$ holds for non-collapsed components, 
then ${C}$ is given by adding $C_i=\PP^1$ as leaves 
to a tree ${C}'$ of $\PP^1$ collapsed to $P$, 
with maps $f_i: C_i\to X$ 
satisfying $f_i^*({D})=d_i(C_i\cap {C}')$. 
\end{enumerate}
\end{Proposition}

Now we consider the case where the image of 
$\underline{f}$ has $2$ components. 
We employ the notation of \S\S\ref{ss_basic}. The following Corollary is a direct application of Proposition \ref{cor_description_log_maps}.

\begin{Corollary}\label{cor_configuration}
Let $X$ be a divisorial log scheme 
given by a smooth variety ${X}$ 
and a smooth divisor ${D}$. 

For a genus 0 stable log map $f: (C/W, x_1)\to X$ 
with $\underline{W}=\Spec\CC$, 
assume that $C$ has $2$ non-collapsed components $C_1$ and $C_2$ 
with $f(C_i)\not\subseteq D$, 
that $f\big{|}_{C_i}$ are birational, 
and that $f(C_1)\not= f(C_2)$. 

Assume $d_i:=D.f(C_i)>0$ and let $d=\gcd(d_1, d_2)$ and $e_i=d_i/d$. 
Then the following holds. 

(1) 
$C$ is a chain of smooth rational curves $C_1$, $C_0$ and $C_2$ in this order, 
$C_0$ mapping to a point $P\in D$, 
and we have $(f|_{C_i})^*D=d_iR_i$ for $i=1, 2$, 
where $R_i=C_i\cap C_0$. 

We think of $C_0$ as the ``first'' component at $R_1, R_2$, 
i.e. 
we set $g_1(1)=0$, $g_2(1)=1$ and $g_1(2)=0$, $g_2(2)=2$. 

(2)
With the notation in \S\S\ref{ss_basic}, 
the homomorphism on the log structure at the ghost level 
is given by 
$\varphi_{R_1}(1)=\overline{(0, (d_1, 0))}$, 
$\varphi_{R_2}(1)=\overline{(0, (d_2, 0))}$ 
and $\varphi_{\eta_0}(1)=d_1\cdot\rho_{R_1}=d_2\cdot\rho_{R_2}$. 
The type of $f$ is 
given by $\boldsymbol{u}=(u_{x_1}, u_{R_1}, u_{R_2})=(d_1+d_2, -d_1, -d_2)$. 

(3)
The stable log map $f$ is basic if and only if 
$\overline{\shM}_W$ is isomorphic to $\NN$ and, 
identifying $\overline{\shM}_W$ with $\NN$, 
$\rho_{R_1}=e_2$, $\rho_{R_2}=e_1$ 
and $\varphi_{\eta_0}(1)=de_1 e_2$. 
The last condition can be replaced by 
``either $\rho_{R_1}=e_2$, $\rho_{R_2}=e_1$ 
or $\varphi_{\eta_0}(1)=de_1 e_2$.'' 
\end{Corollary}

\subsection{Coordinates and log structures}

In the rest of this section, we make the following assumptions, 
which are part of those in Theorem \ref{theorem_main} 
after taking a neighborhood of $Z_1\cup Z_2$. 
\begin{Assumptions}\label{assump_pair}
Let $(X, D)$ be a pair of 
a smooth variety and a smooth divisor. 
Let $Z_1$ and $Z_2$ be proper integral curves in $X$ satisfying the following: 
\begin{enumerate}
\item[(1)]
$Z_i$ is a rational curve of class $\beta_i$ maximally tangent to $D$, 
\item[(3)]
$Z_1\cap D$ and $Z_2\cap D$ consist of the same point $P$, 
and 
\item[(4')]
$Z_1\not=Z_2$. 
\end{enumerate}
Write $d_i=de_i=D.Z_i$ with $\gcd(e_1, e_2)=1$. 
\end{Assumptions}

We will study the deformation space of stable log maps 
by an explicit calculation. 
As a preparation, 
let us fix coordinate systems. 

In order to deal with the deformations of $C$, 
we make it stable by 
adding marked points. 
The moduli scheme $\bM_{0, 5}$ 
of $5$-marked genus $0$ stable curves 
has a point corresponding to our curve $C=C_1\cup C_0\cup C_2$, 
with markings $x_2, x_3\in C_1$, $x_1\in C_0$ and $x_4, x_5\in C_2$.

\begin{Notation}\label{notation_curve}
As a formal scheme supported by $1$ point, we have the following description 
for the formal neighborhood $\bM$ of $[(C, x_1, \dots, x_5)]\in\bM_{0, 5}$, 
the universal curve $\cC$ 
(which has the same underlying space as $C$) 
over $\bM$ 
and the marked sections $\tx_1, \dots, \tx_5$ extending $x_1, \dots, x_5$. 

\begin{enumerate}
\item
$\bM$ is a formal $2$-disk with coordinates $\mu_1$ and $\mu_2$. 
\item
$\cC=\cU'_1\cup\cU_1\cup\cU_2\cup\cU'_2$, 
where
\begin{itemize}
\item 
$\cU'_i=\bA^1_\bM=\{(z_{i0}, \mu_1, \mu_2)\}$, 
\item
$\cU_i=\{(z_{i1}, z_{i2}, \mu_1, \mu_2))| z_{i1}z_{i2}=\mu_i\}
\subset\bA^2_\bM$, 
\item
$\{z_{i0}\not=0\}\subset\cU'_i$ and $\{z_{i1}\not=0\}\subset\cU_i$ are patched 
by $z_{i0}=1/z_{i1}$, and 
\item
$\{z_{12}\not=0\}\subset\cU_1$ and $\{z_{22}\not=0\}\subset\cU_2$ are patched 
by $z_{12}=1/z_{22}$
(and $z_{11}=\mu_1z_{22}$, $z_{21}=\mu_2z_{12}$). 
\end{itemize}
We denote the point $(z_{i1}=z_{i2}=0)$ by $R_i$. 
\item
The marked sections are 
$\tx_1: (z_{12}=1)$ on $\cU_1$, 
$\tx_2: (z_{10}=0)$ on $\cU'_1$, 
$\tx_3: (z_{11}=1)$ on $\cU_1$, 
$\tx_4: (z_{21}=1)$ on $\cU_2$ 
and 
$\tx_5: (z_{20}=0)$ on $\cU'_2$. 
\end{enumerate}
We denote the central fiber 
of $\cU_i$ by $U_i$, 
that of $\tx_1$ by $x_1$, etc. 
For $i=1$ or $2$, write $C_i^\circ$ for $C_i\setminus R_i$ 
and $\cC_i^\circ$ for the corresponding 
open subspace of $\cC$. Figure \ref{fig:coord} illustrates the notations for the coordinates at the central fiber. 
\end{Notation}

\begin{figure}[htb]
\centering
\caption{Coordinates along $C$}
\begin{tikzpicture}[line cap=round,line join=round,x=0.4cm,y=0.4cm]
\clip(-0.8,1) rectangle (12.8,11.5);
\draw [line width=1pt,domain=0:3] plot(\x,{(--10-3*\x)});
\draw [line width=1pt,domain=1:11] plot(\x,{(--2)});
\draw [line width=1pt,domain=9:12] plot(\x,{(-26 --3*\x)});
\draw [->] (0.9,8.5) -- (1.4,7) node[near start,right]{$z_{10}$};
\draw [->] (2.967,2.3) -- (2.467,3.8) node[right]{$z_{11}$};
\draw [->] (2.967,2.3) -- (4.547,2.3) node[above]{$z_{12}$};
\draw [->] (11.1,8.5) -- (10.6,7) node[near start,left]{$z_{20}$};
\draw [->] (9.033,2.3) -- (9.533,3.8) node[left]{$z_{21}$};
\draw [->] (9.033,2.3) -- (7.453,2.3) node[above]{$z_{22}$};

\draw [fill=black] (0.5,8.5) circle (1.8pt) node[left]{$x_2$};
\draw [fill=black] (2,4) circle (1.8pt) node[left]{$x_3$};
\draw [fill=black] (6,2) circle (1.8pt) node[below]{$x_1$};
\draw [fill=black] (11.5,8.5) circle (1.8pt) node[right]{$x_4$};
\draw [fill=black] (10,4) circle (1.8pt) node[right]{$x_5$};

\draw (12,2) node {$C_0$};
\draw (0,10) node[above] {$C_1$};
\draw (12,10) node[above] {$C_2$};
\end{tikzpicture}
\label{fig:coord}
\end{figure}
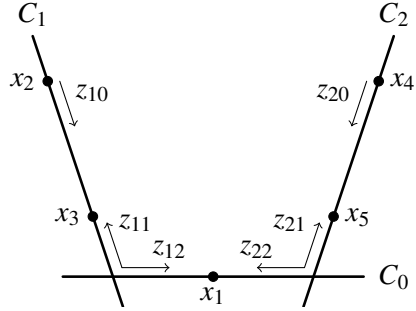

\begin{Lemma}\label{lem_expansion}
Let $S=\Spec R$ where $R$ is a local finite-dimensional $\CC$-algebra 
with a basis $\{r_1, \dots, r_n\}$, 
and $\cC_S$ the curve induced from $\cC$ by a morphism $S\to \bM$. 

Identifying the underlying topological spaces of $\cC$ and $\cC_S$ 
with that of $C$, 
let $U$ be an open neighborhood of $R_i$ in $U_i$. 

Then any element $F$ of $\cO_{\cC_S}(U)$ 
can be uniquely written as 
\[
F = \sum_{k=1}^n r_k\left\{
F_0^{(k)}+z_{i1}F_1^{(k)}(z_{i1})+z_{i2}F_2^{(k)}(z_{i2})
\right\}, 
\]
where $F_0^{(k)}\in\CC$, 
$F_1^{(k)}(z_{i1})\in\cO_{C_i}(U\cap C_i)$ 
and 
$F_2^{(k)}(z_{i2})\in\cO_{C_0}(U\cap C_0)$. 
Here, if $i$ is e.g.\ $1$, we regard $\cO_{C_1}(U\cap C_1)$ as a subring of 
$\cO_{\cC_S}(U)$ using the projection map $\cU_1\to \Spec\CC[z_{11}]$, 
etc. 
\end{Lemma}
\begin{proof}
This basically follows from $z_{i1}z_{i2}=\mu_i$ and flatness (or an explicit calculation). 
\end{proof}

\begin{Notation}\label{notations_log_str}
The standard log structures on $\bM$ and $\cC$ are the divisorial log structures 
defined by the degeneracy locus and its inverse image and the sections. 
We replace the log structure at $\tx_2, \dots, \tx_5$ 
with the one induced from $\bM$. 
They are explicitly given as follows. 
\begin{itemize}
\item
$\cM_{\bM, 0} = \coprod_{a, b=0}^\infty \bmu_1^a\bmu_2^b\cO_{\bM, 0}^\times $, 
with $\alpha(\bmu_1)=\mu_1$ and $\alpha(\bmu_2)=\mu_2$. 
\item
$\cM_{\cC, R_1} = 
\coprod_{a, b, c=0}^\infty \bz_{11}^a\bz_{12}^b\bmu_2^c\cO_{\cC, R_1}^\times$
and
$\cM_{\cC, R_2} = 
\coprod_{a, b, c=0}^\infty  \bz_{21}^a\bz_{22}^b\bmu_1^c\cO_{\cC, R_2}^\times$, \\
with $\alpha(\bz_{ij})=z_{ij}$. 
\item
$\cM_{\cC, x_1} = 
\coprod_{a, b, c=0}^\infty \bz^{\prime a}\bmu_1^b\bmu_2^c\cO_{\cC, x_1}^\times $
with $\alpha(\bz')=z_{12}-1$. 
\item
At other points, 
$\cM_\cC=\coprod_{a, b=0}^\infty \bmu_1^a\bmu_2^b\cO_\cC^\times $. 
\end{itemize}
The structure homomorphism and the generization maps are given by 
the following rule: 
\begin{itemize}
\item
$\bmu_i$ at various points are identified. 
\item
$\bmu_i$ maps to $\bz_{i1}\bz_{i2}\in\shM_{\cC, R_i}$. 
\item
If $\{j, j'\}=\{1, 2\}$, 
$\bz_{ij}$ and $\bz_{ij'}$ 
generizes to $z_{ij}$ and $\bmu_i z_{ij}^{-1}$ on $\{z_{ij}\not=0\}$. 
$\bz'$ generizes to $z_{12}-1$. 
\end{itemize}
\end{Notation}
The parameter spaces we will be concerned with are of the following form. 

\begin{Notation}\label{notation_curve_family}
For a nonnegative integer $n$ 
and $m(s)\in s\cdot\CC[s]/(s^{n+1})$, 
let $S_{n, m(s)}$ be the log scheme $(\underline{S}, \cM)$ 
defined as follows. 
\begin{itemize}
\item
$\underline{S}= S_n = \Spec\CC[s]/(s^{n+1})$. 
\item
$\cM=\NN\times\cO_{\underline{S}}^\times$, 
with $\alpha(k, u(s)) = m(s)^k u(s)$. 
\end{itemize}

For coprime positive integers $r_1, r_2$ 
and $u_1(s), u_2(s)\in \cO_{\underline{S}}^\times$, 
let $\cC_{n, m(s), (r_1, u_1(s)), (r_2, u_2(s))}$ 
denote the log scheme over $S_{n, m(s)}$ 
obtained from $\cC$ by taking the fiber product 
with the log morphism $S_{n, m(s)}\to \bM$ 
defined by $\bmu_i\mapsto (r_i, u_i(s))$ 
(and $\mu_i\mapsto m(s)^{r_i}u_i(s)$) 
in the category of log schemes.  
\end{Notation}

\begin{Lemma}
Write $S$ for $S_{n, m(s)}$ 
and $\cC_S$ for $\cC_{n, m(s), (r_1, u_1(s)), (r_2, u_2(s))}$. 
Then $S$ and $\cC_S$ are fs log schemes 
and $\cC_S/S$ with the induced section 
is a $1$-marked pre-stable log curve. 

As a scheme, it is represented as in Notation \ref{notation_curve} 
with $\mu_i$ replaced by $m(s)^{r_i}u_i(s)$. 

The induced log structure can be described as follows. 
(We use the same symbols for points and functions on $\cC$ and $\cC_S$ etc.) 
\begin{itemize}
\item
$\cM_{S, 0} = \coprod_{a=0}^\infty \boldm^a\cO_{S, 0}^\times $, 
with $\alpha(\boldm)=m(s)$. 
\item
$\cM_{\cC_S, R_i} =
\bigcup_{a, b, c=0}^\infty \bz_{i1}^a\bz_{i2}^b\boldm^c\cO_{\cC_S, R_i}^\times  =
\coprod_{a, b=0}^\infty\coprod_{c=0}^{r_i-1} 
\bz_{i1}^a\bz_{i2}^b\boldm^c\cO_{\cC_S, R_i}^\times $
subject to the relation 
$\bz_{i1}\bz_{i2}=\boldm^{r_i}u_i(s)$, 
with $\alpha(\bz_{ij})=z_{ij}$. 
\item
$\cM_{\cC_S, x_1} = 
\coprod_{a, b=0}^\infty \bz^{\prime a}\boldm^b\cO_{\cC_S, x_1}^\times $ 
with $\alpha(\bz')=z_{12}-1$. 
\item
At other points, $\cM_{\cC_S}=\coprod_{a=0}^\infty \boldm^a\cO_{\cC_S}^\times $. 
\end{itemize}
The structure homomorphism and the generization maps are given by 
the following relations: 
\begin{itemize}
\item
$\boldm$ at various places are to be identified. 
\item
If $\{j, j'\}=\{1, 2\}$, 
$\bz_{ij}$ and $\bz_{ij'}$ 
generizes to $z_{ij}$ and $\boldm^{r_i}u_i(s)z_{ij}^{-1}$ on $\{z_{ij}\not=0\}$. 
$\bz'$ generizes to $z_{12}-1$. 
\end{itemize}
\end{Lemma}
\begin{proof}
One thing that is not so obvious is that the fiber product in the category of log schemes
is a fs log scheme in this case, 
but this is a known fact in the study of moduli of log curves. 
In fact, the ghost sheaf at $R_1$ is the pushout 
\[\overline{\shM}_S\oplus_{\overline{\shM}_{\bM}} \overline{\shM}_{\cC, R_1}
\cong \NN\oplus_{\NN^2}\NN^3
\]
defined from 
$\NN^2\to\NN; (a, b)\mapsto r_1 a +r_2 b$ 
and $\NN^2\to\NN^3; (a, b)\mapsto (a, a, b)$. 
This is isomorphic to the submonoid of $\NN^2$ 
generated by $(r_1, 0)$, $(0, r_1)$ and $(1, 1)$ 
(corresponding to $\overline{\bz}_{11}, \overline{\bz}_{12}$ and $\overline{\boldm}$, 
respectively). 
\end{proof}

\begin{Notation}\label{notation_central_map}
Let $(w_1, w_2, \dots, w_N)$ be \'etale coordinates 
on an affine open neighborhood $W$ of $P$ in $X$ 
such that $w_1=0$ defines $D$ in $W$. 

Let $f^{(i)}: \PP^1\to Z_i\subset X$ be the normalization map. 
We may assume that $\infty \not\in (f^{(i)})^{-1}(W)$. 

For $i=1$ or $2$, 
we identify the domain of $f^{(i)}$ with $C_i$, 
and by extending to $C_0$ by a constant map, 
we obtain a stable map $f_0: (C, x_1)\to X$. 
Denote $f_0^{-1}(W)$ by $V$, 
and then we have $C_0\subset V$ (and in particular $R_1, R_2\in V$) 
and $x_2, x_5\not\in V$. 
We define $A_i(z_{i1})$ and $B_{ji}(z_{i1})$, $i=1, 2$, by 
\begin{eqnarray*}
w_1 & \mapsto & z_{i1}^{d_i} A_i(z_{i1}), 
\qquad A_i(z_{i1})\in\cO_{C_i}(V\cap C_i)^\times, \\
w_j & \mapsto & B_{ji}(z_{i1}), 
\qquad B_{ji}(z_{i1})\in\cO_{C_i}(V\cap C_i), \qquad 2\leq j\leq N. 
\end{eqnarray*}
By the assumptions, we have $B_{j1}(0)=B_{j2}(0)$. 
We write 
\begin{eqnarray*}
V_0 & = & U_1\cap U_2, \\
V_1 & = & V\cap(U_1\setminus x_1), \\
V_2 & = & V\cap(U_2\setminus x_1), 
\end{eqnarray*}
and denote the corresponding open subspaces of $\cC$ 
by $\cV$ and $\cV_i$. 
\end{Notation}

\subsection{Calculation}

Recall that $\barM_\beta=\barM_\beta(X, D)$ is the moduli stack of 
maximally tangent genus $0$ basic stable log maps of class $\beta$ 
to the log scheme associated to $(X, D)$. 

We consider $S=S_n:=\Spec \CC[s]/(s^{n+1})$. 
This will be sufficient 
since the tangent space to $\barM_\beta$ 
is at most $1$-dimensional 
as we will see later. 
(If fact, 
most of what we state below generalize 
to the spectrum of any local finite-dimensional $\CC$-algebra.)

When referring to a base change of something to $S$, 
we sometimes drop $S$, 
e.g.\ we write just $\cV_i$ for $(\cV_i)_S$. 

\begin{Lemma}\label{lem_description_1}
Let $(\cC_S/S, \{\tx_1\}, f)$ be a basic stable log map over $S$ 
such that the cycle theoretic image of the central fiber 
is $Z_1+Z_2$. 

(1)
One can identify the restriction $\underline{f}\big{|}_0$ to the central fiber 
with $f_0$ constructed in Notation \ref{notation_central_map}. 
There exist $m(s)=\sum_{k=1}^n m^{(k)}s^k$ and 
$u_i(s)=\sum_{k=0}^n u_i^{(k)}s^k$ 
such that $\cC_S/S$ is isomorphic to 
$\cC_{n, m(s), (e_2, u_1(s)), (e_1, u_2(s))}/S_{n, m(s)}$. 
(Note that $e_2$ comes first.)

(2)
Via an isomorphism as in (1), 
the log map $f$ gives rise to 
\begin{eqnarray*}
\tilde{A}_i 
 & \in & \cO_{\cC_S}(\cV_i)^\times, \qquad i=1, 2 \\
\tilde{B}_{2i}, \dots, \tilde{B}_{Ni}
 & \in & \cO_{\cC_S}(\cV_i), \qquad i=1, 2, 
\end{eqnarray*}
with $z_{i1}^{d_i}\tilde{A}_i \equiv z_{i1}^{d_i}A_i \mod (s)$ 
and $\tilde{B}_{ji} \equiv B_{ji} \mod (s)$,  
characterized by the condition that 
\begin{eqnarray*}
(f|_{\cV_i})^{\flat}\bw_1 & = & \bz_{i1}^{d_i}\tilde{A}_i, \\
(\underline{f}|_{\cV_i})^*w_j & = & \tilde{B}_{j i} \qquad j=2, \dots, N. 
\end{eqnarray*}
Here, $\bw_1$ is $w_1$ considered as a local section of $\shM_X$. 

Conversely, such data uniquely determines $f$ 
(under the conditions described in the next lemma). 
\end{Lemma}
\begin{proof}
(1)
By Corollary \ref{cor_configuration} (1), 
we can identify $\underline{f}\big{|}_0$ with $f_0$. 
By Corollary \ref{cor_configuration} (3), 
$\shM_{S, 0}\cong \NN$ and hence 
$S$ is isomorphic to $S_{n, m(s)}$ for some $m(s)$. 

Since the central fiber has $3$ components, by \cite{K00}, 
(modulo adding $4$ sections and then subtracting), 
$\cC_S$ is induced from some log morphism $S\to \bM$. 
By Corollary \ref{cor_configuration} (3), 
$\bmu_1$ maps to $e_2$ and $\bmu_2$ maps to $e_1$ 
at the level of ghost sheaves. 
Thus $S\to \bM$ is of the asserted form.

(2)
By Corollary \ref{cor_configuration} (2), 
$(f|_{\cV_i})^{\flat}\bw_1$ is of the form as above. 

For the converse: 
By the compatibility of $f^*$ and $f^{\flat}$, 
we have 
$(\underline{f}|_{\cV_i})^*w_1 = z_{i1}^{d_i}\tilde{A}_i$. 
Since $w_1, \dots, w_N$ are \'etale coordinates on $W$, 
this determines $f|_{\cV_S}$. 
If $j: V\to C$ denotes the inclusion map, 
then $\cO_{\cC_S}\to j_*\cO_{\cV_S}$ is injective, 
and therefore 
$\tilde{A}_i , \tilde{B}_{2i}, \dots, \tilde{B}_{Ni}$ 
determine $f$. 
\end{proof}

By Lemma \ref{lem_expansion}, 
we can expand $\tilde{A}_i$ and $\tilde{B}_i$ as 
\begin{eqnarray*}
\tilde{A}_i & = & \sum_{k=0}^n s^k\left\{
A_{i0}^{(k)}+z_{i1}A_{i1}^{(k)}(z_{i1})+z_{i2}A_{i2}^{(k)}(z_{i2})
\right\} \\
\tilde{B}_{j i} & = & \sum_{k=0}^n s^k\left\{
B_{ji0}^{(k)}+z_{i1}B_{ji1}^{(k)}(z_{i1})+z_{i2}B_{ji2}^{(k)}(z_{i2})
\right\} 
\end{eqnarray*}
with 
\begin{eqnarray*}
A_{i0}^{(k)}, B_{ji0}^{(k)}& \in & \CC, \\
A_{i1}^{(k)}(z_{i1}), B_{ji1}^{(k)}(z_{i1})
 & \in & \cO_{C_i}(V_i\cap C_i)=\cO_{C_i}(V\cap C_i), \\
A_{i2}^{(k)}(z_{i2}), B_{ji2}^{(k)}(z_{i2})
 & \in & \cO_{C_0}(V_i\cap C_0)=\cO_{C_0}((U_i\cap C_0)\setminus x_1). 
\end{eqnarray*}

Let us write down the conditions for these data to actually give a log map. 

\begin{Lemma}\label{lem_conditions}
The above data give a family in $\barM_\beta$ 
if and only if the following hold. 
\begin{enumerate}
\item[(a)]\label{bw1_on_V0}
($\bw_1$ on $\cV_0$)\qquad
$z_{12}^{-d_1}u_1(s)^{d_1}\tilde{A}_1=z_{22}^{-d_2}u_2(s)^{d_2}\tilde{A}_2$ 
on $\cV_0\setminus x_1$, 
and it extends to a function on $\cV_0$ 
with vanishing order $d_1+d_2$ along $\tx_1$. 
\item[(b)]\label{wj_on_V0}
($w_2, \dots, w_N$ on $\cV_0$)\qquad
For $j=2, \dots, N$, 
$\tilde{B}_{j1}|_{\cV_0\setminus x_1}=\tilde{B}_{j2}|_{\cV_0\setminus x_1}$, 
and it extends to a regular function at $x_1$. 
\item[(c)]\label{wj_on_X1X2}
($w_1, \dots, w_N$ on $\cC_1^\circ$, $\cC_2^\circ$)\qquad
The morphism $\cV\cap \cC_i^\circ\to S$ given by 
$z_{i1}^{d_i}\tilde{A}_i$ and $\tilde{B}_{2i}, \dots, \tilde{B}_{Ni}$ 
extends to a morphism $\cC_i^\circ\to X$. 
\end{enumerate}
\end{Lemma}

\begin{proof}
We first show that the conditions are necessary:

For (a), on $\cV_0$ we have 
$\bz_{11}=\boldm^{e_2}z_{12}^{-1}u_1(s)$ and $\bz_{21}=\boldm^{e_1}z_{22}^{-1}u_2(s)$, 
and so 
\begin{eqnarray*}
\bz_{11}^{d_1}\tilde{A}_1 & = & \boldm^{de_1e_2}z_{12}^{-d_1}u_1(s)^{d_1}\tilde{A}_1,  \\
\bz_{21}^{d_2}\tilde{A}_2 & = & \boldm^{de_1e_2}z_{22}^{-d_2}u_2(s)^{d_2}\tilde{A}_2. 
\end{eqnarray*}
These must be equal, so we have the first assertion. 
At $x_1$ it is equal to $\boldm^{de_1e_2}(\bz')^{d_1+d_2}\cdot(\hbox{unit})$, hence the second assertion. 
It is obvious that (b) and (c) must hold in order to define $\underline{f}$.

The conditions are also sufficient:
From (a), by the equalities above we have 
$z_{11}^{d_1}\tilde{A}_1 = z_{21}^{d_2}\tilde{A}_2$ on $\cV_0\setminus x_1$, 
and it extends to $x_1$. 
Together with (b) and (c), these data give a stable map $\underline{f}$. 
By (a), we can lift it to a log map 
by sending $\bw_1$ to $\bz_{11}^{d_1}\tilde{A}_1$, 
$\bz_{21}^{d_2}\tilde{A}_2$ and $\boldm^{de_1e_2}(\bz')^{d_1+d_2}\cdot(\hbox{unit})$ 
at $R_1$, $R_2$ and $x_1$, respectively. 
\end{proof}

Let 
$c_0$ be a $d_2$-th roots 
of $(-1)^{d_1+d_2} A_1(0)/A_2(0)$ 
and set $c_p:=e^{2\pi p \sqrt{-1} /d_2}c_0$ for $p=1, \dots, d_2-1$. 
The following theorem proves Theorem \ref{theorem_main}. 

\begin{Theorem}\label{thm:412}
Let $(X, D)$ and $Z_i$ be as in Assumption \ref{assump_pair} 
and $f^{(i)}, A_i, B_{ji}$ be as in Notation \ref{notation_central_map}. 
For $n<\min\{e_1, e_2\}$, 
let $S_{n, s}$ 
and $\cC_{n, p}=\cC_{n, s, (e_2, 1), (e_1, c_p)}$ 
be as defined in Notation \ref{notation_curve_family}. 
\begin{enumerate}
\item
The underlying pre-stable curve of $\cC_{n, p}/S_{n, s}$ is the trivial family. 
\item
The functions 
\begin{eqnarray*}
\tilde{A}_i & = & A_i(z_{i1}) + A_i(0)\left\{(1-z_{i2})^{d_1+d_2}-1\right\}, \\
\tilde{B}_{j, i} & = & B_{j, i}(z_{i1})
\end{eqnarray*}
define a stable log map $f_{n, p}: \cC_{n, p}\to X$, 
whose underlying stable map is the trivial family.
\item
The stable log maps $f_{0, p}$ and $f_{0, p'}$ are isomorphic 
if and only if $p\equiv p'\pmod{d}$. 
Each $f_{0, p}$ has no nontrivial automorphisms. 
\item
The stable log map $f_{n, p}$ gives a closed immersion 
of $S_{n, s}$ into $\barM_\beta$. 
\item
The stable log maps $f_{0, p}$ exhaust the maps whose image cycles are $Z_1+Z_2$. 
\item
Under the assumptions of Theorem \ref{theorem_main}, 
the stable log map $f_{\min\{e_1, e_2\}-1, p}$ gives an isomorphism 
of $S_{\min\{e_1, e_2\}-1}$ with a connected component. 
\end{enumerate}
\end{Theorem}
\begin{proof}
(5) and (6) will be proven in the subsections that follow. 

(1)
Since $e_1, e_2> n$, 
we have $s^{e_1}=s^{e_2}=0$ in $\CC[s]/(s^{n+1})$. 
Thus the underlying pre-stable curve is trivial. 

(2)
We have $z_{i1}z_{i2}=0$ from (1). 
Thus $z_{i1}^{d_i}\tilde{A}_i=z_{i1}^{d_i}A_i(z_{i1})$, 
and these data give a trivial stable map. 
In particular, 
the conditions (b) and (c) of Lemma \ref{lem_conditions} are clearly satisfied. 

Let us check Lemma \ref{lem_conditions} (a). 
We have $u_1(s)=1$ and $u_2(s)=c_p$ and therefore, on $\cV_0$, 
\begin{eqnarray*}
z_{12}^{-d_1}u_1(s)^{d_1}\tilde{A}_1 & = & 
z_{12}^{-d_1}A_1(z_{11}) + 
z_{12}^{-d_1}A_1(0)\left\{(1-z_{12})^{d_1+d_2}-1 \right\} \\
& = & 
z_{12}^{-d_1}A_1(0) + 
z_{12}^{-d_1}A_1(0)\left\{(1-z_{12})^{d_1+d_2}-1 \right\} \\
& = & A_1(0)z_{12}^{-d_1}(1-z_{12})^{d_1+d_2}, \\
\end{eqnarray*}
and, similarly, 
\begin{eqnarray*}
z_{22}^{-d_2}u_2(s)^{d_2}\tilde{A}_2 & = & 
c_p^{d_2}A_2(0)z_{22}^{-d_2}(1-z_{22})^{d_1+d_2}. 
\end{eqnarray*}
Since $c_p^{d_2}=(-1)^{d_1+d_2} A_1(0)/A_2(0)$, 
they are equal. 

Clearly they vanish to order $d_1+d_2$ at $x_1$.

(3)
An isomorphism of $f_{0, p}$ and $f_{0, p'}$ 
is given by a compatible pair 
of isomorphisms $\varphi: \cC_{0, p}\to\cC_{0, p'}$ 
and $\psi: S_{0, s}\to S_{0, s}$. 

The morphism $\underline{\varphi}: \underline{\cC}_{0, p}\to\underline{\cC}_{0, p'}$
underlying $\varphi$ is the identity, 
since $\underline{f}_{0, p}$ and $\underline{f}_{0, p'}$ are stable 
and $Z_1\not=Z_2$ by assumption. 
Therefore there exist functions $f_1(z_{12})$ with $f_1(0)=1$ 
such that $\varphi^\flat\bz_{11}=\bz_{11}f_1(z_{12})$, 
and $f_2(z_{11})$ with $f_2(0)=1$ such that 
$\varphi^\flat\bz_{12}=\bz_{12}f_2(z_{11})$. 

On the other hand, $\psi$ is determined by a nonzero complex number $\gamma$ 
such that $\psi^{\flat}\boldm=\boldm\gamma$. 

Then we have 
\begin{eqnarray*}
\bz_{11}\bz_{12}f_1(z_{12})f_2(z_{11}) 
& = & \varphi^\flat(\bz_{11}\bz_{12}) \\
& = & \psi^\flat(\boldm^{e_2}) \\ 
& = & \boldm^{e_2} \gamma^{e_2}. 
\end{eqnarray*}
Thus $f_1\equiv f_2 \equiv 1$ and $\gamma^{e_2}=1$. 

Similarly, there exist functions $g_1(z_{22})$ with $g_1(0)=1$ 
such that $\varphi^\flat\bz_{21}=\bz_{21}g_1(z_{22})$, 
and $g_2(z_{21})$ with $g_2(0)=1$ such that 
$\varphi^\flat\bz_{22}=\bz_{22}g_2(z_{21})$. 
\begin{eqnarray*}
\bz_{21}\bz_{22}g_1(z_{22})g_2(z_{21}) 
& = & \varphi^\flat(\bz_{21}\bz_{22}) \\
& = & \psi^\flat(\boldm^{e_1}c_{p'}) \\ 
& = & \boldm^{e_1} \gamma^{e_1}c_{p'}. 
\end{eqnarray*}
Thus $g_1\equiv g_2\equiv 1$ and $c_p=\gamma^{e_1}c_{p'}$. 
Thus $c_p/c_{p'}$ is an $e_2$-th root of $1$, 
and $p\equiv p'\pmod{d}$ follows. 
We can easily see that the converse is also true. 

In the case $p=p'$, 
the calculation above shows 
$\varphi^\flat\bz_{ij}=\bz_{ij}$ and $\psi^\flat\boldm=\boldm$, 
and therefore that $f_{0, p}$ has no nontrivial automorphisms. 

(4)
By (3), $\barM_\beta$ is an algebraic space near $[f_{0, p}]$. 
Therefore, the assertion is clear if $n=0$. 
If $n\geq 1$, the truncation $f_{1, p}$ is nontrivial, 
since $\alpha(\boldm)\not=0$. 
Thus $S_n\to\barM_\beta$ has nonzero tangent map, 
and it is a closed immersion by Nakayama's Lemma. 
\end{proof}

\subsection{Central fiber}

Now we prove Theorem \ref{thm:412} (5). 
Thus we set $n=0$ in this subsection. 

We know that the curve 
is as in Lemma \ref{lem_description_1} (1) with $n=0$ 
and the maps are given by data as in Lemma \ref{lem_description_1} (2) 
satisfying the conditions of Lemma \ref{lem_conditions}. 
Here $m(s)=0$ since it is in the maximal ideal. 
Note that by replacing $\boldm$ (which does not affect $m(s)$) 
we may suppose that $u_1=1$.

For the expansions of $\tilde{A}_i$ and $\tilde{B}_{ji}$, the condition of Lemma \ref{lem_description_1} (2) states 
\begin{eqnarray*}
z_{i1}^{d_i}(A_i^{(0)}+z_{i1}A_{i1}^{(0)}(z_{i1})+z_{i2}A_{i2}^{(0)}(z_{i2})) & = & z_{i1}^{d_i}A_i(z_{i1}), \\
B_{ji0}^{(0)}+z_{i1}B_{ji1}^{(0)}(z_{i1})+z_{i2}B_{ji2}^{(0)}(z_{i2}) & = & B_{ji}(z_{i1}) 
\end{eqnarray*}
on $V_i$. 
From the first equality and $z_{i1}z_{i2}=0$, we have
$A_i^{(0)}+z_{i1}A_{i1}^{(0)}(z_{i1}) = A_i(z_{i1})$ 
and in particular $A_i^{(0)}=A_i(0)$. 
Thus it remains to show that 
$A_i(0)+z_{i2}A_{i2}^{(0)}(z_{i2})=A_i(0)(1-z_{i2})^{d_1+d_2}$ 
and that $u_2$ is one of $c_p$, 
the latter being equivalent to  $u_2^{d_2}=(-1)^{d_1+d_2}A_1(0)/A_2(0)$. 

In fact, 
for the function $A_{12}^{(0)}(z_{12})\in \CC[z_{12}, 1/(z_{12}-1)]$, 
Lemma \ref{lem_conditions} (a) says that 
$A_1(0)+z_{12}A_{12}^{(0)}(z_{12})$ is regular also at $z_{12}=1$. 
Therefore $A_{12}^{(0)}(z_{12})$ is in fact a polynomial. 
Similarly for $A_{22}^{(0)}(z_{22})$. 
Restricting to $\cV_0$ so that $z_{11}=0$ and $z_{21}=0$, by Lemma \ref{lem_conditions} (a) we have 
\begin{equation} 
\begin{aligned}
A_1(0)+z_{12}A_{12}^{(0)}(z_{12}) & =  
z_{12}^{d_1}u_1^{-d_1}\cdot z_{22}^{-d_2}u_2^{d_2}(A_2(0)+z_{22}A_{22}^{(0)}(z_{22})) \\
& =  
u_2^{d_2} z_{12}^{d_1+d_2}(A_2(0)+z_{12}^{-1}A_{22}^{(0)}(z_{12}^{-1})), 
\end{aligned}\label{eq1}
\end{equation}
and it follows that this is a polynomial of degree $d_1+d_2$. 
It also must have vanishing order $d_1+d_2$ at $z_{12}=1$, 
and therefore $A_1(0)+z_{12}A_{12}^{(0)}(z_{12})=A_1(0)(1-z_{12})^{d_1+d_2}$. 
By symmetry we also have 
$A_2(0)+z_{22}A_{22}^{(0)}(z_{22}) = A_2(0)(1-z_{22})^{d_1+d_2}$, 
and \eqref{eq1} implies that 
$(-1)^{d_1+d_2}A_1(0) = u_2^{d_2}A_2(0)$. 
(This is basically the same calculation as we did in the proof of (2).)

\subsection{Extending deformations}

Now we make the same assumption as in Theorem \ref{theorem_main}. 
In particular $N=2$, and we denote $B_{2 i}$ by $B_i$ etc. 
We may assume that 
\begin{equation}\label{ineq_b_dash}
B_i'(0)\not=0. 
\end{equation}
In fact, 
since $Z_i$ is smooth at the intersection point $P$, 
this always holds if $d_i>1$. 
If $d_i=1$, this holds after changing coordinates if necessary.

By symmetry, assume that $d_1\leq d_2$. 
Recall that $d_i=de_i$ with $\gcd(e_1, e_2)=1$. 

\begin{Lemma}\label{413}
Theorem \ref{thm:412} (6) (and hence Theorem \ref{theorem_main}) is reduced to 
the following claims. 
\begin{itemize}
\item[(A)]
If $e_1=1$, any extension of $f_{0, p}$ to $S_1$ is trivial. 
\item[(B)] 
If $e_1\geq 2$, 
the set of extensions of $f_{0, p}$ to $S_1$ 
forms an (at most) $1$-dimensional vector space. 
\item[(C)]
For $e_1\geq 2$, 
the family $f_{e_1-1, p}$ cannot be extended 
to a family over $S_{e_1}$. 
\end{itemize}
\end{Lemma}
\begin{proof}
Since $f_{0, p}$ has no nontrivial automorphisms, 
$\barM_\beta$ is an algebraic space near $f_{0, p}$. 
If $e_1=1$, it is a reduced point by (A). 

If $e_1>1$ the tangent space is $1$-dimensional by (B). 
Thus the moduli space is \'etale locally a closed subspace of $\Spec\CC[s]$. 
If it contains $S_{e_1}$, 
our family $f_{e_1-1, p}$ induces a map $\iota: S_{e_1-1}\to S_{e_1}$ 
such that $f_{e_1-1, p}$ is isomorphic to the pullback by $\iota$ of a 
basic stable log map over $S_{e_1}$, 
and its tangent map is nonzero since $m(s)=s\not\equiv 0 \mod s^2$. 
By composing with an automorphism of $S_{e_1}$, 
we may assume that $\iota$ is the standard closed immersion, 
and we would have an extension of $f_{e_1-1, p}$ to $S_{e_1}$. 
\end{proof}

So, for $1\leq n\leq e_1$, 
we consider the family $f_{n-1, p}$ and 
study its extensions to $S_n$. 
We write 
\[
v:=m^{(1)} 
\]
and we have 
\begin{equation*}
m(s) = 
\begin{cases} 
vs & \text{if $n=1$}, \\
vs + m^{(n)}s^n \quad\text{and $v=1$} & \text{if $n\geq 2$}, 
\end{cases}
\end{equation*}
\begin{eqnarray*}
u_1(s) & = & 1 + u_1^{(n)}s^n, \\
u_2(s) & = & c_p + u_2^{(n)}s^n, \\
\tilde{A}_i & = & A_i(z_{i1})+A_i(0)\{(1-z_{i2})^{d_1+d_2}-1\} \\ 
& & 
+ s^n(A_{i0}^{(n)}+z_{i1}A_{i1}^{(n)}(z_{i1})+z_{i2}A_{i2}^{(n)}(z_{i2})), \\
\tilde{B}_i & = & B_i(z_{i1}) + 
s^n(B_{i0}^{(n)}+z_{i1}B_{i1}^{(n)}(z_{i1})+z_{i2}B_{i2}^{(n)}(z_{i2})) 
\end{eqnarray*}
with $s^{n+1}=0$, 
and see when the conditions of Lemma \ref{lem_conditions} 
are satisfied. 

We have 
\begin{eqnarray*}
z_{11} z_{12} & = & m(s)^{e_2}u_1(s) \\
& = & v^{e_2}s^{e_2}, \\ 
z_{21}z_{22} & = & m(s)^{e_1}u_2(s) \\
& = & c_p v^{e_1}s^{e_1}. 
\end{eqnarray*}

\begin{Lemma} \label{lem_b}
Condition (b) of Lemma \ref{lem_conditions} holds if and only if the following hold: 
\[
B_{10}^{(n)}= B_{20}^{(n)} 
\]
and 
\[
B_{12}^{(n)}(z_{12}) \equiv B_{12}^{(n)}(0) 
\text{ and } B_{22}^{(n)}(z_{22})\equiv B_{22}^{(n)}(0)
\]
with 
\begin{eqnarray*}
B_{12}^{(n)}(0)  & = & 
\begin{cases}
B'_2(0)c_p v^n & \text{if $n=e_1$}, \\
0 & \text{otherwise}, 
\end{cases} \\
B_{22}^{(n)}(0)  & = &
\begin{cases}
B'_1(0) v^n & \text{if $n=e_2$, i.e., $n=e_1=e_2=1$}, \\
0 & \text{otherwise}. 
\end{cases}
\end{eqnarray*}
\end{Lemma}
\begin{proof}
Lemma \ref{lem_conditions} (b) says that 
\begin{eqnarray*}
& & B_1(0) + 
s^n(B_{10}^{(n)}+z_{12}B_{12}^{(n)}(z_{12})) + s^{e_2}B'_1(0)v^{e_2}z_{12}^{-1} \\
& = &
B_2(0) + 
s^n(B_{20}^{(n)}+z_{12}^{-1}B_{22}^{(n)}(z_{12}^{-1})) 
+ s^{e_1}B'_2(0)c_pv^{e_1}z_{12}, 
\end{eqnarray*}
and that both sides are extendable to $z_{12}=1$. 

From the regularity at $z_{12}=1$, 
we see that $B_{i2}^{(n)}(z_{i2})$ are polynomials, 
and comparing the coefficients, we obtain the assertion. 

For the other implication, note that $B_1(0)=B_2(0)$ is assumed. 
\end{proof}

We have a natural trivialization 
$\cC_i^\circ\cong C_i^\circ\times S\cong (\mathbb{P}^1\setminus\{0\})\times S$ 
given by the functions $z_{i0}$ and $z_{i1}$. 
Then the morphism $\cC_i^\circ\cap \cV\to X$ 
by data $\tilde{A}_i$ and $\tilde{B}_{ji}$ 
can be considered as a deformation of $f^{(i)}|_{C_i^\circ\cap V}$. 
Since it is trivial modulo $s^n$ and $(s^n)^2=0$, 
we obtain a vector field $\bv_i\in\Gamma(C_i^\circ\cap V, (f^{(i)})^*T_X)$ 
to $X$ along $f^{(i)}|_{C_i^\circ\cap V}$. 
Specifically, if $\varphi$ is a local regular function on $X$ 
we can write 
$(\underline{f}|_{\cC_i^\circ})^*\varphi = (f^{(i)})^*\varphi+s^n D(\varphi)$. 
Then it is easy to check that $D$ defines a derivation 
$(f^{(i)}|_{C_i^\circ\cap V})^{-1}\cOX\to \cO_{C_i^\circ\cap V}$, 
hence a section $\bv_i\in\Gamma(C_i^\circ\cap V, (f^{(i)})^*T_X)$. 

By the same reasoning, 
if the condition (c) of Lemma \ref{lem_conditions} holds, 
then $\bv_i$ extends to a section in $\Gamma(C_i^\circ, (f^{(i)})^*T_X)$. 

Let us explicitly write down $\bv_i$. 
On $\cC_1^\circ\cap \cV$, 
from $s^{n+1}=0$, $z_{12}=z_{11}^{-1}v^{e_2}s^{e_2}$ and $e_2\geq n$ 
we have
\[
z_{11}^{d_1}\tilde{A}_1 = 
z_{11}^{d_1}
A_1(z_{11}) 
+ s^n z_{11}^{d_1}(A_{10}^{(n)}+z_{11}A_{11}^{(n)}(z_{11})) 
+ s^{e_2} z_{11}^{d_1-1}(-(d_1+d_2)A_1(0))v^{e_2}
\]
and 
\[
\tilde{B}_1 = 
B_1(z_{11}) + s^n (B_{10}^{(n)} + z_{11}B_{11}^{(n)}(z_{11})). 
\]

Thus we have
\[
\bv_1=z_{11}^{d_1}(A_{10}^{(n)}+z_{11}A_{11}^{(n)}(z_{11}))\partial_{w_1}
+ 
(B_{10}^{(n)} + z_{11}B_{11}^{(n)}(z_{11}))\partial_{w_2}
\]
if $n<e_2$ and
\[
\bv_1=z_{11}^{d_1}(z_{11}^{-1}(-(d_1+d_2)A_1(0))v^n+A_{10}^{(n)}+z_{11}A_{11}^{(n)}(z_{11}))\partial_{w_1}
+ 
(B_{10}^{(n)} + z_{11}B_{11}^{(n)}(z_{11}))\partial_{w_2}
\]
if $n=e_2$ (i.e. $n=e_1=e_2=1$). 

Similarly we have 
\[
\bv_2=z_{21}^{d_2}(A_{20}^{(n)}+z_{21}A_{21}^{(n)}(z_{21}))\partial_{w_1}
+ 
(B_{20}^{(n)} + z_{21}B_{21}^{(n)}(z_{21}))\partial_{w_2} 
\]
if $n<e_1$ and 
\[
\bv_2=z_{21}^{d_2}(z_{21}^{-1}(-(d_1+d_2)A_2(0))c_p v^n+A_{20}^{(n)}+z_{21}A_{21}^{(n)}(z_{21}))\partial_{w_1}
+ 
(B_{20}^{(n)} + z_{21}B_{21}^{(n)}(z_{21}))\partial_{w_2} 
\]
if $n=e_1$. 

Recall that we are making the same assumptions as in Theorem \ref{theorem_main}. 
\begin{Lemma}
Let $\cE_i$ be the $\cO_{C_i}$-submodule of $(f^{(i)})^*T_X$ 
generated by $(f^{(i)})^*T_X(-\log D)$ and $T_{C_i}$. 
\begin{enumerate}
\item
The sheaf $\cE_i$ can also be described as the $\cO_{C_i}$-submodule of $(f^{(i)})^*T_X$
generated by $(f^{(i)})^*T_X(-\log D)$ and $z_{i1}^{d_i-1}\partial_{w_1}$ at $R_i$. 
\item
If the condition (c) of Lemma \ref{lem_conditions} holds, then  
$\bv_i$ extends to a global section of $\cE_i$. 
\item
There is a commutative diagram with exact rows and columns as follows. 
\[
\xymatrix{
 & 0 \ar[d] & 0 \ar[d] & & \\
0 \ar[r] & T_{C_i}(-\log R_i) \ar[r]\ar[d] & 
(f^{(i)})^*T_X(-\log D)\ar[r]\ar[d] & \cO_{C_i}(-1) \ar[r]\ar@{=}[d] & 0 \\
0 \ar[r] & T_{C_i} \ar[r]\ar[d] & \cE_i \ar[r]\ar[d] & \cO_{C_i}(-1) \ar[r] & 0 \\
 & \CC \ar@{=}[r]\ar[d] & \CC\ar[d] & & \\
 & 0 & 0 & & \\
}
\]
\item
The natural map 
$\Gamma(T_{C_i})\to\Gamma(\cE_i)$ is an isomorphism. 
\end{enumerate}
\end{Lemma}
\begin{proof}
(1) 
Direct calculations. 

(2)
As explained right after the definition of $\bv_i$, 
if we assume Lemma \ref{lem_conditions} (c), 
then $\bv_i$ extends to a section of $(f^{(i)})^*T_X$ over $C_i^\circ$. 
Note that $(f^{(i)})^*T_X|_{C_i^\circ}=\cE_i|_{C_i^\circ}$, 
as can be seen from 
$(f^{(i)})^*T_X(-\log D)\subseteq \cE_i\subseteq (f^{(i)})^*T_X$. 

We see that $\bv_i$ also belongs to $(\cE_i)_{R_i}$ 
from our explicit description of $\bv_i$ and (1), 
and we have the assertion. 

(3)
The first homomorphism in the first row 
is injective with an invertible cokernel 
since $f^{(i)}$ is a log map which is immersive 
(although the immersivity at $R_i$ is not needed \emph{here}). 
From $(K_X+D).Z_i=0$ we see that the cokernel is $\cO_{C_i}(-1)$. 
Then the assertion is easy to prove. 

(4) follows from (3). 
\end{proof}

Thus we have global vector fields 
$a_1(z_{11})\partial_{z_{11}}$ and $a_2(z_{21})\partial_{z_{21}}$ 
on $C_1$ and $C_2$ respectively 
such that 
\[
\bv_1=a_1(z_{11})(f^{(1)})_*\partial_{z_{11}}, 
\bv_2=a_2(z_{21})(f^{(2)})_*\partial_{z_{21}}. 
\]
Comparing at $R_i$, we have 
\begin{equation}\label{eq_a_1}
a_1(0) = 
\begin{cases} 
0 & \text{if $n<e_2$},  \\
-\frac{d_1+d_2}{d_1}v^n & \text{if $n=e_2$ (i.e. $n=e_1=e_2=1$)},  
\end{cases}
\end{equation}
\begin{equation}\label{eq_a_2}
a_2(0) = 
\begin{cases} 
0 & \text{if $n<e_1$}, \\
-\frac{d_1+d_2}{d_2}c_pv^n & \text{if $n=e_1$},  
\end{cases}
\end{equation}
\begin{eqnarray*}
a_1(0)B'_1(0) & = & B_{10}^{(n)}, \\
a_2(0)B'_2(0) & = & B_{20}^{(n)}. 
\end{eqnarray*}
We also have $B_{10}^{(n)}=B_{20}^{(n)}$ by Lemma \ref{lem_b}. 

\begin{Lemma}\label{416}
If $n=e_1$, then $v=0$. 
\end{Lemma}
\begin{proof}
If $n=e_1=e_2$, then they are all $1$. 
By $(Z_1.Z_2)_P=d$ we have $A_1(0)(B'_2(0))^d\not=A_2(0)(B'_1(0))^d$. 
Also, $c_p^d=A_1(0)/A_2(0)$ by definition. 
From these we obtain $v=0$. 

If $n=e_1<e_2$, then we have 
$a_1(0)=0$, and then $v=0$ 
as we have $B'_2(0)\not=0$ (since $d_2\geq 2$ and $C_2$ is smooth at the point of intersection $P$). 
\end{proof}

\begin{proof}[Proof of Lemma \ref{413} (C)]
We are considering the case $n=e_1\geq 2$, 
so $v=m^{(1)}=1$ by our setup. 
But we also have $v=0$ by Lemma \ref{416}, which yields a contradiction.
\end{proof}

\begin{Lemma}\label{417}
The deformation of the underlying prestable curve is trivial, 
$B_{12}^{(n)}(z_{12})=B_{22}^{(n)}(z_{22})=0$, 
and $a_1(0)=a_2(0)=0$. 
\end{Lemma}
\begin{proof}
If $n=e_1$, then $v=0$ by Lemma \ref{416}, and 
from equations \eqref{eq_a_1} and \eqref{eq_a_2} 
we have $a_1(0)=a_2(0)=0$ in any case. 
Similarly, from Lemma \ref{lem_b} it follows that 
$B_{12}^{(n)}(z_{12})=B_{22}^{(n)}(z_{22})=0$. 

The deformation parameters of the underlying curve are 
\[
m(s)^{e_2}u_1(s)=v^{e_2}s^{e_2}=0 \text{ and }
m(s)^{e_1}u_2(s)=c_p v^{e_1}s^{e_1}=0, 
\]
for $v=0$ if $n=e_1$ and $s^{e_1}=s^{e_2}=0$ if $n<e_1$. 
So the deformation of the curve is trivial. 
\end{proof}

Since $a_i(0)=0$, 
$a_1(z_{11})\partial_{z_{11}}$ and $a_2(z_{21})\partial_{z_{21}}$ 
define an automorphism of $\underline{\cC}_{S_n}$ over $S_n$, 
and by untwisting 
we may assume that the underlying stable map is 
a trivial deformation, i.e. $\bv_1=\bv_2=0$, 
or more explicitly, 
\[
A_{10}^{(n)}=A_{11}^{(n)}(z_{11}) = 
B_{10}^{(n)} = B_{11}^{(n)}(z_{11}) = 
A_{20}^{(n)}=A_{21}^{(n)}(z_{21}) = 
B_{20}^{(n)} = B_{21}^{(n)}(z_{21}) = 0. 
\]

\begin{Lemma}\label{418}
Assuming $A_{10}^{(n)}=A_{20}^{(n)}=0$, 
condition (a) of Lemma \ref{lem_conditions} implies 
$A_{12}^{(n)}(z_{12})=A_{22}^{(n)}(z_{22})=0$ and 
$d_1u_1^{(n)} = d_2c_p^{-1}u_2^{(n)}$. 
\end{Lemma}
\begin{proof}
Again noting that either $v=0$ or $s^{e_1}=s^{e_2}=0$, 
and also that $z_{11}=z_{21}=0$ on $\cV_0$, 
we have 
\begin{eqnarray*}
z_{12}^{-d_1} u_1(s)^{d_1}\tilde{A}_1 & = & 
A_1(0) z_{12}^{-d_1}(1-z_{12})^{d_1+d_2} \\
& & 
+ s^n z_{12}^{-d_1}(A_{10}^{(n)}+z_{12}A_{12}^{(n)}(z_{12}) 
+ d_1 A_1(0) u_1^{(n)}(1-z_{12})^{d_1+d_2}) 
\end{eqnarray*}
and
\begin{eqnarray*}
z_{22}^{-d_2} u_2(s)^{d_2}\tilde{A}_2 & = & 
A_2(0) c_p^{d_2} z_{22}^{-d_2}(1-z_{22})^{d_1+d_2} \\
& & 
+ s^n c_p^{d_2} z_{22}^{-d_2}(A_{20}^{(n)}+z_{22}A_{22}^{(n)}(z_{22}) 
+ d_2 A_2(0) c_p^{-1} u_2^{(n)}(1-z_{22})^{d_1+d_2}) 
\end{eqnarray*}
on $\cV_0\setminus\{x_1\}$. 

By the regularity condition at $z_{12}=1$ from Lemma \ref{lem_conditions} (a), 
$A_{i2}^{(n)}(z_{i2})$ are polynomials. 
The equality in Lemma \ref{lem_conditions} (a) is equivalent to 
\begin{eqnarray*}
& & A_{10}^{(n)}+z_{12}A_{12}^{(n)}(z_{12}) 
+ d_1 A_1(0) u_1^{(n)}(1-z_{12})^{d_1+d_2} \\
& = & 
c_p^{d_2} z_{12}^{d_1+d_2}(A_{20}^{(n)}+z_{12}^{-1}A_{22}^{(n)}(z_{12}^{-1}))
+ d_2 A_2(0) c_p^{d_2-1} u_2^{(n)}(z_{12}-1)^{d_1+d_2}, 
\end{eqnarray*}
and from this we see that $\deg (A_{10}^{(n)}+z_{12}A_{12}^{(n)}(z_{12}))\leq d_1+d_2$. 
Similarly, $\deg (A_{20}^{(n)}+z_{22}A_{22}^{(n)}(z_{22}))\leq d_1+d_2$. 
From the condition on the vanishing order at $\tx_1$, 
we see that 
$A_{10}^{(n)}+z_{12}A_{12}^{(n)}(z_{12})=A_{10}^{(n)}(1-z_{12})^{d_1+d_2}$, 
$A_{20}^{(n)}+z_{22}A_{22}^{(n)}(z_{22})=A_{20}^{(n)}(1-z_{22})^{d_1+d_2}$, 
with 
\[
A_{10}^{(n)} + d_1A_1(0)u_1^{(n)} = 
(-1)^{d_1+d_2}c_p^{d_2}(A_{20}^{(n)} + d_2c_p^{-1}A_2(0)u_2^{(n)}). 
\]
But $A_{10}^{(n)}=A_{20}^{(n)}=0$, so that
$A_{12}^{(n)}(z_{12}) = A_{22}^{(n)}(z_{22}) = 0$
and thus \\ $d_1A_1(0)u_1^{(n)} = (-1)^{d_1+d_2}c_p^{d_2} d_2c_p^{-1}A_2(0)u_2^{(n)}$, 
i.e.~$d_1u_1^{(n)} = d_2c_p^{-1}u_2^{(n)}$. 
\end{proof}

\begin{proof}[Proof of Lemma \ref{413} (A) and (B)]
By Lemma \ref{417}, the curve $\cC_{S_n}/S_n$ is trivial and we may assume that 
\begin{eqnarray*}
& A_{10}^{(n)} = A_{11}^{(n)}(z_{11}) = A_{12}^{(n)}(z_{12}) = 0, & \\
& B_{10}^{(n)} = B_{11}^{(n)}(z_{11}) = B_{12}^{(n)}(z_{12}) = 0, & \\
& A_{20}^{(n)} = A_{21}^{(n)}(z_{21}) = A_{22}^{(n)}(z_{22}) = 0, & \\
& B_{20}^{(n)} = B_{21}^{(n)}(z_{21}) = B_{22}^{(n)}(z_{22}) = 0, & 
\end{eqnarray*}
i.e. $\tilde{A}_i$ and $\tilde{B_i}$ are constant with respect to $s$. 

Let $\boldm' := \boldm(1+(u_1^{(n)}/e_2)s^n)$, 
then we have 
\[
\alpha(\boldm')=\alpha(\boldm)\left(1+\frac{u_1^{(n)}}{e_2}s^n\right)= \alpha(\boldm) 
\]
and 
\[
\boldm^{e_2} u_1(s) = (\boldm')^{e_2}. 
\]
So, if we replace $\boldm$ by $\boldm'$, 
we see that $m(s)$ does not change and that $u_1^{(n)} = 0$. 
By Lemma \ref{418}, $u_2^{(n)}=0$. 
Since $\tilde{A}_i$ and $\tilde{B}_i$ do not depend on $\boldm$, 
they do not change either. 

If $n=e_1=1$, we also have $v=0$, and the deformation is trivial. 

If $n=1<e_1$, this says that the only deformation parameter is $v$. 
\end{proof}

\section{Example: Comparison between stable log maps and relative stable maps}

\label{sec:exx}

\begin{Example}\label{ex1}

We illustrate by an example the difference in our setting between the moduli space of stable log maps and the relative maps of \cite{Lirel1,Lirel2}. For this example the divisor is normal crossing, so the more appropriate moduli space is the one of \cite{Ran19}, but since the calculations are local, it is not relevant.

Start with $\mathbb{P}^2$ with its toric boundary $D=L_1\cup L_2\cup L_3$. Choose $P\in L_1$ a smooth point of $D$. Choose some local coordinates $w_1,w_2$ around $P$ such that $L_1$ is given by $w_1=0$ and $P=(0,0)$. We consider $Z'_1$ and $Z'_2$ two general smooth conics tangent to $L_1$ at $P$. For example, consider the conics given near $P$ by
\begin{align*}
Z'_1 &: \quad \, \, \, (w_1-1)^2 + w_2^2 - 1=0, \\
Z'_2 &: \quad 2(w_1-2)^2 + w_2^2 - 8=0.
\end{align*}
In particular, $(Z'_1.Z'_2)_P=2$ and $Z'_1$ and $Z'_2$ meet in 2 other points. Also, we may assume that $Z'_1\cup Z'_2$ is disjoint from the torus-fixed points. Write $P_1,\dots,P_8$ for the distinct points of intersection of $(Z'_1\cup Z'_2) \setminus P$ with $L_2\cup L_3$.

Let $X$ be the blowup of $\mathbb{P}^2$ in $P_1,\dots,P_8$ with exceptional divisors $E_1,\dots,E_8$ and by abusing notation slightly we denote by $D$ the proper transform of $L_1\cup L_2\cup L_3$. Let $Z_1$ and $Z_2$ be the proper transforms of $Z'_1$ and $Z'_2$. Then $Z_1$ and $Z_2$ satisfy the hypotheses of Theorem \ref{theorem_main} for $(X,D)$ and $Z_1+Z_2$ is in the class $\beta=4H-\sum_{i=1}^8 E_i$, where $H$ is the pullback of the hyperplane class.

Consider the chain of $\mathbb{P}^1$s given by $C=C_1\cup_{R_1} C_0 \cup_{R_2} C_2$ with $R_i$ nodes with local equations $z_{11}z_{12}=0$ and $z_{21}z_{22}=0$, where $z_{i1}$ are local coordinates for $C_i$ ($i=1, 2$) and $z_{i2}$ are inhomogeneous coordinates on $C_0$ satisfying $z_{12}z_{22}=1$ on $C_0\setminus\{R_1,R_2\}$.  The nodes $R_1$ and $R_2$ are the points at infinity of the projective completion of the affine curve   $z_{12}z_{22}=1$.
Choosing parametrizations of $Z_1$ and $Z_2$, there is only one $1$-marked stable map $\underline{f} : C_1\cup C_0 \cup C_2\to X$ with $\underline{f}_*C=Z_1+Z_2$ which underlies a stable log map of maximal tangency. Locally near $P$ and for $f_i:=\underline{f}|_{C_i}$,
\begin{align*}
f_1(z_{11}) &= \left(\frac{2z_{11}^2}{z_{11}^2+1},\frac{2z_{11}}{z_{11}^2+1}\right), \\
f_2(z_{21}) &= \left(\frac{8z_{21}^2}{2z_{21}^2+1},\frac{8z_{21}}{2z_{21}^2+1}\right),
\end{align*}
and $C_0$ is mapped to $P$. 

By Theorem \ref{thm:412}, we have two nonisomorphic stable log maps over $\underline{f}$ each of which has multiplicity $1$ in $\barM_\beta$. In what follows, we illustrate the log structure assuming $n=0$. 

Pulling back along $f_i$ as in Notation \ref{notation_central_map},
\[
w_1 \, \mapsto \, z_{i1}^{2} A_i(z_{i1}), 
\]
so that we have $A_1(z_{11})=2/(z_{11}^2+1)$ and $A_2(z_{21})=8/(2z_{21}^2+1)$ with $A_1(0)/A_2(0)=1/4$.

The base is $S=\Spec\CC$ with log structure $\cM_S = \coprod_{a=0}^\infty \boldm^a\CC^\times$, 
with $\alpha(\boldm^a)=0^a$.
At each node of $C$, the log structure $\shM_{C,R_i}$ is 
$
\coprod_{a, b,c=0}^\infty
\bz_{i1}^a\bz_{i2}^b\boldm^c\cO_{C, R_i}^\times
$
subject to the relation $\bz_{i1}\bz_{i2}=\boldm u_i$ with $\alpha(\bz_{ij})=z_{ij}$. Up to isomorphism, we may set $u_1=1$.

Near $P$, let $\bw_1$ be the unique lift of $w_1$ to $\shM_X$. Recall that $x_1$ is the marked point. By Theorem \ref{thm:412}, near $R_i$ away from $x_1$, we have
\begin{align*}
f_1^\flat (\bw) &= \bz_{11}^2 \left( \frac{2}{z_{11}^2+1} +  2((1-z_{12})^{4}-1) \right), \\
f_2^\flat (\bw) &= \bz_{21}^2 \left( \frac{8}{2z_{21}^2+1} + 8((1-z_{22})^{4}-1)\right).
\end{align*}
By the generization maps to $C_0$, sending $\bz_{i1}$ to $\boldm u_i z_{i2}^{-1}$ respectively, they have to agree on $C_0\setminus \{R_1,R_2,x_1\}$. Thus, 
\[ 
\boldm^{2}z_{12}^{-2} \left( 2((1-z_{12})^{4}) \right) = \boldm^{2}u_2^2 z_{22}^{-1} \left( 8(1-z_{22})^{4} \right)
\]
as $z_{i1}=0$ on $C_0\setminus \{R_1,R_2,x_1\}$. Since $z_{12}z_{22}=1$, we conclude $u_2=\pm 1/2$. Therefore we have two stable log maps. 
\end{Example}

\begin{Example}\label{ex2}
Let us look at the relative side. 
Since we work locally near $Z_1\cup Z_2$, 
we may think of $D$ as smooth.  Before proceeding with the example, we quickly recall some of the notation and ideas of the theory of relative stable maps, referring the reader to \cite{Lirel1} for more detail.

In the theory of relative stable maps to $X$, the target space $X$ is allowed to deform in a specific family over $\bA^n$ to singular limits $X[n]$.  The central fiber $X[n]_0$ is a chain of $n+1$ smooth components obtained by gluing $X$ to $\bP^1$-bundles $\bP_1,\ldots,\bP_n$ over $D$.  The gluings are along $D\subset X$ and sections of the $\bP_j$.  There are $n$ singular divisors $D_1,\ldots, D_n$ in $X[n]$.  The divisor $D_1$ is the intersection of $X$ and $\bP_1$, while $D_k$ is the intersection of $\bP_{k-1}$ and $\bP_k$ for $k>1$.  We will describe $X[1]_0$ in more detail below.

If $f:C\to X[n]_0$ is a relative stable map, then no component of $C$ is mapped into any $D_k$.  Furthermore, there is the \emph{predeformability\/} condition: if $p\in f^{-1}(D_k)$, then $p$ is a node of $C$.  One component $C_i$ of $C$ containing $p$ satisfies $f(C_i)\subset \bP_{k-1}$, while another such component $C_j$ satisfies $f(C_j)\subset \bP_k$, with an obvious modification if $k=1$.  Then the predeformability requirement is that $\mathrm{mult}_p(f|_{C_i}^*(D_k))=\mathrm{mult}_p(f|_{C_j}^*(D_k))$.
For the predeformability of a family of maps, 
we require a little more 
(\cite[Definition 2.9, Lemma 2.4, Definition 2.3]{Lirel1}). 
In the case at hand, we explain the condition later. 

In the current case, 
it turns out to be sufficient to consider $X[1]_0$, 
the central fiber of the blow-up of $X\times \mathbb{A}^1$ along $D\times \{0\}$. 
This is also obtained by gluing $X$ 
and the $\mathbb{P}^1$-bundle $\mathbb{P}(N_{D/X}\oplus\mathcal{O}_D)$ over $D$. 
 
We can describe $X[1]_0$ explicitly as follows. 
Let $X_0^\prime=\mathbb{A}^2\setminus (L_2\cup L_3)$, where $L_2\cup L_3$ is the union of two distinct lines that will correspond to $L_2\cup L_3\subset\bP^2$. Consider $\mathbb{A}^3$ with coordinates $w_1, w_2, v_1$ 
and projection $p: \mathbb{A}^3\to \mathbb{A}^2$ given by $w_1, w_2$, 
let $X'=p^{-1}(X_0^\prime)\cap V(w_1v_1)$ 
and $U=(L_1\cap X_0^\prime)\times\mathbb{A}^1\subset \mathbb{A}^2$ 
with coordinates $w_2, u_1$. 
Then we may think of $X[1]_0$ 
as covered by open sets $X'$, $U$ and $X\setminus D$: 
The open sets $(v_1\not=0)\subset X'$ and $(u_1\not=0)\subset U$ 
are identified by $u_1v_1=1$, 
and $(w_1\not=0)\subset X'$ is identified with an open subset of $X\setminus D$ 
in a natural way. 

We have a projection map $\pi: X[1]_0\to X$ with $\pi|_{X'}=p|_{X'}$, 
which identifies $(v_1=0)\subset X'$ with $X_0^\prime$ 
and contracts $U$ to $D$. 
The divisor $(u_1=0)\subset U$ will be considered as the boundary divisor. 

Let $S=\mathrm{Spec}\ \mathbb{C}[s]/(s^2)$. 
We will define a nontrivial family of relative stable maps $\mathcal{C}_S\to X[1]_0$ 
whose central fiber, composed with $\pi$, gives $\underline{f}$ in the previous example. 
To do this, 
we take $\mathcal{C}_S$ to be the family of curves 
defined by $\mu_1=s$ and $\mu_2=s/4$ 
(hence $z_{11}z_{12}=s$ and $z_{21}z_{22}=s/4$) 
and give a stable map $\mathcal{C}_S\to X[1]_0$ over $S$ such that 
the following hold: 
\begin{itemize}
\item (pullback of the boundary)
The pullback of $(u_1=0)$ is $4\cdot \tilde{x}_1$. 
\item (predeformability)
Near $R_i$, one can write 
$w_1=z_{i1}^2\tilde{A}_i$, 
$v_1=z_{i2}^2\tilde{B}_i$ 
with $\tilde{A}_i, \tilde{B}_i$ invertible and $\tilde{A}_i\tilde{B}_i\in \mathbb{C}[s]/(s^2)$. 
\end{itemize}
Once such a family of relative stable maps is obtained, 
it is nontrivial since $\mathcal{C}_S$ is a nontrivial deformation. 

We set 
\begin{eqnarray*}
\tilde{A}_1 = \frac{2(1-z_{12})^4}{z_{11}^2+1}, \qquad
\tilde{B}_1  =  \frac{z_{11}^2+1}{(1-z_{12})^4}, \\
\tilde{A}_2  = \frac{8(1-z_{22})^4}{2z_{21}^2+1}, \qquad 
\tilde{B}_2  =  \frac{2z_{21}^2+1}{(1-z_{22})^4}, \\
w_2 =  \frac{2z_{11}+s(4z_{11}^2-4+2z_{12})}{z_{11}^2+1} \hbox{ near $R_1$}, \\
w_2  = \frac{8z_{21}+s(8z_{21}^2-4+2z_{22})}{2z_{21}^2+1} \hbox{ near $R_2$}. 
\end{eqnarray*}
These data give a stable map $\mathcal{C}_S\to X[1]_0$. 
In fact, 
let $\cV\subset \mathcal{C}_S$ be the open subscheme 
supported on $\underline{f}^{-1}(X_0^\prime)\setminus\{x_1\}$. 
We first see that the two sets of $w_1, v_1$ 
(defined from $\tilde{A}_i$ and $\tilde{B}_i$ as above) 
and $w_2$ coincide on $\mathcal{C}_0\setminus\{R_1, R_2, x_1\}$ 
and give a morphism $\cV\to\bA^3$: 
On $\mathcal{C}_0\setminus\{R_1, R_2, x_1\}$, 
$z_{11}^2\tilde{A}_1= z_{21}^2\tilde{A}_2 = 0$, 
$z_{12}^2\tilde{B}_1 = z_{12}^2/(1-z_{12})^4$ 
is equal to $z_{22}^2\tilde{B}_2=z_{22}^2/(1-z_{22})^4$, 
and the two expressions for $w_2$ are 
\[
\frac{2z_{11}+s(4z_{11}^2-4+2z_{12})}{z_{11}^2+1}  
= \frac{2sz_{12}^{-1}+s(4(sz_{12}^{-1})^2-4+2z_{12})}{(sz_{12}^{-1})^2+1}
= s(2 z_{12}^{-1}-4+2z_{12})
\]
and 
\[
\frac{8z_{21}+s(8z_{21}^2-4+2z_{22})}{2z_{21}^2+1}
= \frac{8(s/4)z_{22}^{-1}+s(8((s/4)z_{22}^{-1})^2-4+2z_{22})}{2((s/4)z_{22}^{-1})^2+1}
= s(2 z_{22}^{-1}-4+2z_{22}), 
\]
which are equal. 
Note that we have $w_2|_{C_0}=0$ by reduction modulo $s$, as it should be. 

Since $w_1v_1=z_{i1}^2z_{i2}^2\tilde{A}_i\tilde{B}_i=0$, 
we have a morphism $\tilde{f}': \cV\to X'$. 

On $\cV\cap (\mathcal{C}_1\setminus R_1)$, 
\[
(w_1, w_2)=
\left(\frac{2z_{11}^2-8sz_{11}}{z_{11}^2+1}, \frac{2z_{11}+4s(z_{11}^2-1)}{z_{11}^2+1}\right), 
\]
and this is 
the reparametrization of $f_1$ by $z_{11}\mapsto z_{11}-2s(z_{11}^2+1)$. 
Similarly, the restriction of $\tilde{f}'$ to $\cV\cap (\mathcal{C}_2\setminus R_2)$ 
is given by 
the reparametrization of $f_2$ by $z_{21}\mapsto z_{21}-(1/2)s(2z_{21}^2+1)$. 
This implies that $\tilde{f}'$ extends to a morphism 
$\mathcal{C}_S\setminus \{x_1\}\to X[1]_0$. 
(We may also verify that $w_1$ and $w_2$ satisfy the equation for $Z'_i$.) 

On $\mathcal{C}_0\setminus \{R_1, R_2\}$, 
we have $u_1=(v_1)^{-1}=(1-z_{12})^4/z_{12}^2$, 
so $\tilde{f}'$ extends to a morphism $\mathcal{C}_S\to X[1]_0$, 
and the condition on the pullback of the boundary is satisfied. 
Predeformability also holds since $\tilde{A}_1\tilde{B}_1=2$ and $\tilde{A}_2\tilde{B}_2=8$. 

It follows from the main result of \cite{Tak17} 
that the corresponding point of the moduli space of relative stable map
is actually isomorphic to $S=\mathrm{Spec}\ \mathbb{C}[s]/(s^2)$, 
in contrast to the $2$ reduced points in the moduli space 
of basic stable log maps. 

\end{Example}

\bibliographystyle{gtart}
\bibliography{biblio}

\def\cprime{$'$} \def\cprime{$'$}
\begin{thebibliography}{}
\providecommand\bibmarginpar{\leavevmode\marginpar}
\def\urlstyle#1{{\tt #1}}

\bibitem{AbramChen14}
\textbf{D Abramovich}, \textbf{Q Chen},
\emph{Stable logarithmic
  maps to {D}eligne-{F}altings pairs {II}}, Asian J. Math. 18 (2014) 465--488

\bibitem{AF16}
\textbf{D Abramovich}, \textbf{B Fantechi}, \emph{Orbifold techniques in
  degeneration formulas}, Ann. Sc. Norm. Super. Pisa Cl. Sci. (5) 16 (2016)
  519--579

\bibitem{AMW}
\textbf{D Abramovich}, \textbf{S Marcus}, \textbf{J Wise},
  \emph{Comparison theorems for {G}romov-{W}itten invariants of smooth pairs
  and of degenerations}, Ann. Inst. Fourier (Grenoble) 64 (2014) 1611--1667

\bibitem{AIK77}
\textbf{A\,B Altman}, \textbf{A Iarrobino}, \textbf{S\,L Kleiman},
  \emph{Irreducibility of the compactified {J}acobian}, from ``Real and complex
  singularities ({P}roc. {N}inth {N}ordic {S}ummer {S}chool/{NAVF} {S}ympos.
  {M}ath., {O}slo, 1976)'', Sijthoff and Noordhoff, Alphen aan den Rijn (1977)
  1--12

\bibitem{AM93}
\textbf{P\,S Aspinwall}, \textbf{D\,R Morrison},
  \emph{Topological field theory and rational curves}, Comm. Math. Phys. 151
  (1993) 245--262

\bibitem{BN20}
\textbf{L Barrott}, \textbf{N Nabijou}, \emph{Tangent curves to degenerating
  hypersurfaces} \href{https://arxiv.org/abs/2007.05016}{arXiv:2007.05016},
  2020

\bibitem{Beau99}
\textbf{A Beauville},
  \emph{Counting rational curves on {$K3$} surfaces}, Duke Math. J. 97 (1999)
  99--108

\bibitem{Bou19b}
\textbf{P Bousseau}, \emph{A proof of {N}. {T}akahashi's conjecture on genus
  zero {G}romov-{W}itten theory of $(\mathbb{P}^2,{E})$ and a refined
  sheaves/{G}romov--{W}itten correspondence}
  \href{https://arxiv.org/abs/1909.02992}{arXiv:1909.02992}, 2019

\bibitem{Bou19a}
\textbf{P Bousseau}, \emph{Scattering diagrams, stability conditions, and
  coherent sheaves on $\mathbb{P}^2$}
  \href{https://arxiv.org/abs/1909.02985}{arXiv:1909.02985}, 2019

\bibitem{BBvG1}
\textbf{P Bousseau}, \textbf{A Brini}, \textbf{M van Garrel}, \emph{On the
  log-local principle for the toric boundary}
  \href{https://arxiv.org/abs/1908.04371}{arXiv:1908.04371}, 2019

\bibitem{BBvG2}
\textbf{P Bousseau}, \textbf{A Brini}, \textbf{M van Garrel}, \emph{Stable maps
  to Looijenga pairs}
  \href{https://arxiv.org/abs/2011.08830}{arXiv:2011.08830}, 2020

\bibitem{BBvG3}
\textbf{P Bousseau}, \textbf{A Brini}, \textbf{M van Garrel}, \emph{Stable maps
  to Looijenga pairs: orbifold examples}
  \href{https://arxiv.org/abs/2012.10353}{arXiv:2012.10353}, 2020

\bibitem{Bri06}
\textbf{T Bridgeland}, 
  \emph{Stability conditions on a non-compact {C}alabi-{Y}au threefold},
  Comm. Math. Phys. 266 (2006) 715--733

\bibitem{CCE08}
\textbf{L Caporaso}, \textbf{J Coelho}, \textbf{E Esteves},
\emph{Abel maps of
  {G}orenstein curves}, Rend. Circ. Mat. Palermo (2) 57 (2008) 33--59

\bibitem{Catanese1982}
\textbf{F Catanese}, \emph{Pluricanonical-{G}orenstein-curves}, from
  ``Enumerative geometry and classical algebraic geometry ({N}ice, 1981)'',
  Progr. Math. 24, Birkh\"{a}user Boston, Boston, MA (1982)  51--95

\bibitem{Chen14deg}
\textbf{Q Chen},
  \emph{The degeneration formula for logarithmic expanded degenerations}, J.
  Algebraic Geom. 23 (2014) 341--392

\bibitem{Chen14}
\textbf{Q Chen}, 
  \emph{Stable logarithmic maps to {D}eligne-{F}altings pairs {I}}, Ann. of
  Math. (2) 180 (2014) 455--521

\bibitem{Xi99}
\textbf{X Chen}, \emph{Rational curves on {$K3$} surfaces}, J. Algebraic Geom.
  8 (1999) 245--278

\bibitem{CKYZ99}
\textbf{T-M Chiang}, \textbf{A Klemm}, \textbf{S-T Yau}, \textbf{E Zaslow},
  \emph{Local mirror
  symmetry: calculations and interpretations}, Adv. Theor. Math. Phys. 3
  (1999) 495--565

\bibitem{CGKT1}
\textbf{J Choi}, \textbf{M van Garrel}, \textbf{S Katz}, \textbf{N Takahashi},
  \emph{Local {BPS} invariants:
  enumerative aspects and wall-crossing}, Int. Math. Res. Not. IMRN  (2020)
  5450--5475

\bibitem{CGKT2}
\textbf{J Choi}, \textbf{M van Garrel}, \textbf{S Katz}, \textbf{N Takahashi},
  \emph{Log {BPS} numbers of log
  {C}alabi-{Y}au surfaces}, Trans. Amer. Math. Soc. 374 (2021) 687--732

\bibitem{FGS99}
\textbf{B Fantechi}, \textbf{L G\"{o}ttsche}, \textbf{D van Straten},
  \emph{Euler number of the compactified {J}acobian and multiplicity of
  rational curves}, J. Algebraic Geom. 8 (1999) 115--133

\bibitem{vGGR}
\textbf{M van Garrel}, \textbf{T Graber}, \textbf{H Ruddat},
  \emph{Local
  {G}romov-{W}itten invariants are log invariants}, Adv. Math. 350 (2019)
  860--876

\bibitem{vGOR}
\textbf{M van Garrel}, \textbf{D\,P Overholser}, \textbf{H Ruddat},
  \emph{Enumerative
  aspects of the {G}ross-{S}iebert program}, from ``Calabi-{Y}au varieties:
  arithmetic, geometry and physics'', Fields Inst. Monogr. 34, Fields Inst.
  Res. Math. Sci., Toronto, ON (2015)  337--420

\bibitem{vGWZ13}
\textbf{M van Garrel}, \textbf{T\,W\,H Wong}, \textbf{G Zaimi},
  \emph{Integrality of
  relative {BPS} state counts of toric del {P}ezzo surfaces}, Commun. Number
  Theory Phys. 7 (2013) 671--687

\bibitem{Gra20}
\textbf{T Gr\"afnitz}, \emph{Tropical correspondence for smooth del Pezzo log
  Calabi-Yau pairs} \href{https://arxiv.org/abs/2005.14018}{arXiv:2005.14018},
  2020

\bibitem{GPS10}
\textbf{M Gross}, \textbf{R Pandharipande}, \textbf{B Siebert},
  \emph{The tropical
  vertex}, Duke Math. J. 153 (2010) 297--362

\bibitem{GS13}
\textbf{M Gross}, \textbf{B Siebert},
  \emph{Logarithmic
  {G}romov-{W}itten invariants}, J. Amer. Math. Soc. 26 (2013) 451--510

\bibitem{HL}
\textbf{D Huybrechts}, \textbf{M Lehn},
\emph{The Geometry of
  Moduli Spaces of Sheaves}, 2 edition, Cambridge Mathematical Library,
  Cambridge University Press (2010)

\bibitem{IP18}
\textbf{E-N Ionel}, \textbf{T\,H Parker},
  \emph{The
  {G}opakumar-{V}afa formula for symplectic manifolds}, Ann. of Math. (2) 187
  (2018) 1--64

\bibitem{K00}
\textbf{F Kato}, \emph{Log
  smooth deformation and moduli of log smooth curves}, Internat. J. Math. 11
  (2000) 215--232

\bibitem{Katz08}
\textbf{S Katz},
  \emph{Genus zero {G}opakumar-{V}afa invariants of contractible curves}, J.
  Differential Geom. 79 (2008) 185--195

\bibitem{KLR}
\textbf{B Kim}, \textbf{H Lho}, \textbf{H Ruddat}, \emph{The degeneration
  formula for stable log maps}
  \href{https://arxiv.org/abs/1803.04210}{arXiv:1803.04210}, 2018

\bibitem{KMPS10}
\textbf{A Klemm}, \textbf{D Maulik}, \textbf{R Pandharipande}, \textbf{E
  Scheidegger}, \emph{Noether-{L}efschetz theory and the {Y}au-{Z}aslow conjecture}, J.
  Amer. Math. Soc. 23 (2010) 1013--1040

\bibitem{Kon95}
\textbf{M Kontsevich},
  \emph{Enumeration of rational curves via torus actions}, from ``The moduli
  space of curves ({T}exel {I}sland, 1994)'', Progr. Math. 129, Birkh\"{a}user
  Boston, Boston, MA (1995)  335--368

\bibitem{Lirel1}
\textbf{J Li},
  \emph{Stable morphisms to singular schemes and relative stable morphisms},
  J. Differential Geom. 57 (2001) 509--578

\bibitem{Lirel2}
\textbf{J Li},
  \emph{A degeneration formula of {GW}-invariants}, J. Differential Geom. 60
  (2002) 199--293

\bibitem{LW15}
\textbf{J Li}, \textbf{B Wu},
  \emph{Good degeneration
  of Quot-schemes and coherent systems}, Comm. Anal. Geom 23 (2015) 841--921

\bibitem{Lin17a}
\textbf{Y-S Lin}, 
  \emph{Open {G}romov-{W}itten invariants on elliptic {K}3 surfaces and
  wall-crossing}, Comm. Math. Phys. 349 (2017) 109--164

\bibitem{Lin17b}
\textbf{Y-S Lin},
  \emph{Correspondence {T}heorem between {H}olomorphic {D}iscs and {T}ropical
  {D}iscs on {K3} {S}urfaces}, J. Differential Geom. 117 (2021) 41--92

\bibitem{Man19}
\textbf{T Mandel}, \emph{Theta bases and log {G}romov--{W}itten invariants of
  cluster varieties} \href{https://arxiv.org/abs/1903.03042}{arXiv:1903.03042},
  2019

\bibitem{ManRu19}
\textbf{T Mandel}, \textbf{H Ruddat}, \emph{Tropical quantum field theory,
  mirror polyvector fields, and multiplicities of tropical curves}
  \href{https://arxiv.org/abs/1902.07183}{arXiv:1902.07183}, 2019

\bibitem{ManRu}
\textbf{T Mandel}, \textbf{H Ruddat},
\emph{Descendant log
  {G}romov--{W}itten invariants for toric varieties and tropical curves},
  Trans. Amer. Math. Soc. 373 (2020) 1109--1152

\bibitem{Ma95}
\textbf{Y\,I Manin},
  \emph{Generating functions in algebraic geometry and sums over trees}, from
  ``The moduli space of curves ({T}exel {I}sland, 1994)'', Progr. Math. 129,
  Birkh\"{a}user Boston, Boston, MA (1995)  401--417

\bibitem{MR20}
\textbf{D Maulik}, \textbf{D Ranganathan}, \emph{Logarithmic Donaldson-Thomas
  theory} \href{https://arxiv.org/abs/2006.06603}{arXiv:2006.06603}, 2020

\bibitem{Mikh05}
\textbf{G Mikhalkin},
  \emph{Enumerative tropical algebraic geometry in {$\Bbb R^2$}}, J. Amer.
  Math. Soc. 18 (2005) 313--377

\bibitem{Muk84}
\textbf{S Mukai}, \emph{Symplectic
  structure of the moduli space of sheaves on an abelian or {$K3$} surface},
  Invent. Math. 77 (1984) 101--116

\bibitem{NR}
\textbf{N Nabijou}, \textbf{D Ranganathan}, \emph{{G}romov-{W}itten theory with
  maximal contacts} \href{https://arxiv.org/abs/1908.04706}{arXiv:1908.04706},
  2019

\bibitem{NiSi}
\textbf{T Nishinou}, \textbf{B Siebert},
  \emph{Toric
  degenerations of toric varieties and tropical curves}, Duke Math. J. 135
  (2006) 1--51

\bibitem{Pan17}
\textbf{R Pandharipande}, \emph{Maps, sheaves and {$K3$} surfaces}, from
  ``Lectures on geometry'', Clay Lect. Notes, Oxford Univ. Press, Oxford (2017)
   159--185

\bibitem{Ran19}
\textbf{D Ranganathan}, \emph{Logarithmic {G}romov-{W}itten theory with
  expansions} \href{https://arxiv.org/abs/1903.09006}{arXiv:1903.09006}, 2019

\bibitem{She12}
\textbf{V Shende}, 
  \emph{Hilbert schemes of points on a locally planar curve and the {S}everi
  strata of its versal deformation}, Compos. Math. 148 (2012) 531--547

\bibitem{Tak96}
\textbf{N Takahashi}, \emph{Curves in the complement of a smooth plane cubic
  whose normalizations are $\mathbb{A}^1$}
  \href{https://arxiv.org/abs/alg-geom/9605007}{arXiv:alg-geom/9605007}, 1996

\bibitem{Tak17}
\textbf{N Takahashi}, \emph{On the multiplicity of reducible relative stable
  morphisms} \href{https://arxiv.org/abs/1711.08173}{arXiv:1711.08173}, 2017

\bibitem{Tak01}
\textbf{N Takahashi}, \emph{Log
  mirror symmetry and local mirror symmetry}, Comm. Math. Phys. 220 (2001)
  293--299

\bibitem{Toda09}
\textbf{Y Toda},
  \emph{Limit stable objects on {C}alabi-{Y}au 3-folds}, Duke Math. J. 149
  (2009) 157--208

\bibitem{Toda10}
\textbf{Y Toda},
  \emph{Curve counting theories via stable objects {I}. {DT}/{PT}
  correspondence}, J. Amer. Math. Soc. 23 (2010) 1119--1157

\bibitem{Toda12}
\textbf{Y Toda}, 
  \emph{Stability conditions and curve counting invariants on {C}alabi-{Y}au
  3-folds}, Kyoto J. Math. 52 (2012) 1--50

\bibitem{Voi96}
\textbf{C Voisin}, \emph{A mathematical proof of a formula of {A}spinwall and {M}orrison},
  Compositio Math. 104 (1996) 135--151

\bibitem{Wise19}
\textbf{J Wise},
  \emph{Uniqueness of minimal morphisms of logarithmic schemes}, Algebr.
  Geom. 6 (2019) 50--63

\bibitem{YZ96}
\textbf{S-T Yau}, \textbf{E Zaslow},
  \emph{B{PS} states,
  string duality, and nodal curves on {$K3$}}, Nuclear Phys. B 471 (1996)
  503--512

\end{thebibliography}

\end{document}